\documentclass[english]{elsarticle}
\usepackage[T1]{fontenc}
\usepackage[latin9]{inputenc}
\usepackage[a4paper]{geometry}
\usepackage{xcolor}
\usepackage{babel}
\usepackage{array}
\usepackage{verbatim}
\usepackage{float}
\usepackage{booktabs}
\usepackage{units}
\usepackage{textcomp}
\usepackage{amsmath}
\usepackage{amssymb}
\usepackage{stmaryrd}
\usepackage{graphicx}
\usepackage{setspace}
\usepackage[unicode=true, bookmarks=true,bookmarksnumbered=false,bookmarksopen=false, breaklinks=true,pdfborder={0 0 0},pdfborderstyle={},backref=false,colorlinks=false]{hyperref}

\makeatletter

\providecommand{\tabularnewline}{\\}
\newcommand{\lyxdot}{.}

\usepackage{babel}
\usepackage{amsthm}
\usepackage{eqnarray}
\usepackage{leftidx}
\usepackage[overload]{empheq}
\newtheorem{theorem}{Theorem}[section]
\newtheorem{remark}[theorem]{Remark}
\newtheorem{proposition}[theorem]{Proposition}
\newtheorem{hyp}[theorem]{Hypothesis}
\newcommand{\norm}[1]{\left\Vert #1\right\Vert}
\usepackage[left,modulo]{lineno}

\allowdisplaybreaks

\makeatother

\begin{document}

\begin{frontmatter}{}

\title{An explicit asymptotic preserving low Froude scheme for the multilayer
shallow water model with density stratification}

\author[imt-insa]{F.~Couderc}

\ead{couderc@math.univ-toulouse.fr}

\author[imt-insa]{A.~Duran}

\ead{aduran@math.univ-toulouse.fr}

\author[imt-insa]{J.-P.~Vila\corref{cor1}}

\ead{vila@insa-toulouse.fr}

\cortext[cor1]{Corresponding author}

\address[imt-insa]{Institut de Mathématiques de Toulouse; UMR5219, Université de Toulouse;
CNRS, INSA, F-31077 Toulouse, France.}
\begin{abstract}
We present an explicit scheme for a two-dimensional multilayer shallow
water model with density stratification, for general meshes and collocated
variables. The proposed strategy is based on a regularized model where
the transport velocity in the advective fluxes is shifted proportionally
to the pressure potential gradient. Using a similar strategy for the
potential forces, we show the stability of the method in the sense
of a discrete dissipation of the mechanical energy, in general multilayer
and non-linear frames. These results are obtained at first-order in
space and time and extended using a simple second-order MUSCL extension.
With the objective of minimizing the diffusive losses in realistic
contexts, sufficient conditions are exhibited on the regularizing
terms to ensure the scheme's linear stability at first and second-order
in time and space. The other main result stands in the consistency
with respect to the asymptotics reached at small and large time scales
in low Froude regimes, which governs large-scale oceanic circulation.
Additionally, robustness and well-balanced results for motionless
steady states are also ensured. These stability properties tend to
provide a very robust and efficient approach, easy to implement and
particularly well suited for large-scale simulations. Some numerical
experiments are proposed to highlight the scheme efficiency: an experiment
of fast gravitational modes, a smooth surface wave propagation, an
initial propagating surface water elevation jump considering a non
trivial topography, and a last experiment of slow Rossby modes simulating
the displacement of a baroclinic vortex subject to the Coriolis force. 
\end{abstract}

\end{frontmatter}{}

\section{Introduction}

The study of geophysical phenomena involves three-dimensional and
turbulent free surface flows with complex geometries. Numerical simulation
of such flows still remains a very demanding challenge, continuously
motivated by environmental, security or economic issues. Since the
past decades, substantial advances have been realized in terms of
mathematical modelling to reduce the original primitive equations
complexity, leading to the emergence of \textit{shallow water} models.
In the particular case of oceans, the density stratification, which
is mainly related to the temperature and salinity variations, can
profoundly affect the water flow dynamics. Taking these aspects under
consideration, the inviscid multilayer shallow water model, which
involves an arbitrary number of superposed immiscible layers, offers
a simple way to integrate the vertical density distribution with a
satisfactory time computation request. The model presented in this
work thus corresponds to a vertical discretization of the primitive
equations, where the flow is described through a superposition of
layers with constant density, as detailed in \citep{Vallis2006},
and shown in Fig.\ref{fig:multilayer-sw-sketch}. One should note
that, thanks to a general formulation of the pressure law, the model
and associated numerical scheme presented in this work has a larger
applicability range, possibly unrelated to large-scale oceanic circulation.
Let us mention for instance the single-layer case, with specific one
or two-dimensional applications to hydraulic or coastal engineering,
or the Euler equations for gas dynamics.

\begin{figure}[!tbh]
\begin{centering}
\includegraphics[height=5cm]{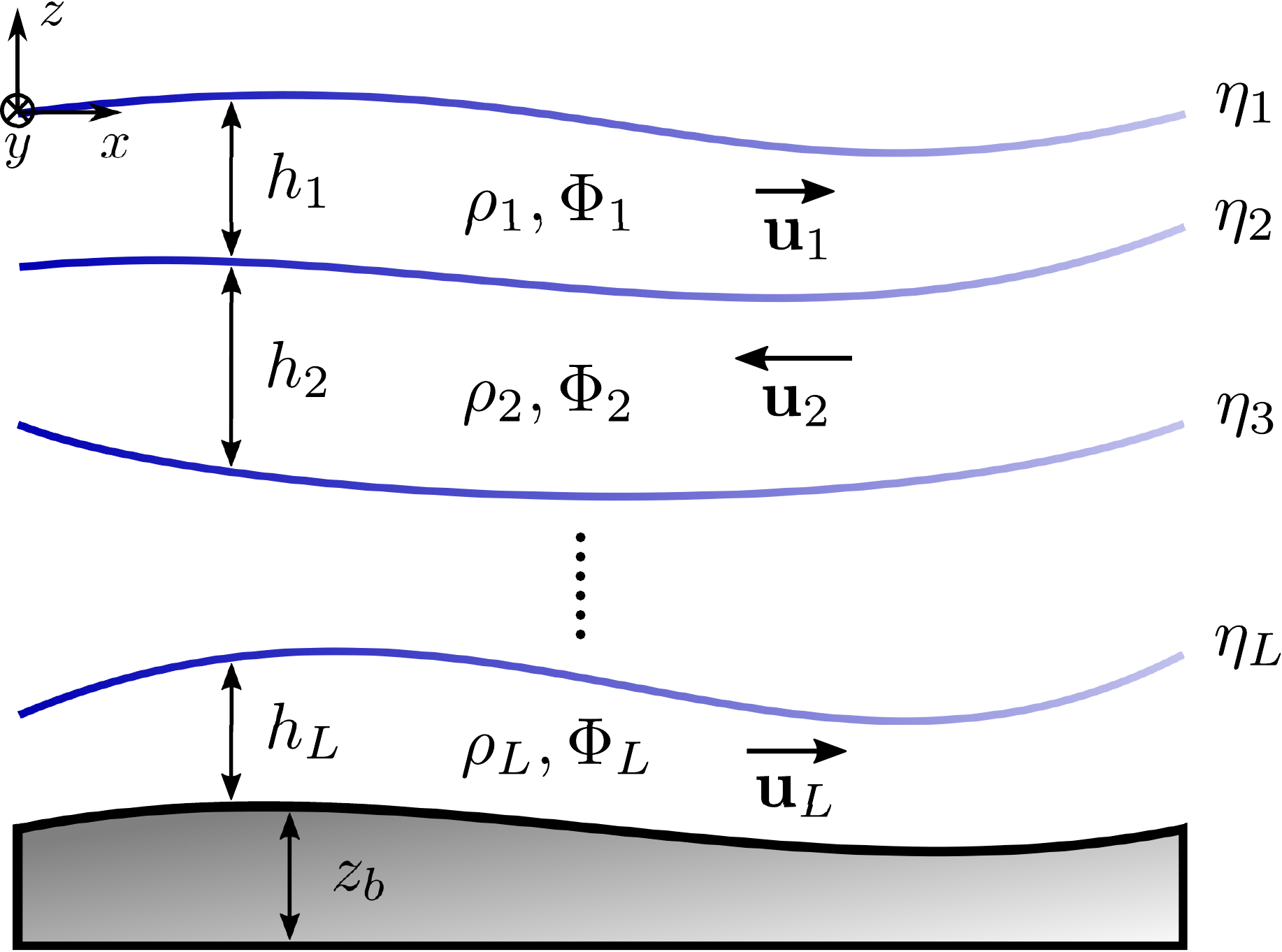} 
\par\end{centering}
\caption{Multilayer shallow water model with density stratification sketch.
$h_{i},u_{i},\rho_{i}$ respectively stand for the water height, horizontal
velocity field and density of the fluid in the \textit{i-th} layer.
$\eta_{i}=z_{b}+\sum_{k=i}^{^{L}}h_{k}$ is the water surface elevation
with respect to the bottom topography $z_{b}$, and the effective
mass in each layer is $H_{i}=h_{i}\rho_{i}$. All the model variables
are collocated along the $z$ coordinate.\label{fig:multilayer-sw-sketch}}
\end{figure}

Naturally, allowing an arbitrary number of layers confers a much more
complex nature to the flow. Indeed, in addition to non-linearities,
it is a known fact that the multilayer equations exhibit particular
structural properties, making the system theoretically and numerically
more demanding. For instance, the hyperbolic structure can be violated
if the shear velocity between two layers is too high, possibly leading
to Kelvin-Helmholtz instabilities. Preventing the complex eigenvalues
appearance is a quite complicated task, and this possible local hyberbolicity
loss can significantly reduce the application range of the numerical
schemes. These stability conditions are rigorously characterized in
\citep{Monjarret_PHD}, where a general criterion of hyperbolicity
and local well-posedness is given, under a particular asymptotic regime
and weak stratification assumptions of the densities and the velocities.
A similar study has been realized in \citep{Duchene2016} in the limit
of small density contrast. It is shown that, under reasonable conditions
on the flow, the system is well-posed on a large time interval. A
second notable difficulty comes from the pressure law, introducing
a non conservative coupling between the layers in the general case. 

As a consequence, if a large range of approaches devoted to the single
layer case are available in the literature, with the handling of complex
geometries and rugged topography using unstructured environments (\citep{Bollermann2013},
\citep{Hou2013}, \citep{Nikolos2009}), robust treatment of friction
forces with wetting and drying (\citep{Burguete2008}, \citep{Cea2012},
\citep{Murillo2012}), and allowing high order resolutions (\citep{Duran2015},
\citep{Xing2013}, \citep{Meister2016}), the quantity of advances
concerning the multilayer system is less plentiful. Nevertheless,
when the number of layers is restricted to two, several techniques
have been proposed on the basis of classical non-linear stability
criteria, generally borrowed from the advances made on the single
layer system. Thus, as concerns the two layers approximations, one
can note for instance the \textit{Q-scheme} proposed in \citep{Castro2001},
the recent relaxation approach \citep{Abgrall2009} able to guarantee
the preservation of motionless steady states, or the so called central-upwind
scheme in \citep{Kurganov2009}. Other splitting and upwind schemes
can be found, with for instance in \citep{Diaz2014} (see also its
extension to three layers proposed in \citep{Chertok2013} with a
study of the hyperbolicity range), the \textit{f-wave} propagation
finite volume method in \citep{Mandli2013} handling dry states or
the \textit{well-balancing} and positivity-preserving results established
in \citep{Berthon2015} within a splitting approach. At last, numerical
methods for one-dimensional multilayer shallow water models with mass
exchange are also proposed without density stratification in \citep{Audusse2014}
and with in \citep{Audusse2011}. The approach is quite different
since the layer depths are not independent variables and only the
free surface is treated, and also because a part of the coupling terms
are treated as a source term. 

That being so, and although a first relevant approximation for ocean
modelling may be provided by a bi-fluid stratification, the number
of layers involved in most of current oceanic flow simulations with
modern operational softwares is much more important in practice, in
the order of several tens. This level of refinement ensures a reasonable
compromise between the needs imposed by an accurate vertical discretization
and computational constraints. Unfortunately, extending the approaches
previously mentioned to the general case is quite difficult to achieve.
One of the reasons is that they are not specially designed to preserve
the asymptotics observed in low Froude number regimes. This requirement
is mandatory for the simulation of oceanic flows, since the velocities
magnitude are very moderate compared to the gravity wave speed far
from the coast. Considering the integration time of realistic simulations,
this limitation is also due to the paramount importance of the mechanical
energy dissipation, which has to be guaranteed in order to produce
physically acceptable solutions.

Adapting the choices made to express the distribution of the pressure
law, which is also generally formulated, in some sense, by mean of
staggered discretizations of the vertical direction, the multilayer
equations formulated in this work are closely connected to those used
in the majority of operational oceanic simulation softwares like HYCOM
\citep{Bleck2002}, ROMS \citep{Shchepetkin2005} or NEMO \citep{Madec2008},
in isopycnal coordinates (\textit{i.e.} when the flow is represented
along the lines of constant density). These softwares have been developed
on staggered grids, sharing an Arakawa C-grid type as a general basis
with orthogonal curvilinear coordinates to take into account irregular
lateral boundaries. This kind of horizontal space discretization prevents
from well known spurious computational modes observed in low Froude
number regimes. The barotropic and baroclinic modes are resolved with
a time splitting technique allowing to use different time steps, as
the barotropic wave speed is much higher than the larger baroclinic
one, and this allows to save time computation. The barotropic continuity
equation is often resolved with a FCT (Flux Corrected Transport) scheme
and the momentum equations discretized with centered schemes of order
two or four. As concerns time integration, Leapfrog-type schemes are
usually employed, coupled with stabilization procedures using a Robert-Asselin
filter in order to minimize the dissipation. A detailed report outlining
the stability aspects related to oceanic modelling is available in
\citep{Lemarie2015}. If these approaches have been largely successfully
applied, they can exhibit some weaknesses for some practical applications.
The global stability of the numerical methods is not always guaranteed,
threatened for instance by the occurrence of vanishing water heights
or the difficulty to handle boundary conditions.

The permanent willingness to improve the quality and the versatility
of numerical resolutions gave rise to an incresing interest for unstructured
geometries during the past decade. The use of such environments may
appear of major interest for many practical applications, and notably
for oceanic circulation, for which geometrical flexibility allows
to describe complex shaped shoreline coastlines and many different
scales. Thus, an increasing number of projects are based on unstructured
meshes, coping with numerical and implementation issues that have
not yet been overcomed on these geometries. In this connection, a
quite complete review of the most recent results oriented toward ocean
modelling can be found in \citep{Danilov2013}. The SLIM \citep{SLIM}
and FVCOM \citep{FVCOM} projects can be cited as examples. Among
the available works, mention can be made of \citep{LeRoux2012} with
the study of Finite Element methods stability applied to the rotating
shallow water equations. It is concluded that all the numerical schemes
considered are, at some point, concerned with spurious solutions.
Some reference works devoted to the derivation of numerical schemes
for the single layer rotating shallow water equations using unstructured
meshes can be cited, as for instance the collocated upwind Finite
Volume approach in \citep{Beljadid2013}, or the works on hexagonal
staggered grids in \citep{Ringler2010} and \citep{Thuburn2009},
dedicated to the geostrophic balance and the modelling of Rossby waves.
These works were recently extended in \citep{Cotter2014} in the context
of higher order discretizations. Note that such stability problems
were recently addressed on regular C-grids in \citep{Stewart2016},
where the issue of mechanical energy conservation is also investigated.
Note also the fully unstructured edge-based method available in \citep{Szmelter2010},
or the staggered scheme \citep{Gassmann2012} devoted to the conservation
of mechanical energy.

The present work describes a numerical strategy devoted to approximate
the solutions of the two-dimensional multilayer shallow water system
with a density stratification. The scheme is formulated in a fully
explicit context and applicable for general meshes. On the basis of
the constraints discussed above, the main objective is the enforcement
of two essential stability results that are the \textit{asymptotic-preserving}
property with respect to low Froude number regimes, and the discrete
dissipation of mechanical energy. The outline of this paper is organized
as follows. In \S 2, we propose a regularization of the model that
allows a better control of the mechanical energy production. We then
give the formulation of the explicit scheme, designed to provide a
discrete equivalent to this formalism, i.e. that allows the decrease
of the mechanical energy. The \S 3 is devoted to stability issues.
Well-balanced and robustness properties are addressed first. We then
show a control on the mechanical energy production, and put it in
correlation with our investigations in the linear case. Asymptotic
preserving properties are established in a semi-continuous context
in \S 4. A last step of numerical validation is finally proposed
to assess the scheme abilities for large-scale simulations. Four test
cases are proposed, implying the study of linear and non-linear solutions,
analysis of convergence rate considering a non trivial topography,
discontinuous solutions, and a last test in a realistic context.

\section{Preliminaries}

\subsection{Physical model}

Denoting $L$ the number of layers involved in the description of
the flow, $t$ and $\mathbf{x}=(x,y)$ the time and space variables,
the dynamics is governed by a general conservation law which consists
of a set of $3\times L$ equations linking the mass in each layer
$H_{i}(t,\mathbf{x})\geq0$ to the horizontal velocity $\mathbf{u}_{i}(t,\mathbf{x})$.
The system is submitted to gravitational forces through the scalar
potential $\Phi_{i}({\boldsymbol{H}},\mathbf{x})$, where ${\boldsymbol{H}}=\,^{t}\left(H_{1},\cdots,H_{L}\right)$:
\begin{equation}
\left\{ \begin{array}{lclcl}
\partial_{t}H_{i} & + & \mathrm{div}\left(H_{i}{\textbf{u}_{i}}\right) & = & 0\\
\partial_{t}(H_{i}{\textbf{u}_{i}}) & + & \mathrm{div}\left(H_{i}{\textbf{u}_{i}}\otimes{\textbf{u}_{i}}\right) & = & -H_{i}\nabla\Phi_{i}/\varepsilon^{2}
\end{array}\,.\right.\label{Model0}
\end{equation}
In the above equations, the parameter $\varepsilon$ is introduced
to account for the scale factor between inertial and potential forces.
This ratio is commonly referred to as \textit{Froude number }or\textit{
Mach number, }depending on the physical context. Similarly, the scalar
potential $\Phi_{i}$ introduced to account for the pressure law may
take different formulations. In the case of the multilayer shallow
water system, and assuming a constant density ${\rho_{i}}$ for each
layer $i$, the effective mass corresponds to $H_{i}={\rho_{i}}{h_{i}}$,
${h_{i}}$, standing for the layer thickness (see Fig.\ref{fig:multilayer-sw-sketch}).
Then, denoting by $z_{b}$ the bottom topography, the scalar potential
is given by (see \citep{Vallis2006}) :

\begin{equation}
{\Phi_{i}}=g\left(z_{b}+{\displaystyle {\sum_{j=1}^{L}}\dfrac{{\rho_{j}}}{\rho_{\max(i,j)}}{h_{j}}}\right)\,.\label{phii}
\end{equation}

From a more general viewpoint, the potential and kinetic energies
attached to the system are defined by $\partial_{H_{i}}\mathcal{E}={\Phi_{i}}$
and ${\mathcal{K}_{i}}=\dfrac{1}{2}H_{i}\norm{\textbf{u}_{i}}^{2}$.
We recall the conservation law satisfied by the mechanical energy
$E=\mathcal{E}/\varepsilon^{2}+{\displaystyle {\sum_{i=1}^{L}}\:{\mathcal{K}_{i}}}$
for regular solutions, corresponding to the second law of thermodynamics:
\begin{equation}
\partial_{t}E+{\sum_{i=1}^{L}}\:\mathrm{div}\Big(\left(H_{i}{\Phi_{i}}/\varepsilon^{2}+{\mathcal{K}_{i}}\right)\,{\textbf{u}_{i}}\Big)=0\,.\label{Energy}
\end{equation}

As concerns numerical resolution of (\ref{Model0}), based on the
constraints discussed above, several guidelines are to be followed,
principally based on two particular stability criteria. The first
one, that has so far not been rigorously addressed in the general
multilayer case, concerns the capability to describe the low Froude
number asymptotics (i.e. when $\varepsilon\ll1$). In these regimes,
and as shown in our numerical experiments, Godunov-type schemes may
bring too much dissipation and do not guarantee a good description
of the flow. It is therefore crucial to work on the basis of rigorous
consistency results. As stated in \citep{Dellacherie2010} in the
context of Euler equations, these asymptotic behaviours are principally
governed by the gradient pressure treatment (corresponding to $\nabla\Phi_{i}$
in (\ref{Model0})), for which centred approaches should be favoured. 

The second essential point is related to the mechanical energy dissipation.
More precisely, this means that the total energy attached to the discrete
system will not increase in time, in accordance with the continuous
frame (\ref{Energy}). This property is crucial for geophysical flows,
since an inappropriate discretization of the system may lead to energy
production and break the stability of the system in large times. Such
considerations of physically admissible solutions are studied in the
numerical approach \citep{Bouchut2010} for the one-dimensional model,
where a semi-discrete entropy inequality is established in addition
to the well-balancing property, treating the non-conservative coupling
part as a source term. A stronger result is obtained in the two layers
case with a fully discrete version \citep{Bouchut2008}. An interesting
approach can be found in \citep{Grenier2013}, in the context of a
compressible multifluid model. Inspired from the ideas of the AUSM
methods for gas dynamics (see \citep{Liou1993} and \citep{Liou2006})
the formalism implies a modified velocity transport, shifted proportionally
to the pressure gradient, whose goal is to provide a control on the
energy budget at the continuous level. On this basis, a simple and
efficient Finite Volume like scheme is derived, designed to provide
a discrete equivalent of this result. More recently, a general extension
has been proposed in \citep{Parisot2015} with the semi-implicit scheme
for the two-dimensional multilayer shallow water model. Note that
in addition, the mentioned approaches have the common feature of being
asymptotic-preserving with respect to low Froude number regimes, notably
thanks to a centred discretization of the pressure gradient, as discussed
above.

To get a better picture of the formalism, we point out that this strategy
can be interpreted at the continuous level as a discrete form of the
following regularized model: 
\begin{equation}
\left\{ \begin{array}{lclcl}
\partial_{t}H_{i} & + & \mathrm{div}\left(H_{i}\left(\textbf{u}_{i}-\delta\textbf{u}_{i}\right)\right) & = & 0\\
\partial_{t}(H_{i}\textbf{u}_{i}) & + & \mathrm{div}\left(H_{i}\textbf{u}_{i}\otimes\left(\textbf{u}_{i}-\delta\textbf{u}_{i}\right)\right) & = & -H_{i}\nabla\Phi_{i}/\varepsilon^{2}
\end{array}\,,\right.\label{Model}
\end{equation}
$\delta{\textbf{u}_{i}}$ standing for a generic perturbation on the
velocity. This modification has the following impact on the energy
conservation (\ref{Energy}): 
\begin{equation}
\partial_{t}E+\sum_{i=1}^{L}\:div\Big(\left(H_{i}\Phi_{i}/\varepsilon^{2}+\mathcal{K}_{i}\right)\,\left(\textbf{u}_{i}-\delta\textbf{u}_{i}\right)\Big)=-\sum_{i=1}^{L}\:\delta\textbf{u}_{i}.\nabla\Phi_{i}/\varepsilon^{2}\,,\label{Energy2}
\end{equation}
which formally justifies a calibration of $\delta{\textbf{u}_{i}}$
in terms of the pressure gradient, to ensure a global decrease of
the mechanical energy. 

Following these lines, we aim at proposing a discrete equivalent of
(\ref{Energy2}), in a fully explicit context. In this environment,
the use of a shifted velocity transport $\left({\textbf{u}_{i}}-\delta{\textbf{u}_{i}}\right)$
is not sufficient to ensure a mechanical energy control, and a correction
term is also needed on the scalar potential $\Phi_{i}$. It may also
be shown that this adjustment, expressed in terms of discharge divergence,
has also regularizing virtues on the energy budget at the continuous
level. The practical advantages of an explicit formulation in comparison
with the semi-implicit formulation proposed in \citep{Parisot2015}
stand in the exemption of resolving the nonlinear system arising from
for the continuity equation, an easier implementation of boundary
conditions, and high order extensions in space and time can be more
relatively easily derived. If the time step can be more restrictive,
it is far from being obvious to compare the relative performances
of the explicit and semi-implicit approaches in terms of accuracy
\textit{vs.} computation time. The particular difficulty to derive
high order time stepping schemes for semi-implicit strategies without
loosing strong stability properties makes things worse. At last, mixed
formulations can also be derived decoupling the time advancement of
the fast barotropic mode with the semi-implicit scheme, and the slow
baroclinic modes using the explicit scheme, like it is already done
in oceanic simulation softwares. Such a numerical model, that couples
the benefits of the two approaches, is currently under study.

As mentioned before, the equations (\ref{Model0}) enjoys a large
range of applicability, so that the present approach is not only limited
to large-scale oceanic circulation. Generally, we need the following
regularity hypothesis on the potential forces:

\begin{hyp}\textit{Regularity assumptions on the potential forces}
\label{Regularity} 
\begin{itemize}
\item The potential $\mathcal{E}$ is a regular and convex function of the
mass, which means that the Hessian ${\boldsymbol{\mathcal{H}}}$ given
by (see \citep{Vallis2006}):
\begin{equation}
{\boldsymbol{\mathcal{H}}_{ij}}=\partial_{H_{i}{H_{j}}}^{2}\mathcal{E}=\partial_{{H_{j}}}{\Phi_{i}}\:,\:(i,j)\in\llbracket1,\ldots,L\rrbracket^{2}\:,\:\label{Hessian}
\end{equation}
is positive-definite. 
\item The potential is a symmetric and linear function of the mass, that
is $\Phi={\boldsymbol{\mathcal{H}}}.\boldsymbol{H}$ and ${\boldsymbol{\mathcal{H}}}$
symmetric. 
\item The $L^{2}$ norm of ${\boldsymbol{\mathcal{H}}}$ is uniformly bounded
with respect to space and time, more precisely: 
\begin{equation}
\vert\vert\vert{\boldsymbol{\mathcal{H}}}(\boldsymbol{H},\mathbf{x})\vert\vert\vert_{L^{2}}\,\leq\,{C_{{\boldsymbol{\mathcal{H}}}}}\,.\label{HL2}
\end{equation}
\end{itemize}
\end{hyp}

\begin{remark} \label{Hess}

In the case where the scalar potential is given by (\ref{phii}),
the Hessian ${\boldsymbol{\mathcal{H}}}({\boldsymbol{H}},\mathbf{x})$
is constant in space and time: 
\begin{equation}
{\boldsymbol{\mathcal{H}}}_{i,j}=g\:\rho_{j}/\rho_{max(i,j)}\,,\label{Hij}
\end{equation}
and the requirements listed in Hypothesis \ref{Regularity} are trivially
satisfied. The $L^{2}$ norm of $\mathcal{H}$ is thus evaluated in
a pre-processing step, and we simply take ${C_{{\boldsymbol{\mathcal{H}}}}}\,=\vert\vert\vert{\boldsymbol{\mathcal{H}}}(\boldsymbol{H},\mathbf{x})\vert\vert\vert_{L^{2}}.$
Note also that this formulation automatically brings the conservation
of the total momentum, as shown in \citep{Parisot2015}. However,
this is not sufficient to guarantee the well-posedness of the problem:
some conditions can be found in \citep{Monjarret_PHD}, regarding
${\boldsymbol{\mathcal{H}}}$ as a natural symmetrizer of the system.
These conditions are based on smallness assumptions on the shear velocity
and are sufficient to ensure that the system is hyperbolic. These
low-shear conditions, easy to check numerically, were always widely
satisfied in our operational situations. Hence, these aspects will
not be discussed further in this work, and we refer to the references
above for details.

\end{remark}

\subsection{Notations}

We consider in this work a tesselation $\mathbb{T}$ of the computational
domain $\Omega\subset\mathbb{R}^{2}$. We will denote ${m_{K}}$ the
area and ${m_{\partial K}}$ the perimeter of a cell $K\in\mathbb{T}$.
The boundary of $K$ will be denoted $\partial K$, and for any edge
$e\in\partial K$, ${m_{e}}$ the length of the corresponding boundary
interface and $\textbf{n}_{e,K}$ the outward normal to $e$ pointing
to the neighbour $K_{e}$ (see Fig.\ref{Geom1}).

\begin{figure}[!tbh]
\centering{}\includegraphics[width=0.3\textwidth]{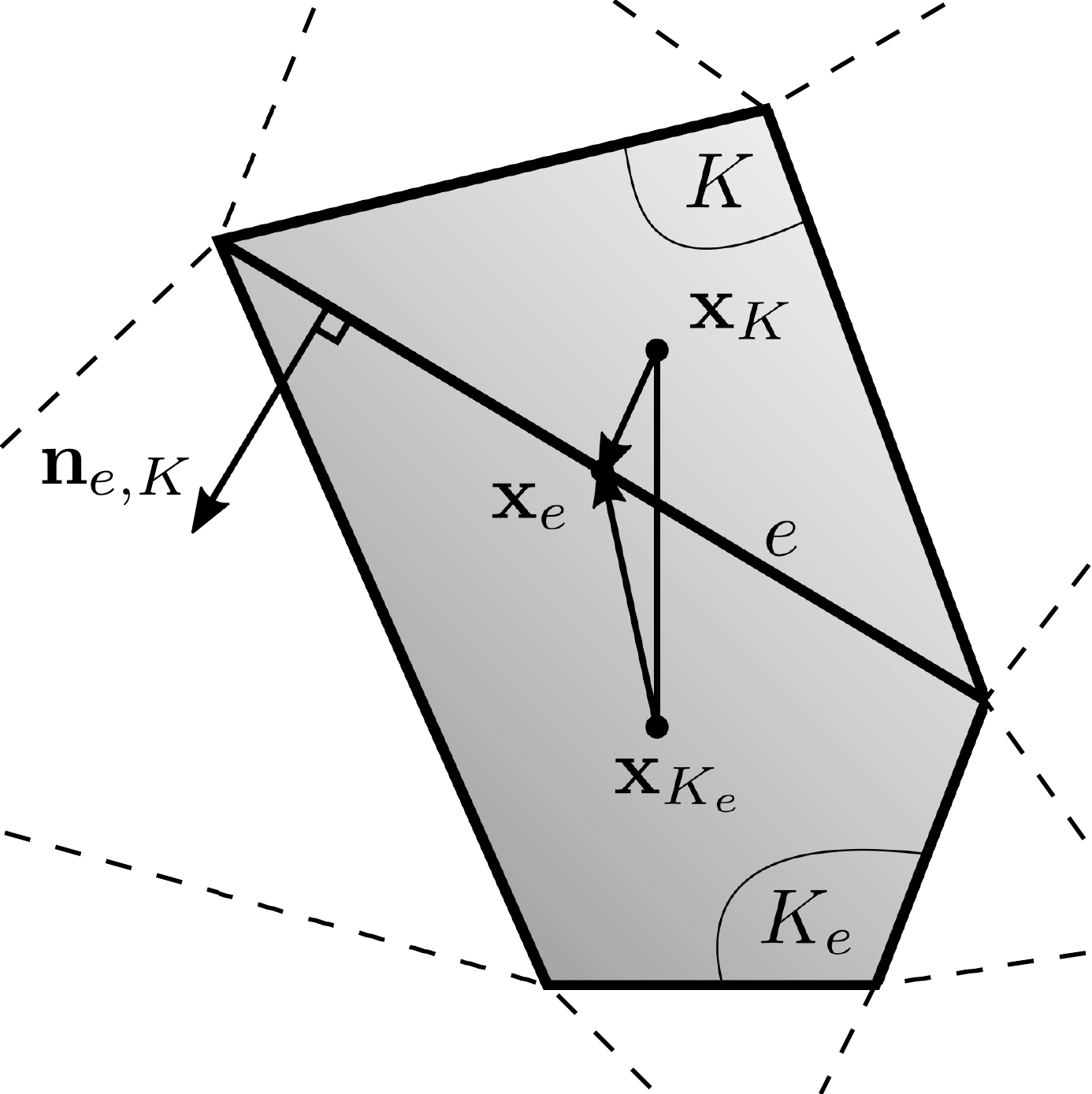}\caption{Geometric settings - Focus on the edge $e\in\partial K\cap\partial K_{e}$;
$\textbf{n}_{e,K}$ is the outward normal to $e$, pointing to $K_{e}$,
$\mathbf{x}_{K}$ indicates the mass center of $K$ and $\mathbf{x}_{e}$
is the middle of $e$. \label{Geom1}}
\end{figure}

Let's now introduce some useful notations. For a scalar piecewise
constant function $w$ we define: 
\[
\overline{w}_{e}=\dfrac{1}{2}\left(w_{K_{e}}+w_{K}\right)\quad,\quad\boldsymbol{\delta}w_{e}=\dfrac{1}{2}\left(w_{K_{e}}-w_{K}\right)\textbf{n}_{e,K}\,,
\]
and similary, for a piecewise constant vectorial function $\textbf{w}$:
\[
\overline{\textbf{w}}_{e}=\dfrac{1}{2}\left(\textbf{w}_{K_{e}}+\textbf{w}_{K}\right)\quad,\quad{\delta}\textbf{w}_{e}=\dfrac{1}{2}\left(\textbf{w}_{K_{e}}-\textbf{w}_{K}\right).\textbf{n}_{e,K}\,.
\]
We also set: $w^{\pm}=\dfrac{1}{2}(w\pm\vert w\vert)$ the positive
and negative parts of a scalar function $w$.

\subsection{Numerical approach}

The numerical scheme we consider is the following:%

\begin{subequations}
\begin{empheq}[left=\empheqlbrace]{alignat=3}
& H_{K,i}^{n+1} & \;=\; & H_{K,i}^{n} & \;-\; & {\displaystyle \dfrac{\Delta t}{m_{K}}\sum_{e\in\partial K}\left(\mathcal{F}_{e,i}^{n}.\mathbf{n}_{e,K}\right)m_{e}} \label{mass} \\
 & H_{K,i}^{n+1}\mathbf{u}_{K,i}^{n+1} & \;=\; & H_{K,i}^{n}\mathbf{u}_{K,i}^{n} & \;-\; & \dfrac{\Delta t}{m_{K}}\sum_{e\in\partial K}\left(\mathbf{u}_{K,i}^{n}\left(\mathcal{F}_{e,i}^{n}.\mathbf{n}_{e,K}\right)^{+}+\mathbf{u}_{K_{e},i}^{n}\left(\mathcal{F}_{e,i}^{n}.\mathbf{n}_{e,K_{e}}\right)^{-}\right)m_{e} \label{mom} \\ 
 &  &  &  & \;-\; & \dfrac{\Delta t}{m_{K}}H_{K,i}^{n}\sum_{e\in\partial K}\left(\dfrac{\Phi_{e,i}^{n,\ast}}{\varepsilon^{2}}\mathbf{n}_{e,K}\right)m_{e} \nonumber
\end{empheq}
\end{subequations}where we have set:

\begin{subequations}

\begin{alignat}{3}
 & \mathcal{F}_{e,i}^{n} & \;=\; & \overline{H\textbf{u}}_{e,i}^{n}-\Pi_{e,i}^{n} & \;=\; & \left(\dfrac{H_{K,i}^{n}\textbf{u}_{K,i}^{n}+H_{K_{e},i}^{n}\textbf{u}_{K_{e},i}^{n}}{2}\right)-\Pi_{e,i}^{n}\,,\label{fe}\\
 & \Phi_{e,i}^{n,\ast} & \;=\; & \overline{\Phi}_{e,i}^{n}-\Lambda_{e,i}^{n} & \;=\; & \left(\dfrac{\Phi_{K,i}^{n}+\Phi_{K_{e},i}^{n}}{2}\right)-\Lambda_{e,i}^{n}\,.\label{phiea}
\end{alignat}

\end{subequations}

The quantities ${\Lambda_{e,i}^{n}}$ and $\Pi_{e,i}^{n}$ introduced
above stand for the perturbations respectively assigned to the potential
forces and numerical fluxes, designed to ensure the stability of the
method. They are defined as follows: 
\begin{align}
\Pi_{e,i}^{n} & =\gamma\:\Delta t\:\left(\frac{\widehat{H}}{\text{\ensuremath{\Delta}}}\right)_{e,i}^{n}\:\dfrac{\boldsymbol{\delta}\Phi_{e,i}^{n}}{\varepsilon^{2}}\quad,\quad\gamma\geq0\,,\label{pien}\\
\Lambda_{e,i}^{n} & ={\normalcolor {\normalcolor {\color{teal}{\normalcolor \alpha\:\Delta t\:\left(\frac{C_{\boldsymbol{\mathcal{H}}}}{\Delta_{e}}\right)\:\delta(H\textbf{u})_{e,i}^{n}}}}}\quad,\quad\alpha\geq0\,,\label{alpen}
\end{align}
with the geometric constant: 
\begin{equation}
\frac{1}{\Delta_{e}}=\dfrac{1}{2}\left(\frac{1}{\Delta_{K}}+\frac{1}{\Delta_{K_{e}}}\right)=\dfrac{1}{2}\left(\dfrac{m_{\partial K}}{m_{K}}+\dfrac{m_{\partial K_{e}}}{m_{K_{e}}}\right)\,,\label{eq:d}
\end{equation}
where $d$ is the problem dimension, and the weighted average: 

\begin{equation}
\left(\frac{\widehat{H}}{\text{\ensuremath{\Delta}}}\right)_{e,i}^{n}=\dfrac{1}{2}\left(\left(\frac{\widehat{H}}{\text{\ensuremath{\Delta}}}\right)_{K,i}^{n}+\left(\frac{\widehat{H}}{\text{\ensuremath{\Delta}}}\right)_{K_{e},i}^{n}\right)=\dfrac{1}{2}\left(\widehat{H}_{K}^{n}\dfrac{m_{\partial K}}{2m_{K}}+\widehat{H}_{K_{e}}^{n}\dfrac{m_{\partial K_{e}}}{2m_{K_{e}}}\right)\,,\label{hmu}
\end{equation}
where $\widehat{H}_{K}^{n}$ is $H_{K}^{n}$ in practice and for all
the presented simulations. Nervertheless, for simplification purposes,
the non-linear stability analysis developed in this work implies an
implicit definition of $\widehat{H}_{K}^{n}$, according to (\ref{hath})
(see Remark \ref{Implicit} below). We also refer to Theorem \ref{dissipation}
and subsequent Remark \ref{Rq1} for the calibration of the stabilization
constants $\alpha$ and $\gamma$.

As the increase of space and time order will be discussed throughout
the paper, we give the second-order extensions in space and time in
the Appendix \ref{subsec:Second-order-extension}. A MUSCL spatial
reconstruction scheme (\ref{subsec:second-order-scheme}) is used
(substituting at each side of the edge $e$ the primitive variables
$H_{K}$, $\mathbf{u}_{K}$, $H_{K_{e}}$ and $\mathbf{u}_{K_{e}}$
by reconstructed primitive variables $H_{e,K}$, $\mathbf{u}_{e,K}$,
$H_{e,K_{e}}$ and $\mathbf{u}_{e,K_{e}}$ to evalutate the numerical
fluxes in (\ref{mass}) and (\ref{mom})). The temporal discretization
is achieved using the Heun's method (\ref{subsec:time-stepping-coriolis}).

\begin{remark}\label{def_pi}

$\Pi_{e,i}^{n}$ is related to the potential pressure gradient $\boldsymbol{\delta}\Phi_{e,i}^{n}$
and is intended to reproduce the stabilizing effects of the generic
perturbation $\delta\textbf{u}_{i}$ introduced in the continuous
frame in (\ref{Model}) to regularize the energy budget (\ref{Energy2}).
Similarly, it can be shown that the continuous equivalent of $\Lambda_{e,i}^{n}$,
which involves an approximation of the discharge divergence, brings
an additional dissipation term in (\ref{Energy2}).\end{remark}

\begin{remark}\label{Implicit}

$\widehat{H}_{K,i}^{n}$ is indeed explicit in practice. If the mechanical
energy dissipation will be demonstrated here with an implicit definition
of $\widehat{H}_{K,i}^{n}$ according to (\ref{hath}), taking $\widehat{H}_{K}^{n}=H_{K,i}^{n}$
introduces an error in $\mathcal{O}(\Delta t)$ and is widely sufficient
to preserve the overall stability. These conclusions have been reached
with the support of many numerical experiments (including the propagation
of discontinuous initial solutions), with a particular focus on low
Froude number regimes ($\varepsilon\ll1$), for which we observed
no significant impact. Moreover, the proof of mechanical energy dissipation
proposed in Appendix \ref{Proofs} can be realized in a fully explicit
way, at the price of a more complex analysis and a slight adaptation
of $\Pi_{e,i}^{n}$, leading to very similar results. For readability
reasons we chose not to detail this proof and some insights are available
in Remark \ref{Explicit}. The implicit definition of $\widehat{H}_{K,i}^{n}$
in the non-linear stability proof is considerably lighter and provides
a good overview of the employed strategy.\end{remark}

Let us finally remark that the numerical scheme satisfied by the velocity
is: 
\begin{equation}
\textbf{u}_{K,i}^{n+1}=\textbf{u}_{K,i}^{n}-\dfrac{\Delta t}{m_{K}}{\sum_{e\in\partial K}}\dfrac{{\textbf{u}_{K_{e},i}^{n}}-\textbf{u}_{K,i}^{n}}{{H_{K,i}^{n+1}}}{\big(\mathcal{F}_{e,i}^{n}.\textbf{n}_{e,K}\big)^{-}}{m_{e}}-\dfrac{\Delta t}{m_{K}}\dfrac{{H_{K,i}^{n}}}{{H_{K,i}^{n+1}}}{\sum_{e\in\partial K}}\dfrac{{\Phi_{e,i}^{n,\ast}}}{\varepsilon^{2}}\textbf{n}_{e,K}{m_{e}}\,,\label{u1}
\end{equation}
and note that:
\begin{equation}
{\sum_{e\in\partial K}}{\Phi_{e,i}^{n,\ast}}\textbf{n}_{e,K}{m_{e}}={\sum_{e\in\partial K}}{\boldsymbol{\delta}\Phi_{e,i}^{n}}{m_{e}}-{\sum_{e\in\partial K}}{\Lambda_{e,i}^{n}}\textbf{n}_{e,K}{m_{e}}\,,\label{RelPhi}
\end{equation}
since the main term of (\ref{phiea}) involves a centred discretization
of the potential.

We finally recall the explicit CFL condition on which are usually
based Godunov-type schemes (see \citep{Godlewski1996}):

\begin{equation}
\left(\vert\overline{\textbf{u}}_{e,i}^{n}.\textbf{n}_{e,K}\vert+\dfrac{c_{e,i}^{n}}{\varepsilon}\right)\Delta t\max\left(\dfrac{{m_{e}}}{{m_{K}}},\dfrac{{m_{e}}}{{m_{K_{e}}}}\,\right)\leq\tau_{CFL}\,,\label{CFL_1}
\end{equation}
where $\left(c_{e,i}^{n}\right)_{1\leq i\leq L}$ corresponds to the
square root of the eigenvalues of the matrix $H_{i}{\boldsymbol{\mathcal{H}}}_{i,j}$.

\section{Stability issues}

In this section we focus on crucial linear and non-linear stability
criterion: motionless steady states preservation, water height positivity
preservation, mechanical energy dissipation and linear stabily analysis.
These essential points need to be integrated in the construction of
numerical schemes expected to respond to practical issues. Traditionally,
providing a numerical approach able to account for all these aspects
simultaneously remains a quite complicated task, especially in the
context of general geometries and stratified multiscale models. Nevertheless,
the formalism employed here allows a quite simple treatment of well-balanced
and robustness properties. We finally provide the complete linear
stability analysis in order to derive relaxed stability conditions
comparatively to ones ensuring the mechanical energy dissipation which
are not optimal. Indeed, it will be shown that the linear stability
conditions are far less restrictive.

\subsection{Well Balancing\label{well-balancing}}

As a first stability criterion we study the problem of steady states
preservation. From a general point of view, regarding the difficulty
to derive and handle numerically the full set of steady states observed
in most of realistic evolution processes, it is classical to focus
first on rest states. In our formalism, this leads to the following
trivial solution:
\[
\textbf{u}_{K,i}=0\quad,\quad\Phi_{K,i}=\Phi_{i}\,,
\]
for all volume control $K$ and layer $i$. This trivial solution
is nothing but the generalization to the multilayer case of the classical
\textit{lake at rest} solution in the $L=1$ case: 
\[
\textbf{u}=0\quad,\quad h+z_{b}=0\,,
\]
which has indeed to be exactly preserved to avoid the appearance of
non physical perturbations in the vicinity of flat free surface configurations.
The capability to preserve these particular steady states already
stands for a discriminating property, even in the one layer case,
notably with the increasing interest of unstructured meshes and high
order space schemes. In spite of these difficulties, the proposed
discretization allows their exact preservation without needing any
correction at first-order in space and in a very simple way at second-order.
The following approach is thus intrinsically adapted to the preservation
of such equilibria, standing for a good alternative to the classical
well-balanced methods.

\begin{proposition}\textit{Well Balancing} \\
 The scheme (\ref{mass},\ref{mom}) equipped with the numerical fluxes
(\ref{fe}) and discrete potential (\ref{phiea}) preserves the steady
states at rest defined by $\textbf{u}_{K,i}^{n}=0$ and ${\Phi_{K,i}^{n}}=\Phi_{i}$.

\end{proposition}

\begin{proof} Since the perturbation $\Pi_{e,i}^{n}$ (\ref{alpen})
is expressed in terms of ${\boldsymbol{\delta}\Phi_{e,i}^{n}}$, we
immediately have $\mathcal{F}_{e,i}^{n}=0$ and (\ref{mass}) gives
${H_{K,i}^{n+1}}={H_{K,i}^{n}}$. Then, since ${\Lambda_{e,i}^{n}}=0$,
the momentum equation (\ref{mom}) reduces to: 
\begin{equation}
{H_{K,i}^{n+1}}\textbf{u}_{K,i}^{n+1}=-\dfrac{\Delta t}{m_{K}}{H_{K,i}^{n}}\left(\dfrac{\Phi_{i}}{\varepsilon^{2}}\right)\left({\sum_{e\in\partial K}}\textbf{n}_{e,K}{m_{e}}\right)=0\,,
\end{equation}
which allows to conclude.\end{proof}

The second-order MUSCL spatial reconstruction requires to evaluate
a vectorial slope in each volume control K for all primitive variables
(one can also compute the vectorial slopes from conservative or entropic
variables to reconstruct at the end the primitive variables at the
edge). The resulting scheme (\ref{eq:exp-scheme-1-1}-\ref{eq:exp-scheme-2-1})
produces a non well-balanced scheme in most of practical cases. This
is because the water surface elevation ${\normalcolor \eta_{i}={\color{teal}{\normalcolor z_{b}+}}\sum_{k=i}^{^{L}}h_{k}}$
must be locally linear to produce for each edge $e$ two equal reconstructed
water surface elevation $\eta_{K,i}$ and $\eta_{K_{e},i}$. As a
consequence, the MUSCL spatial reconstruction breaks the well-balanced
property demonstrated previously. One simple way to resolve this drawback
is to evalute the vectorial slope for the water surface elevation
$\eta_{i}$ rather than for the water height $h_{i}$. Considering
an arbiratrary bed elevation $z_{b_{e}}$ at the edge $e$ (that can
be directly evaluated from a continuous function or taking the half
sum from the two adjacent volume control $K$ and $K_{e}$), the two
water heights are finally evaluated substracting the edge bed elevation
$z_{b_{e}}$ to the two reconstructed water surface elevation $\eta_{K,i}$
and $\eta_{K_{e},i}$.

\subsection{Robustness\label{Robustness}}

We investigate here the problem of robustness by proposing a CFL condition
allowing to obtain the preservation of the water height positivity.

\begin{proposition}\textit{Robustness}\label{Rob}\\
We consider the numerical scheme (\ref{mass},\ref{mom}) equipped
with the numerical fluxes (\ref{fe}) and discrete potential (\ref{phiea}).
Assume a CFL condition of the type: 
\begin{equation}
\Delta t\max\left(\dfrac{{m_{\partial K}}}{{m_{K}}},\dfrac{{m_{\partial K_{e}}}}{{m_{K_{e}}}}\right)\left(\vert\overline{\textbf{u}}_{e,i}^{n}.\textbf{n}_{e,K}\vert+\sqrt{\gamma}\sqrt{\dfrac{\vert{\boldsymbol{\delta}\Phi_{e,i}^{n}}\vert}{\varepsilon^{2}}}\right)\leq\left(\dfrac{\beta}{\beta+1}\right)\xi_{e,i}^{n}\label{CFL2}
\end{equation}
for each edge $e=\partial K\cap\partial K_{e}$ , where $0<\beta\leq1$
and:

\begin{equation}
\xi_{e,i}^{n}=\dfrac{\min\left({H_{K,i}^{n}},{H_{K_{e},i}^{n}}\right)}{\max\left({\widehat{H}_{K,i}^{n}},{\widehat{H}_{K_{e},i}^{n}},{H_{K,i}^{n}},{H_{K_{e},i}^{n}}\right)}\,.\label{xien}
\end{equation}
Then:

\begin{equation}
{H_{K,i}^{n+1}}\geq\dfrac{1}{\beta}\dfrac{\Delta t}{m_{K}}{\sum_{e\in\partial K}}-{\big(\mathcal{F}_{e,i}^{n}.\textbf{n}_{e,K}\big)^{-}}{m_{e}}\geq0\,.\label{Robustesse}
\end{equation}
\end{proposition}

\begin{proof} The result being specific to each layer, we drop the
subscript \textit{``i''} for the sake of clarity. Gathering
\[
\dfrac{\Delta t}{m_{K}}{\sum_{e\in\partial K}}-{\big(\mathcal{F}_{e}^{n}.\textbf{n}_{e,K}\big)^{-}}{m_{e}}\leq\dfrac{\Delta t}{m_{K}}{\sum_{e\in\partial K}}\vert\mathcal{F}_{e}^{n}.\textbf{n}_{e,K}\vert{m_{e}}\,,
\]
and 
\[
{H_{K}^{n+1}}\geq{H_{K}^{n}}-\dfrac{\Delta t}{m_{K}}{\sum_{e\in\partial K}}\vert\mathcal{F}_{e}^{n}.\textbf{n}_{e,K}\vert{m_{e}}\,,
\]
we get: 
\[
\begin{split}\beta{H_{K}^{n+1}}-\dfrac{\Delta t}{m_{K}}{\sum_{e\in\partial K}}-{\big(\mathcal{F}_{e}^{n}.\textbf{n}_{e,K}\big)^{-}}{m_{e}}\geq\beta{H_{K}^{n}} & -(1+\beta)\dfrac{\Delta t}{m_{K}}{\sum_{e\in\partial K}}\vert{\overline{H\textbf{u}}_{e}^{n}}.\textbf{n}_{e,K}\vert{m_{e}}\\
 & -(1+\beta)\dfrac{\Delta t}{m_{K}}\gamma\Delta t{\sum_{e\in\partial K}}\left(\frac{\widehat{H}}{\text{\ensuremath{\Delta}}}\right)_{e}^{n}\dfrac{\vert{\boldsymbol{\delta}\Phi_{e}^{n}}.\textbf{n}_{e,K}\vert}{\varepsilon^{2}}{m_{e}}\,.
\end{split}
\]
From this, a sufficient condition to obtain (\ref{Robustesse}) can
be expressed locally as: 
\[
\begin{split} & (1+\beta)\dfrac{\Delta t}{m_{K}}\vert{\overline{H\textbf{u}}_{e}^{n}}.\textbf{n}_{e,K}\vert+(1+\beta)\gamma\Delta t\dfrac{\Delta t}{m_{K}}\left(\frac{\widehat{H}}{\text{\ensuremath{\Delta}}}\right)_{e}^{n}\dfrac{\vert{\boldsymbol{\delta}\Phi_{e}^{n}}.\textbf{n}_{e,K}\vert}{\varepsilon^{2}}\,\leq\beta\dfrac{{H_{K}^{n}}}{{m_{\partial K}}}\,,\end{split}
\]
This leads to: 
\begin{equation}
\mu\vert\overline{\textbf{u}}_{e}^{n}.\textbf{n}_{e,K}\vert+\mu^{2}\gamma\dfrac{\vert{\boldsymbol{\delta}\Phi_{e}^{n}}.\textbf{n}_{e,K}\vert}{\varepsilon^{2}}\leq\left(\dfrac{\beta}{1+\beta}\right)\xi_{e}^{n}\,,
\label{CFL2b}
\end{equation}
where $\mu=\Delta t\max\left(\dfrac{{m_{\partial K}}}{{m_{K}}},\dfrac{{m_{\partial K_{e}}}}{{m_{K_{e}}}}\right)$.
Since the right member of the previous inequality is lower than $1$,
we conclude that (\ref{CFL2b}) is ensured under (\ref{CFL2}).

\end{proof}

\begin{remark}\label{Rob2} $\boldsymbol{\delta}\Phi_{e,i}^{n}$
being in the order of the mesh size, the advective terms govern the
CFL condition (\ref{CFL2}), which is thus far less restrictive than
a time step restriction of the form (\ref{CFL_1}) in the case of
practical applications implying low Froude numbers. Note also that
in these contexts the water heights are far from zero, preventing
the quantity $\xi_{e,i}^{n}$ (\ref{xien}) from being arbitrarily
small. In practice, $\xi_{e,i}^{n}$ reduces to $\dfrac{\min\left({H_{K,i}^{n}},{H_{K_{e},i}^{n}}\right)}{\max\left({H_{K,i}^{n}},{H_{K_{e},i}^{n}}\right)}$
(see Remark \ref{Implicit}) and is very nearly $1$. In more general
terms, solutions are proposed in \citep{Bollermann2013},\citep{Bollermann2011}
to deal with wet/dry fronts when considering CFL conditions of the
form (\ref{CFL2}). From now, taking these aspects under consideration,
we assume that for all $\beta>0$ the positivity result (\ref{Robustesse})
holds under the CFL constraint (\ref{CFL_1}). In other terms:
\begin{equation}
\dfrac{\Delta t}{m_{K}}{\sum_{e\in\partial K}}-{\big(\mathcal{F}_{e,i}^{n}.\textbf{n}_{e,K}\big)^{-}}{m_{e}}\leq\beta{H_{K,i}^{n+1}}\,.\label{Robustesse2}
\end{equation}
In our stability results, we need $\beta=1/4$ (see (\ref{SK}) and
below){\small{}. }We numerically verified that this time step restriction
was always less restrictive than the classical explicit CFL condition
given in (\ref{CFL_1}), based on the gravity wave speed.{\small{}
}Note, however, that $\beta$ can be taken smaller to obtain relaxed
conditions on the stabilization constants $\alpha$ and $\gamma$
(see Remark \ref{Rq1}).\end{remark}

\subsection{Energy dissipation\label{dissip_energy}}

The main result of the current section concerns the dissipation of
the mechanical energy at the discrete level. Denoting ${E^{n}}={\displaystyle \sum_{K\in\mathbb{T}}{m_{K}}\left({\mathcal{E}_{K}^{n}}/\varepsilon^{2}+{\displaystyle {\sum_{i=1}^{L}}{\mathcal{K}_{K,i}^{n}}}\right)}$
the discrete energy at time $n$, we have the following result: \begin{theorem}\label{dissipation}
\textit{Control of the mechanical energy}

We consider the numerical scheme (\ref{mass},\ref{mom}), together
with the corrected potential (\ref{phiea},\ref{alpen}):
\[
{\Phi_{e,i}^{n,\ast}}={\overline{\Phi}_{e,i}^{n}}-{\Lambda_{e,i}^{n}}\quad,\quad{\Lambda_{e,i}^{n}}=\alpha\Delta t{C_{{\boldsymbol{\mathcal{H}}}}}\frac{{{\delta}(H\textbf{u})_{e,i}^{n}}}{\Delta_{e}}\,\,,
\]
and numerical fluxes (\ref{fe},\ref{pien}):
\[
\mathcal{F}_{e,i}^{n}={\overline{H\textbf{u}}_{e,i}^{n}}-\Pi_{e,i}^{n}\quad,\quad\Pi_{e,i}^{n}=\gamma\Delta t\left(\frac{\widehat{H}}{\text{\ensuremath{\Delta}}}\right)_{e,i}^{n}\dfrac{{\boldsymbol{\delta}\Phi_{e,i}^{n}}}{\varepsilon^{2}}\,\,,
\]
Assume that the time step is governed by an explicit CFL condition
(\ref{CFL_1}). Then, the stabilization constants 
\[
\gamma=4\;,\;\alpha=2
\]
ensure the control of the mechanical energy production: 
\begin{equation}
{E^{n+1}}-{E^{n}}\leq\,0\,.
\label{ControlE}
\end{equation}
\end{theorem} To establish the announced result, we first give an
estimate for the kinetic and potential energy productions, and finally
show that the choice $\gamma=4$ and $\alpha=2$ in (\ref{alpen})
and (\ref{pien}) allows a global control of these contributions.
The proof is given in Appendix \ref{Proofs} and organized around
the following steps: 
\begin{itemize}
\item $\#A$ - Estimation of the kinetic energy production (Appendix \ref{App-propK},
Proposition \ref{propK}). 
\item $\#B$ - Estimation of the potential energy production (Appendix \ref{App-propE},
Proposition \ref{propE}). 
\item $\#C$ - Control of the mechanical energy (Appendix\textbf{ }\ref{TotalE}):
we gather the two inequalities resulting from $\#A$ and $\#B$ to
deduce a sufficient condition on the stabilization constants $\gamma$
and $\alpha$ present in the correction terms (\ref{alpen}, \ref{pien})
. 
\end{itemize}
\begin{remark}\label{Rq0} The proof is mainly based on the negativity
of the quadratic polynomials given in (\ref{p}) and (\ref{q}), which
dominant coefficient is expressed in terms of the following quantity:
\[
\rho_{\varepsilon}^{2}=2\dfrac{\left(\Delta t\right)^{2}}{\varepsilon^{2}}\frac{C_{\boldsymbol{\mathcal{H}}}}{\Delta_{e}}\left(\frac{\widehat{H}}{\text{\ensuremath{\Delta}}}\right)_{e,i}^{n}\,.
\]

A basic analysis of the discriminant gives admissibiliy conditions
on $\rho_{\varepsilon}$ that are expected to limit the time step.
In one dimension and in the single layer case for instance, the quantity
$\rho_{\varepsilon}$ reduces to $2{\displaystyle \dfrac{\Delta t}{\Delta x}\dfrac{c}{\varepsilon}}$,
where $c=\sqrt{gh}$ is the gravity wave speed, so that the smallness
assumptions made on $\rho_{\varepsilon}$ are satisfied under a classical
explicit CFL condition (i.e. of the form (\ref{CFL_1})). This is
also the case for the general $L$ layers case in two dimensions,
where the conditions required on $\rho_{\varepsilon}$ are always
satisfied with such a time constraint.\end{remark}

\begin{remark}\label{Rq1}As it has been confirmed by our numerical
experiments, if the values $\gamma=4$ and $\alpha=2$ ensure a global
decrease of the mechanical energy, they also bring too much diffusion
in practice, entailing dramatic restrictions on the space step. This
compels us to seek for relaxed conditions, that can be extracted from
a more general (and complex) analysis of the discrete energy budgets,
not described here for the sake of readability. As a matter of fact,
the optimality of the current approach has been lost within the Jensen's
inequalities used during the estimations of the kinetic and potential
energies (formulas (\ref{JensenK}) and (\ref{JensenP}) respectively).
If an explicit choice has been made on the weights to make things
more concrete, a global result can be established introducing a general
set of constants in these two inequalities. Playing with these parameters
and $\beta$ (see CFL condition (\ref{CFL2})), one can significantly
relax the conditions on the stabilization constants. In the single
layer case and one-dimensional problem for instance, the condition
on $\gamma$ becomes: 
\begin{equation}
\gamma\in\left[\gamma^{-},\gamma^{+}\right]\quad,\quad\text{with}\quad\gamma^{\pm}=\dfrac{1\pm\sqrt{1-\rho_{\varepsilon}^{2}}}{\rho_{\varepsilon}^{2}}\,,\label{gamma_law}
\end{equation}
where we recall that $\rho_{\varepsilon}=2\dfrac{\Delta t}{\Delta x}\dfrac{c}{\varepsilon}$.
A very close result is obtained for $\alpha$:

\begin{equation}
\alpha\in\left[\alpha^{-},\alpha^{+}\right]\quad,\quad\text{with}\quad\alpha^{\pm}=\dfrac{1\pm\sqrt{1-2\rho_{\varepsilon}^{2}}}{2\rho_{\varepsilon}^{2}}\,.\label{alpha_law}
\end{equation}
When $\rho_{\varepsilon}$ (or equivalently the CFL number) decreases,
a more important latitude regarding the choice of $\gamma$ and $\alpha$
is obtained, as illustrated in Fig.\ref{AGK}. And when $\rho_{\varepsilon}$
tends to zero, one recovers the critical value $\gamma=\alpha=1/2$.

\begin{figure}[H]
\begin{centering}
\includegraphics[width=0.49\textwidth]{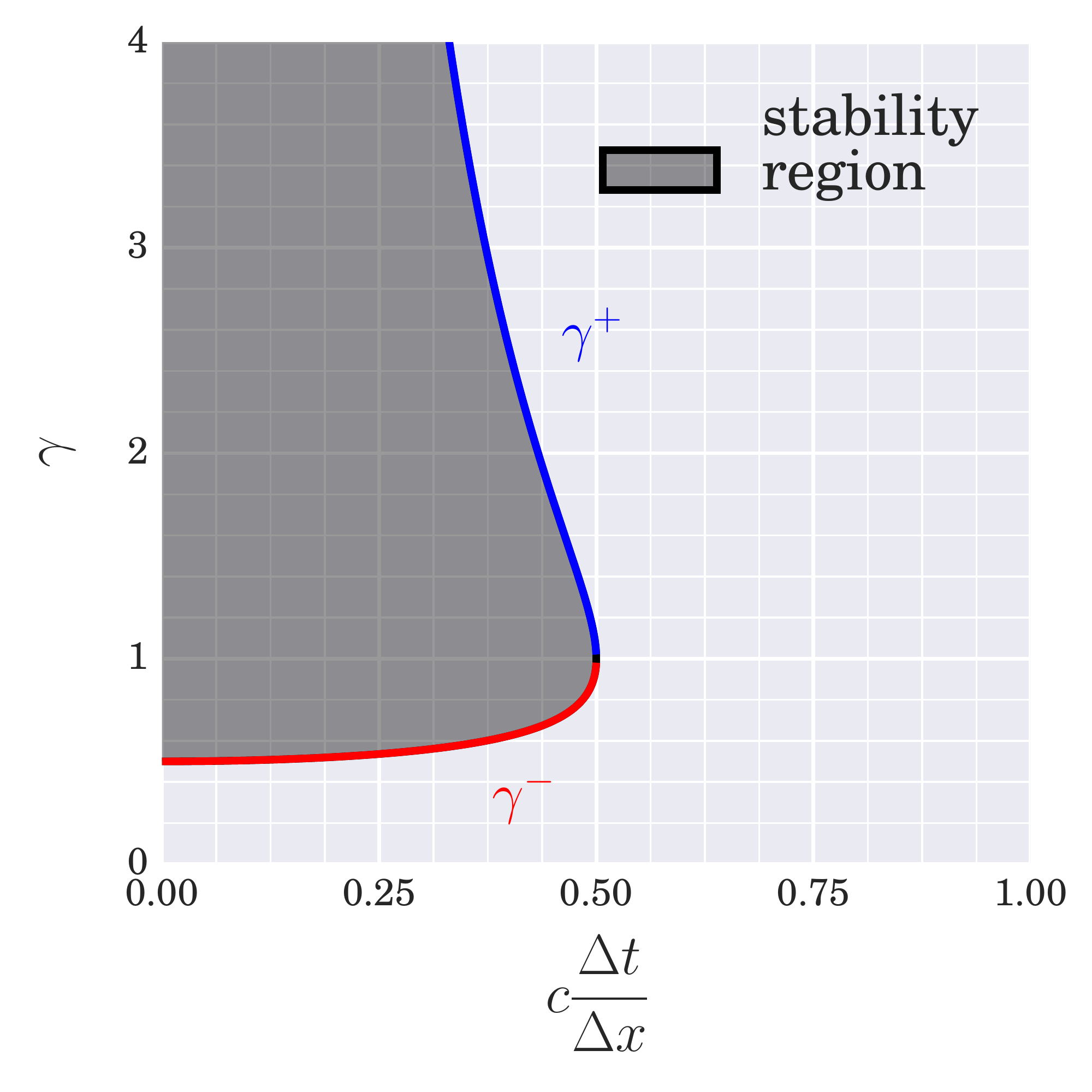}\includegraphics[width=0.49\textwidth]{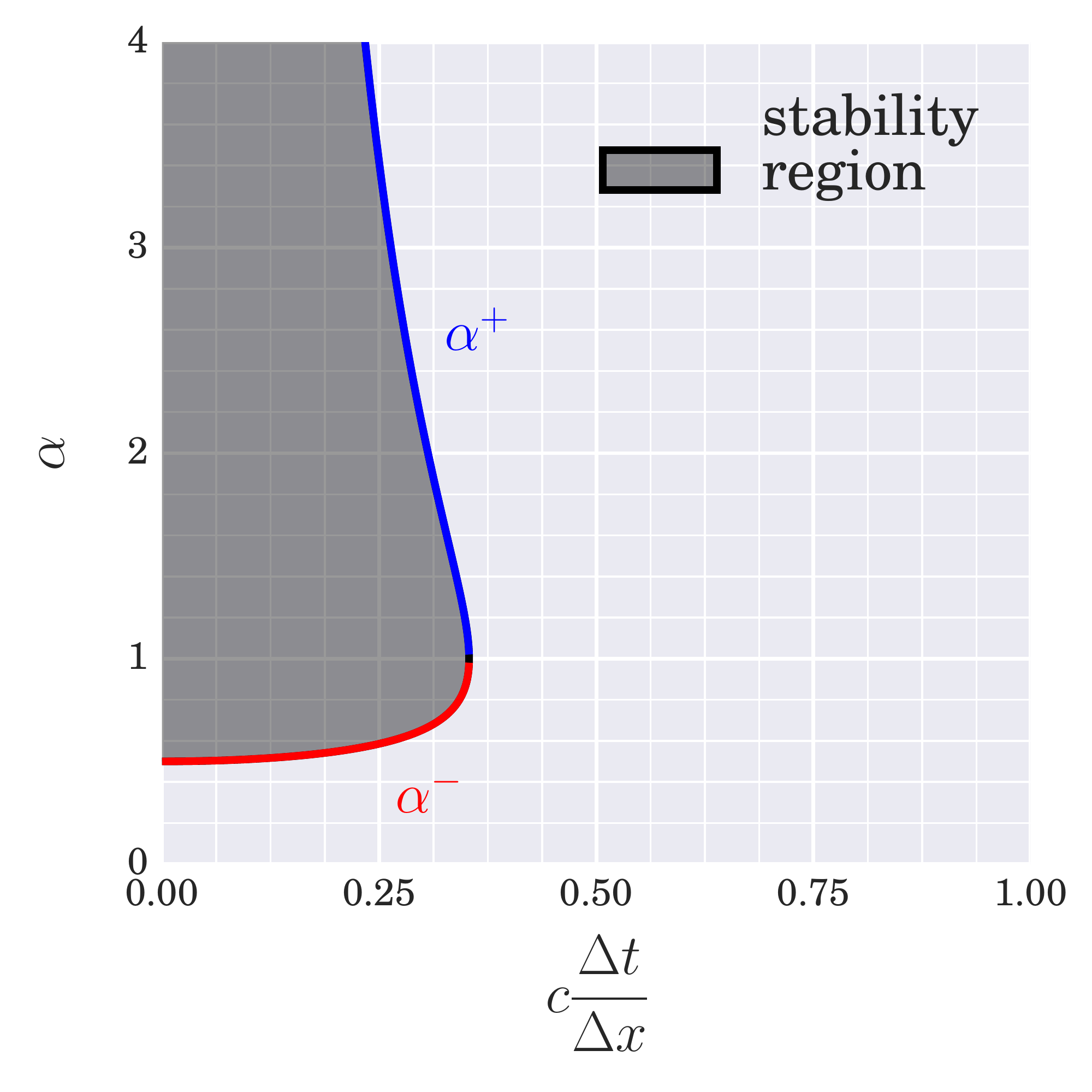}
\par\end{centering}
\caption{Non-linear discrete analysis: evolution of the lower and upper bounds
for $\gamma$ (\textit{left}) and $\alpha$ (\textit{right}) with
respect to the CFL number, based respectively on (\ref{gamma_law})
and (\ref{alpha_law}).\label{AGK} }
\end{figure}
As a result, one can get stability taking $\alpha$ and $\gamma$
in the vicinity of $1/2$ at first-order in space and time, even in
the general case of arbitrary stratifications. As it will be discussed
later, less restrictive conditions will be extracted from the linear
stability analysis (see \S \ref{linear}) with the use of MUSCL space
scheme (Appendix \ref{subsec:second-order-scheme}) coupled with the
Heun's method for time discretization (Appendix \ref{subsec:time-stepping-coriolis}).
Indeed, the stabilizing effects of the second-order time scheme allow
to considerably relax the stabilization constants, in conformity with
our numerical observations.\end{remark} 

\begin{remark}\label{Explicit} 

As discussed in Remark \ref{Implicit}, the rigorous definition of
$\widehat{H}_{K}^{n}$ appearing in the numerical fluxes through (\ref{pien},\ref{hmu})
given in (\ref{hath}) implies an implicit time step. With the support
of some numerical experiments, we already motivated the reasons of
substituting $H_{K}^{n}$ to $\widehat{H}_{K}^{n}$ for practical
applications, since this simplified choice only introduces an error
in the order of $\Delta t$ and does not change the asymptotic behaviour
of the scheme. However, we have to specify here that at the price
of being slightly more restrictive, a fully explicit stability condition
can be given. The strategy implies a global calibration of the stabilization
parameters, (i.e. for which we set $\widehat{H}_{K,i}^{n}=\widehat{H}_{i}^{n}\thinspace,\thinspace\forall\thinspace K\in\mathbb{T}$),
allowing to reduce to the study of a cubic polynomial (rather than
quadratic) at the level of each element. For the sake of readability
and to alleviate the proofs, we made the choice of presenting the
scheme in its present form.\end{remark}

\begin{remark}\label{Control}

As it has been discussed above, the negativity domain of the polynomials
$p$ and $q$ defined in (\ref{p}) and (\ref{q}) respectively can
be enlarged by diminishing the CFL. One of the consequences is that
the control (\ref{ControlE}) can be extended to obtain a strict mechanical
energy decrease. More precisely, let us consider a small parameter
$\delta>0$, and a combination of values $\left(\Delta t,\alpha,\gamma\right)$
satisfying (\ref{p}) and (\ref{q}). Considering the dominant coefficient
of $p$ and $q$, one easily obtains $p(\gamma)<-\delta$ and $q(\alpha)<-\delta$
with a time step $\Delta t$ subject to an $\mathcal{O}(\delta)$
perturbation. Then, gathering (\ref{gamma}) and (\ref{alpha}), we
obtain:
\begin{equation}
\begin{split}{E^{n+1}}-{E^{n}}\leq\, & -\delta\left(\Delta t\right)^{2}\sum_{K}{\sum_{i=1}^{L}}{\sum_{e\in\partial K}}\left(\frac{\widehat{H}}{\text{\ensuremath{\Delta}}}\right)_{e,i}^{n}\norm{\dfrac{{\boldsymbol{\delta}\Phi_{e,i}^{n}}}{\varepsilon^{2}}}^{2}{m_{e}}\,\\
 & -\delta\left(\Delta t\right)^{2}\sum_{K}{\sum_{i=1}^{L}}{\sum_{e\in\partial K}}{C_{{\boldsymbol{\mathcal{H}}}}}\frac{1}{\Delta_{e}}\left(\dfrac{{{\delta}(H\textbf{u})_{e,i}^{n}}}{\varepsilon}\right)^{2}{m_{e}}\,.
\end{split}
\label{Strict_decrease}
\end{equation}
These estimates give a control of $L^{1}(0,T,H_{w}^{1}(\Omega))(u)$
with some ad hoc weighted semi-norm on $H_{w}^{1}$. They insure validity
of Lax Wendroff type theorem for weak consistency of conservative
terms (in divergence form) in mass, momentum and energy equations.
We refer to \citep{Eymard2000} and also \citep{Vila2003} for further
details concerning the use of such estimates to study consistency
and convergence of the methods.\end{remark}

\subsection{Linear stability analysis\label{linear}}

We aim here at assessing the relevance of the previous energy dissipation
considerations through linear stability arguments. For the sake of
clarity the developments of the current section are given for the
one-dimensional problem, considering a regular mesh and the one layer
case ($L=1$). We specify at the end how these results can be easily
extended to the two-dimensional problem. The elements will be indexed
by $k$ and we denote $\mathcal{F}_{k+1/2}^{n}$ the numerical edge
flux between the cells $k$ and $k+1$. Let us take the example of
negative fluxes, for which we have:
\[
\begin{array}{l}
\left(\mathcal{F}_{k+1/2}^{n}\right)^{-}=\mathcal{F}_{k+1/2}^{n}\quad\text{and}\quad\left(\mathcal{F}_{k+1/2}^{n}\right)^{+}=0\,.
\end{array}
\]
In that context the equations (\ref{mass}), (\ref{u1}) constituting
the first-order scheme simplify as follows: 
\begin{equation}
\begin{split} & H_{k}^{n+1}=H_{k}^{n}-\dfrac{\Delta t}{\Delta x}\left(\mathcal{F}_{k+1/2}^{n}-\mathcal{F}_{k-1/2}^{n}\right)\,,\\
 & u_{k}^{n+1}=u_{k}^{n}-\dfrac{\Delta t}{\Delta x}\left(\dfrac{u_{k+1}^{n}-u_{k}^{n}}{H_{k}^{n}}\mathcal{F}_{k+1/2}^{n}\right)-\dfrac{\Delta t}{\Delta x}{\color{teal}{\normalcolor \dfrac{H_{K,i}^{n}}{H_{K,i}^{n+1}}}}\left(\Phi_{k+1/2}^{n,\ast}-\Phi_{k-1/2}^{n,\ast}\right)\,.
\end{split}
\label{scheme2}
\end{equation}
with the following numerical mass flux:
\begin{equation}
\mathcal{F}_{k+1/2}^{n}=\dfrac{H_{k}^{n}u_{k}^{n}+H_{k+1}^{n}u_{k+1}^{n}}{2}-2\gamma\dfrac{\Delta t}{\Delta x}\left(\dfrac{H_{k}^{n}+H_{k+1}^{n}}{2}\right)\left(\dfrac{\Phi_{k+1}^{n}-\Phi_{k}^{n}}{2}\right)\,,
\end{equation}
and a corrected potential of the form:
\begin{equation}
\Phi_{k+1/2}^{n,\ast}=\left(\dfrac{\Phi_{k+1}^{n}+\Phi_{k}^{n}}{2}\right)-2\alpha\dfrac{\Delta t}{\Delta x}C_{\boldsymbol{\mathcal{H}}}\left(\dfrac{H_{k+1}^{n}u_{k+1}^{n}-H_{k}^{n}u_{k}^{n}}{2}\right)\,.\label{eq:Phi1d}
\end{equation}
The scheme (\ref{scheme2}) is linearized around the constant state
$\bar{w}=\left(\bar{H},\bar{u}\right)$. Introducing a generic perturbation
$\tilde{w}_{k}^{n}=\left(\tilde{H}_{k}^{n},\tilde{u}_{k}^{n}\right)$
on the flow, we write:
\[
H_{k}^{n}=\bar{H}+\tilde{H}_{k}^{n}\quad u_{k}^{n}=\bar{u}+\tilde{u}_{k}^{n}\,,
\]
to obtain the following linearized system:

\begin{subequations}
\begin{empheq}[left=\empheqlbrace]{align}
{\tilde{H}_{k}^{n+1}}&={\tilde{H}_{k}^{n}}-\dfrac{\Delta t}{\Delta x}\left[\bar{H}\delta^{n}[\tilde{u}_{k}]+\bar{u}\delta^{n}[\tilde{H}_{k}]-2\gamma\bar{\Phi}_{H}\dfrac{\Delta t}{\Delta x}\bar{H}\Delta^{n}[\tilde{H}_{k}]\right]\,,\label{linearized_1}
\\ {\tilde{u}_{k}^{n+1}}&={\tilde{u}_{k}^{n}}-\dfrac{\Delta t}{\Delta x}\left[\bar{\Phi}_{H}\delta^{n}[\tilde{H}_{k}]+\bar{u}d_{+}^{n}[{\tilde{u}_{k}^{n}}]-2\alpha{C_{{\boldsymbol{\mathcal{H}}}}}\dfrac{\Delta t}{\Delta x}\left(\bar{H}\Delta^{n}[\tilde{u}_{k}]+\bar{u}\Delta^{n}[\tilde{H}_{k}]\right)\right]\, \label{linearized_2}    \end{empheq} \end{subequations}where we have set $\bar{\Phi}_{H}=\partial_{H}\Phi_{\vert\bar{H}}$,
and with the following discrete operators:
\[
\delta^{n}[f]=\dfrac{f_{k+1}^{n}-f_{k-1}^{n}}{2}\quad,\quad\Delta^{n}[f]=\dfrac{f_{k+1}^{n}+f_{k-1}^{n}-2f_{k}^{n}}{2}\quad,\quad d_{+}^{n}[f]=f_{k+1}^{n}-f_{k}^{n}\,.
\]
Looking classically for solutions of the form $w_{k}^{n}=\widehat{w}^{n}e^{{\displaystyle ik\Delta x}}$
to the system (\ref{linearized_1}, \ref{linearized_2}), we obtain
the following amplification matrix:

\[
\widehat{w}^{n+1}=\left(\begin{array}{c|c}
1\;-\;i\,\left(\dfrac{\Delta t}{\Delta x}\right)\,\bar{u}\,\sin\left(\Delta x\right) & -i\,\left(\dfrac{\Delta t}{\Delta x}\right)\,\bar{H}\,\sin\left(\Delta x\right)\\
+2\,\gamma\,\left(\dfrac{\Delta t}{\Delta x}\right)^{2}\,\bar{\Phi}_{H}\,\bar{H}\,\left(\cos\left(\Delta x\right)-1\right)\\
\\
\hline \\
-i\,\left(\dfrac{\Delta t}{\Delta x}\right)\,\bar{\Phi}_{H}\,\sin\left(\Delta x\right) & 1\;-\;\left(\dfrac{\Delta t}{\Delta x}\right)\,\bar{u}\,\left(e^{{\displaystyle i\,\Delta x}}-1\right)\\
+2\,\alpha\,\left(\dfrac{\Delta t}{\Delta x}\right)^{2}\,C_{\boldsymbol{\mathcal{H}}}\,\bar{u}\,\left(\cos\left(\Delta x\right)-1\right) & +2\,\alpha\,\left(\dfrac{\Delta t}{\Delta x}\right)^{2}\,C_{\boldsymbol{\mathcal{H}}}\,\bar{H}\,\left(\cos\left(\Delta x\right)-1\right)
\end{array}\right)\widehat{w}^{n}\,.
\]

If we now focus on the case $L=1$, one should note that the potential
energy is given by $\mathcal{E}=\frac{1}{2}gh^{2}$, and we have $\mathcal{H}=C_{\boldsymbol{\mathcal{H}}}=\bar{\Phi}_{H}=g$
(see (\ref{HL2}) and (\ref{Hessian})). Then the previous amplification
matrix characteristic polynomial induces a relation between the CFL
(i.e. $\bar{c}\dfrac{\Delta t}{\Delta x}$ where $\bar{c}=\sqrt{g\bar{H}}$)
and the stabilization parameters $\gamma$ and $\alpha$. The stabilization
constants are then substituted to their sum and product since it appears
obvious performing calculations. To illustrate that the sum $\gamma+\alpha$
is the main criteria to achieve linear stability and the product $\gamma\alpha$
a secondary influence, we propose several series of analysis in the
one-dimensional shallow water case developed around $\bar{u}=0$,
allowing to draw up a sampling of the linear stability domain, considering
two particular case studies: $\alpha=\gamma$ and $\alpha\gamma=0$.

Fig.\ref{GK1} (\textit{left}) shows the admissible range of CFL numbers
with respect to $\alpha+\gamma$ to achieve linear stability for the
two particular case studies. This analysis highlights $\alpha+\gamma=1$
as a necessary stability condition and a maximum CFL of 1 when $\alpha=\gamma$,
and a maximum CFL of $1/\sqrt{2}$ when $\alpha\gamma=0$. Even if
the maximum admissible CFL is reduced, it is a remarkable result to
find that taking one of the two stabilization constants to zero can
be sufficient to obtain linear stability. These results may be set
in relation with the optimized stability criteria issuing from the
non-linear study, that is the one-dimensional relaxed condition (\ref{gamma_law})
discussed in Remark \ref{Rq1}. As the non linear study requires both
$\alpha$ and $\gamma$ to be strictly positive, only the case $\alpha=\gamma$
is explored in Fig.\ref{GK1} (\textit{right}). As expected, the study
conducted in \S \ref{dissip_energy}, based on a strict energy dissipation
criteria, is more restrictive and fully embedded in the linear analysis.

\begin{figure}[!tbh]
\begin{centering}
\includegraphics[width=0.49\textwidth]{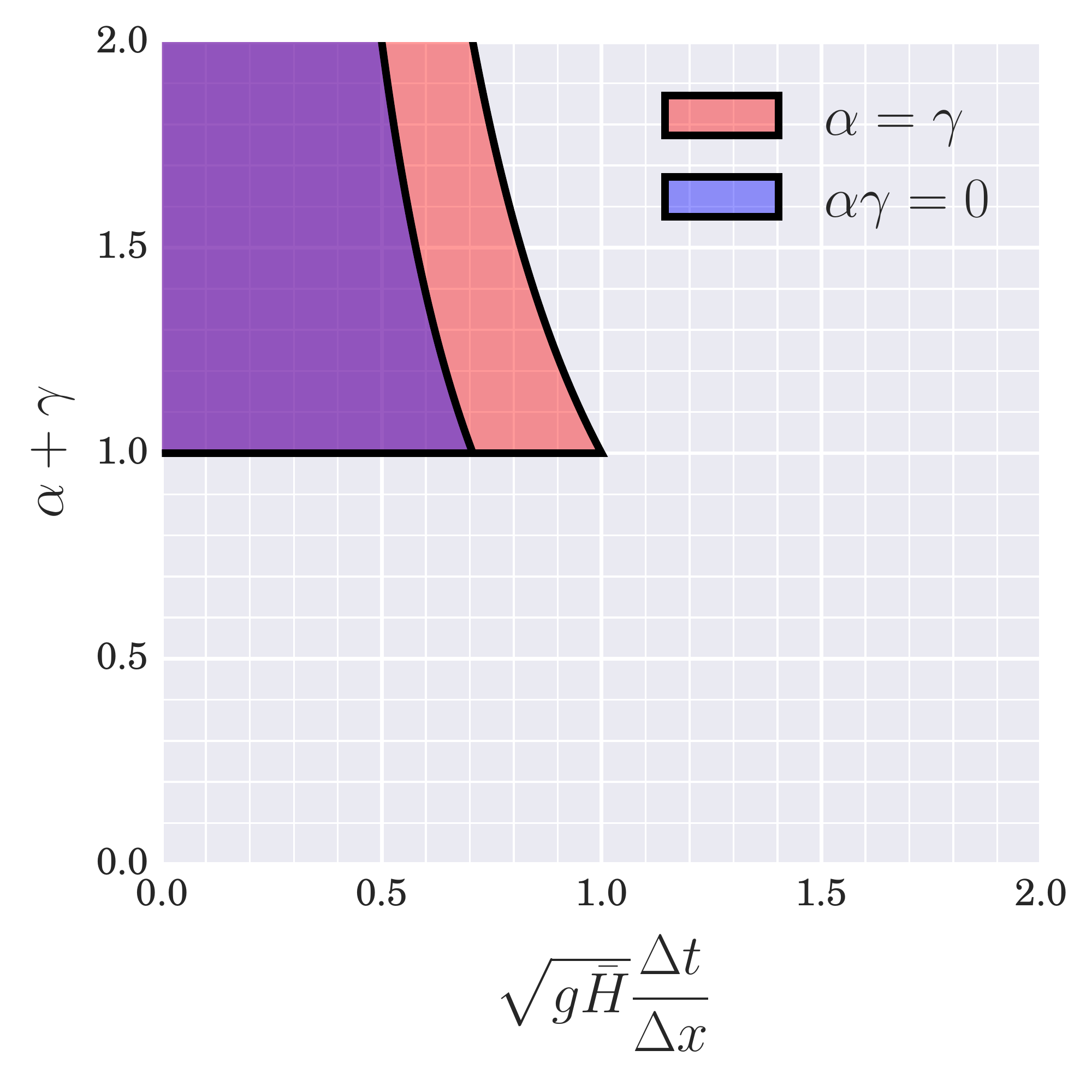} \includegraphics[width=0.49\textwidth]{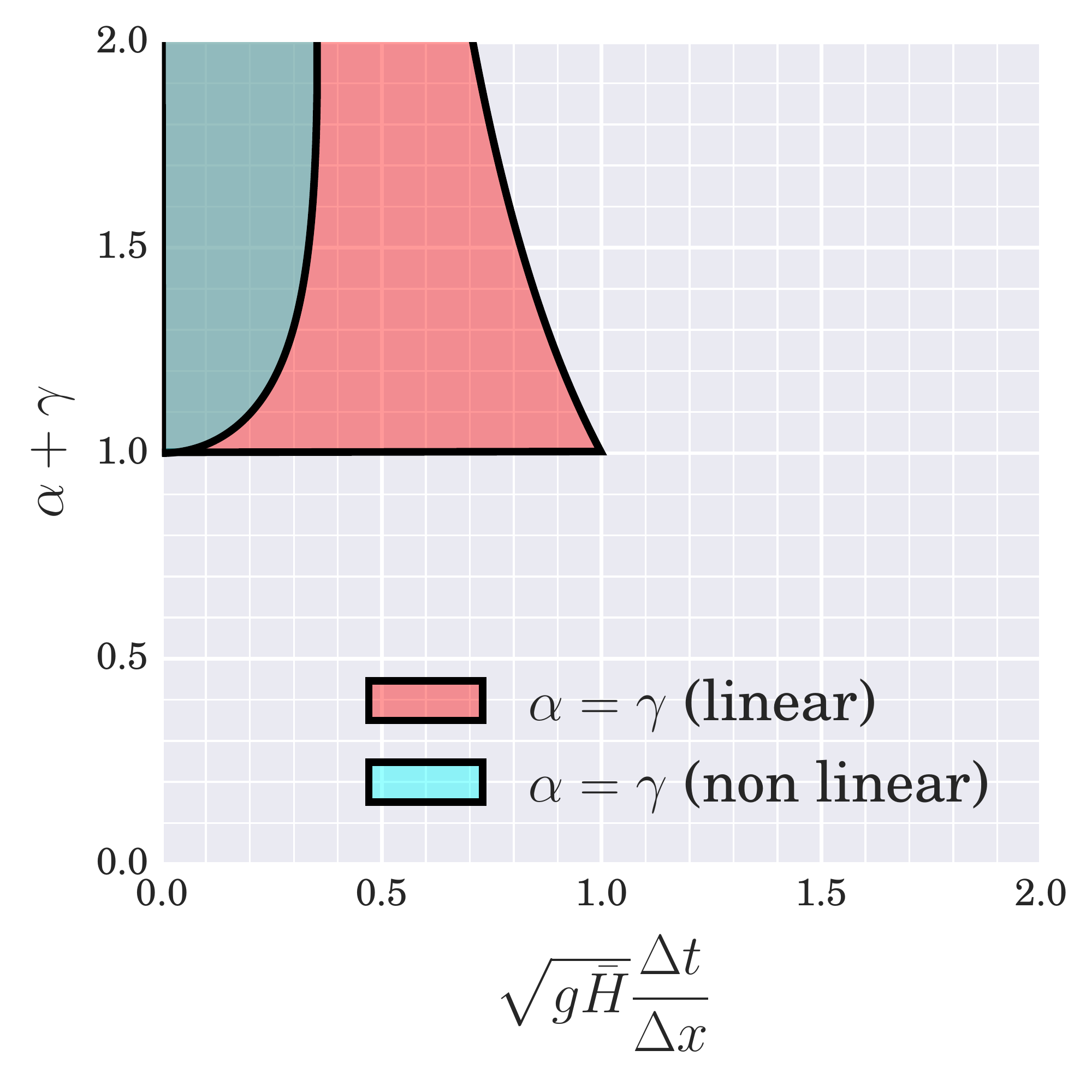} 
\par\end{centering}
\caption{One-dimensional linear stability analysis: (CFL, $\alpha+\gamma$)
sampling in the particular cases $\alpha=\gamma$ (\textit{red}) and
$\alpha\gamma=0$ (\textit{blue}) at first-order in space and time
(\textit{left}). Comparison with the relaxed condition issuing from
(\ref{gamma_law}) in the case $\alpha=\gamma$ (\textit{right}).
Two-dimensional linear stability analysis involves a rescaling of
the CFL numbers dividing them by\textcolor{teal}{{} ${\normalcolor \sqrt{2}}$.}
\label{GK1}}
\end{figure}

\begin{figure}[!tbh]
\begin{centering}
\includegraphics[width=0.49\textwidth]{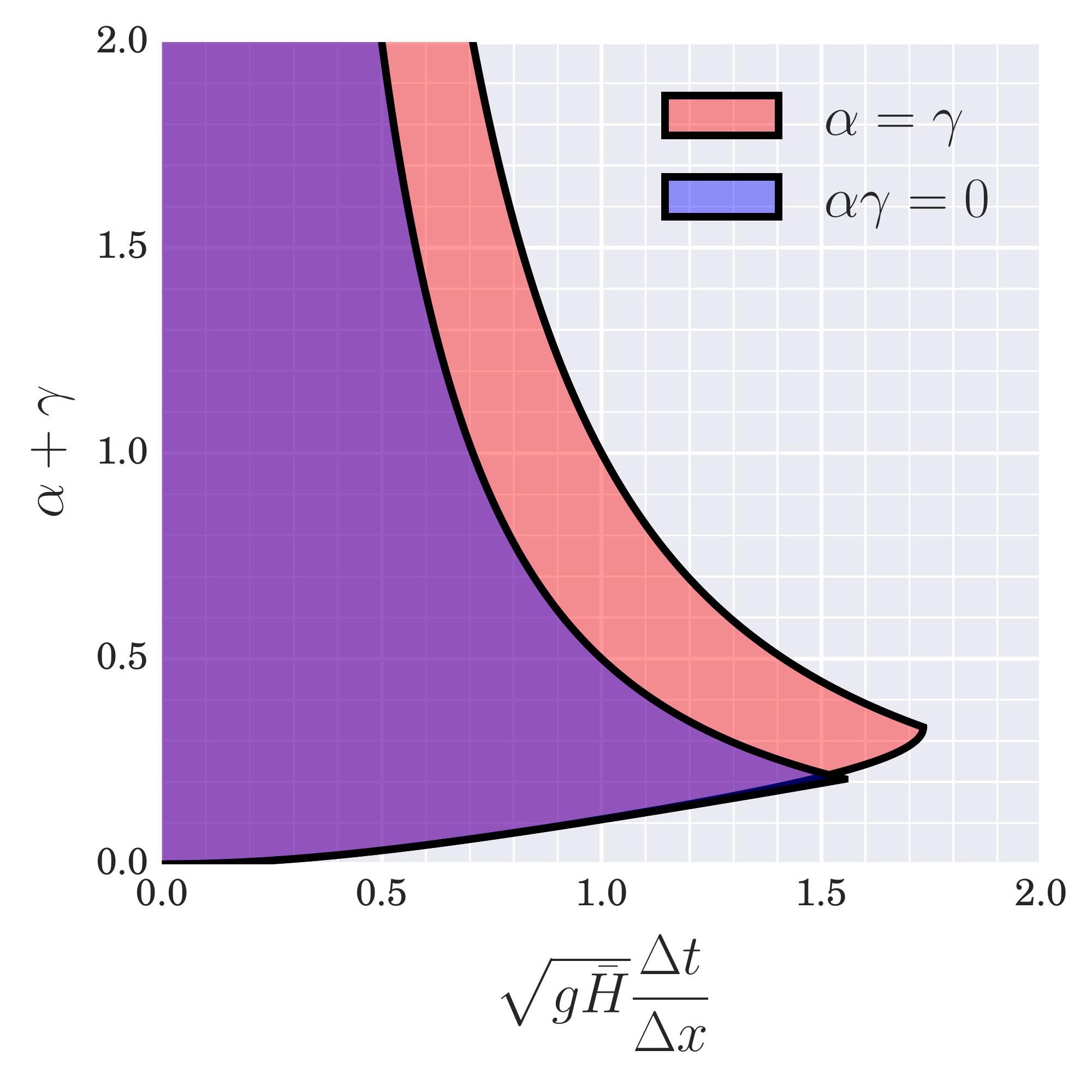} \includegraphics[width=0.49\textwidth]{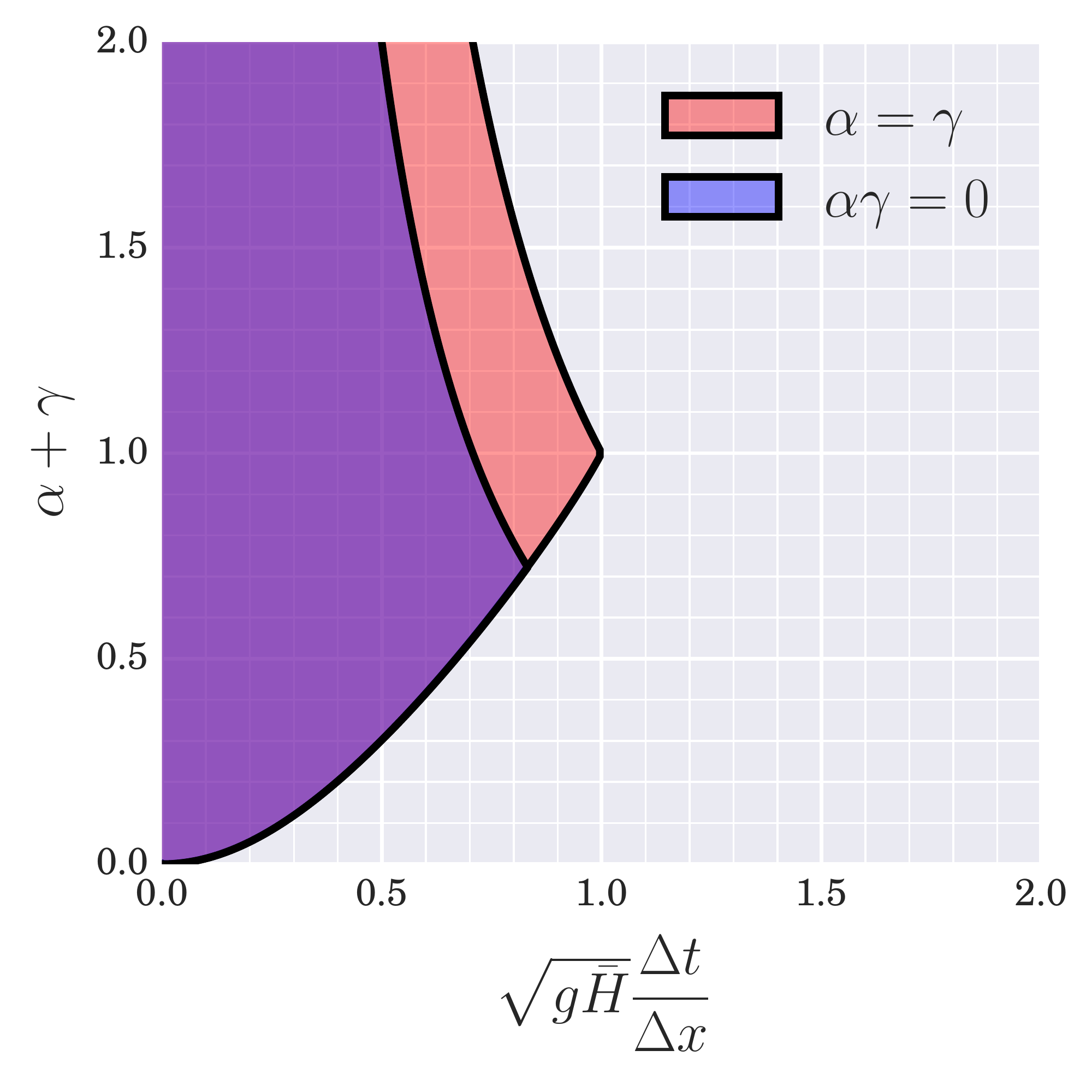} 
\par\end{centering}
\caption{One-dimensional linear stability analysis: (CFL, $\alpha+\gamma$)
sampling in the particular cases $\alpha=\gamma$ (\textit{red}) and
$\alpha\gamma=0$ (\textit{blue}) with a Heun's time stepping scheme.
First-order in space (\textit{left}) and second-order MUSCL reconstruction
scheme in space (\textit{right}). Two-dimensional linear stability
analysis involves a rescaling of the CFL numbers dividing them by\textcolor{teal}{{}
${\normalcolor \sqrt{2}}$.} \label{AG1}}
\end{figure}

As concerns the increase of time and space accuracy, if it is difficult
to exhibit explicit conditions based on the fully discrete model (we
refer however to Appendix \ref{subsec:Formal-second-order} for an
extension to MUSCL schemes), some interesting results can be established
in the linear case. Other series of tests were made integrating a
second-order MUSCL reconstruction in space, together with the Heun's
method for time discretization (see Appendix \ref{subsec:MUSCL} and
\ref{subsec:time-stepping-coriolis} for implementations purposes).
From a general point of view, the improvement of time order comes
with the possibility of substantial practical enhancements. As regards
first-order in space, the CFL can be increased and the admissible
range for $\gamma$ and $\alpha$ is significantly larger, as illustrated
in Fig.\ref{AG1} (\textit{left}). In particular, $\gamma$ and $\alpha$
can both be taken arbitrarily small at the price of sufficient time
step restrictions. The regularizing virtues of the second-order time
algorithm are still observed when considering a MUSCL reconstruction
(see Fig.\ref{AG1} (\textit{right})). These results are in accordance
with those provided by our simulations in linear regimes (see the
dedicated Section \ref{sec:Numerical-test-cases}). These conclusions
are of major interest from the extent that minimizing the diffusive
losses is essential in our applicative contexts.

All the previous results can be easily extended to the two-dimensional
problem. A first remarkable result is that the CFL numbers need to
be rescaled dividing by $\sqrt{2}$, and not $2$ as it could be anticipated.
A second result is that the $\alpha$ stabilization constant has to
be two times smaller to retrieve the one-dimensional results, obtained
with the corrected potential pressure stabilization term (\ref{eq:Phi1d}).

\section{Asymptotic regimes\label{sec:Asymptotic-regimes}}

We show in this part the asymptotic preserving features of the current
approach. Since the scheme reduces to a convex combination of 1d schemes
(see Appendix \ref{subsec:Reformulation-as-convex}), only the 1d
case is investigated. In the one-dimensional frame, for a given time
step $\Delta t$ and space step ${\Delta x}$, the numerical scheme
(\ref{mass}, \ref{mom}) can be interpreted at the semi-discrete
level as follows:

\begin{subequations}
\begin{empheq}[left=\empheqlbrace]{align}
H_{i}^{n+1}- H_{i}^{n}  &=  {\displaystyle \Delta t}\partial_{x}(Hu)_{i}^{n} + \left(\Delta t\right)^{2}\gamma\partial_{x}\left(H_{i}\dfrac{\partial_{x}{\Phi_{i}}}{\varepsilon^{2}}\right)^{n}\label{semi:mass}\\ 
(Hu)_{i}^{n+1}-(Hu)_{i}^{n}&=-\Delta t\left(\partial_{x}\left(\bar{u}_{i}(Hu)_{i}^{\ast}\right)\right)^{n}\vspace{1mm} \label{semi:mom}\\
&  \qquad -\Delta t\left(H_{i}\dfrac{\partial_{x}{\Phi_{i}}}{\varepsilon^{2}}\right)^{n}+\left(\Delta t\right)^{2}\alpha\left(H_{i}\dfrac{\partial_{xx}(Hu)_{i}}{\varepsilon^{2}}\right)^{n} \nonumber \, ,
\end{empheq} \end{subequations}where $(Hu)_{i}^{\ast}=(Hu)_{i}-\Delta t\gamma\left(H_{i}\dfrac{\partial_{x}{\Phi_{i}}}{\varepsilon^{2}}\right)$
, and $\bar{u}_{i}$ stands for the velocity $u_{i}$ perturbed with
a $\mathcal{O}({\Delta x})$ viscosity term resulting from the upwind
strategy on the momentum equations. Note that the space step is submitted
to a classical explicit CFL condition of the form: 
\begin{equation}
\dfrac{\Delta t}{\Delta x}\left(u+\dfrac{c}{\varepsilon}\right)\leq cste\,.\label{CFL_1D}
\end{equation}
Of course, a fully discrete analysis can be proposed, as done in \citep{Parisot2015}.
Nevertheless, at the end of the day, reformulating the scheme (\ref{mass},
\ref{mom}) in terms of discrete operators in one dimension, we are
left with the study of the semi-continuous scheme (\ref{semi:mass},
\ref{semi:mom}) subject to an $\mathcal{O}(\Delta x)$ perturbation,
which has no incidence on the asymptotic beahaviour. Thus, in this
section, the results will be established at the continuous level in
space for the sake of simplicity.

\subsection{Fine time scale}

For small time scale $t=\varepsilon\tau$ the model (\ref{Model0})
degenerates toward a system of wave equations (see \citep{Strauss1992},
\citep{Bresch2011}): 
\begin{equation}
\begin{aligned}\partial_{\tau\tau}^{2}H_{i}-\mathrm{div}(H_{i}\nabla{\Phi_{i}})=0\,.\end{aligned}
\label{Wave}
\end{equation}

\begin{theorem}Consistency with the wave equations (\ref{Wave}):
\\
 Consider the time step scaling $\Delta t=\varepsilon\Delta\tau$.
The semi-discrete model (\ref{semi:mass},\ref{semi:mom}) furnishes
an approximation of the wave equations (\ref{Wave}) with an error
in the order of $\mathcal{O}(\Delta\tau)$.\end{theorem}

\begin{proof}We drop the subscript \textit{``i'' }for the sake
of simplicity. Using the mass equation (\ref{semi:mass}) at times
$n$ and $n+1$: 
\[
\begin{array}{lllllll}
H^{n+1} & - & H^{n} & = & -\varepsilon\Delta\tau\partial_{x}(Hu)^{n} & + & \left(\Delta\tau\right)^{2}\gamma\partial_{x}\left(H\partial_{x}{\Phi}\right)^{n}\,,\\
H^{n} & - & H^{n-1} & = & -\varepsilon\Delta\tau\partial_{x}(Hu)^{n-1} & + & \left(\Delta\tau\right)^{2}\gamma\partial_{x}\left(H\partial_{x}{\Phi}\right)^{n-1}\,,
\end{array}
\]
we write: 
\begin{equation}
\dfrac{H^{n+1}-2H^{n}+H^{n-1}}{\left(\Delta\tau\right)^{2}}=-\dfrac{\varepsilon}{\Delta\tau}\left[\partial_{x}\left((Hu)^{n}-(Hu)^{n-1}\right)\right]+\gamma\left[\partial_{x}\left(\left(H\partial_{x}{\Phi}\right)^{n}-\left(H\partial_{x}{\Phi}\right)^{n-1}\right)\right]\,.\label{Wave_continuous}
\end{equation}
Consider now the momentum equations (\ref{semi:mom}), and multiply
by $\dfrac{\varepsilon}{\Delta\tau}$:
\begin{equation}
\begin{split}\dfrac{\varepsilon}{\Delta\tau}\left((Hu)^{n}-(Hu)^{n-1}\right)= & -\varepsilon^{2}\left(\partial_{x}\left(\bar{u}(Hu)^{\ast}\right)
\right)^{n-1}\\
 & -\varepsilon^{2}\left(H\dfrac{\partial_{x}{\Phi}}{\varepsilon^{2}}\right)^{n-1}+\varepsilon\Delta\tau\alpha\left(H\partial_{xx}\left((Hu)\right)\right)^{n-1}\,.
\end{split}
\label{MomF}
\end{equation}
Going back to the definition of $(Hu)^{\ast}$ we write: 
\[
\varepsilon^{2}(Hu)^{\ast}=\varepsilon^{2}\left(Hu-\varepsilon\Delta\tau\gamma\partial_{x}\left(H\dfrac{\partial_{x}{\Phi}}{\varepsilon^{2}}\right)\right)=\underset{\varepsilon\rightarrow0}{\mathcal{O}}(\varepsilon^{2})+\mathcal{O}(\Delta\tau)\,.
\]
Since $\bar{u}_{i}$ is $\mathcal{O}(1)$, we have as a direct consequence:
\[
\varepsilon^{2}\left(\partial_{x}\left(\bar{u}(Hu)^{\ast}\right)
\right)=\underset{\varepsilon\rightarrow0}{\mathcal{O}}(\varepsilon^{2})+\mathcal{O}(\Delta\tau)\,.
\]
Finally, substituting (\ref{MomF}) in (\ref{Wave_continuous}) we
obtain: 
\begin{equation}
\dfrac{H_{i}^{n+1}-2H_{i}^{n}+H_{i}^{n-1}}{\left(\Delta\tau\right)^{2}}=\partial_{x}\left(H\partial_{x}{\Phi}\right)_{i}^{n}+\underset{\varepsilon\rightarrow0}{\mathcal{O}}(\varepsilon^{2})+\mathcal{O}(\Delta\tau)\,,
\end{equation}
that is the one-dimensional equivalent of (\ref{Wave}) with an error
in the order of $\Delta\tau$ and a second-order perturbation.\end{proof}

\subsection{Large time scale}

Assuming the Hessian ${\boldsymbol{\mathcal{H}}}$ (\ref{Hessian})
well-conditioned with respect to $\varepsilon$, that is the condition
number of ${\boldsymbol{\mathcal{H}}}$ is $\underset{\varepsilon\rightarrow0}{\mathcal{O}}(1)$,
the asymptotic regime associated with large time scales $t=\underset{\varepsilon\rightarrow0}{\mathcal{O}}(1)$
can be derived as a divergence-free model:

\begin{equation}
\left\lbrace \begin{aligned}\mathrm{div}\left(H\textbf{u}_{i}\right) & =0\,\\
\partial_{t}\textbf{u}_{i}+\left(\textbf{u}_{i}.\nabla\right)\textbf{u}_{i} & =-\nabla\Phi_{i}
\end{aligned}
\right.\,.\label{DF}
\end{equation}

\begin{theorem} Consistency with the divergence free model (\ref{DF}):\\
 Consider the time step scaling $\Delta t=\underset{\varepsilon\rightarrow0}{\mathcal{O}}(1)$,
and assume that the spatial perturbation of the potential is in the
order of $\varepsilon^{2}$: 
\begin{equation}
\Phi_{i}(t,x)=\bar{\Phi}_{i}(t)+\varepsilon^{2}\hat{\Phi}_{i}(t,x)\,.\label{SP}
\end{equation}
Then the semi-discrete model (\ref{semi:mass},\ref{semi:mom}) furnishes
an approximation of the wave equations (\ref{DF}) with an error in
the order of $\mathcal{O}(\Delta t)$ and $\mathcal{O}(\Delta t,{\Delta x})$
respectively.\end{theorem}

\begin{proof}Note that with (\ref{SP}), and based on the regularity
assumptions made on the potential forces (\ref{Regularity}), we also
have: 
\[
\partial_{t}\Phi_{i}=\underset{\varepsilon\rightarrow0}{\mathcal{O}}(\varepsilon^{2})\qquad\text{ and }\qquad\partial_{t}H_{i}=\underset{\varepsilon\rightarrow0}{\mathcal{O}}(\varepsilon^{2})\,.
\]
Again, we drop the subscript \textit{``i''} to alleviate the notations.
We directly obtain from (\ref{semi:mass}): 
\begin{equation}
\partial_{x}\left(Hu^{\ast}\right)^{n}=\partial_{x}\left(Hu\right)^{n}-\Delta t\gamma\partial_{x}\left(H\partial_{x}\hat{\Phi}\right)^{n}=\underset{\varepsilon\rightarrow0}{\mathcal{O}}(\varepsilon^{2})\,,\label{DivFree}
\end{equation}
that is the divergence-free condition with an error in $\mathcal{O}(\Delta t)$
and a second-order perturbation. Using the relation: 
\[
(Hu)^{n+1}-(Hu)^{n}=\left(H^{n+1}-H^{n}\right)u^{n+1}+H^{n}\left(u^{n+1}-u^{n}\right)=H^{n}\left(u^{n+1}-u^{n}\right)+\underset{\varepsilon\rightarrow0}{\mathcal{O}}(\varepsilon^{2})\,,
\]
together with the momentum equation (\ref{semi:mom}) we write: 
\begin{equation}
\dfrac{u^{n+1}-u^{n}}{\Delta t}=-\dfrac{1}{H^{n}}\left(\partial_{x}\left(\bar{u}(Hu)^{\ast}\right)+\left(H\partial_{x}\hat{\Phi}\right)\right)^{n}+\Delta t\alpha\left(\dfrac{\partial_{xx}\left(Hu\right)^{n}}{\varepsilon^{2}}\right)+\underset{\varepsilon\rightarrow0}{\mathcal{O}}(\varepsilon^{2})\,.
\end{equation}
For any $n$, we first note that: 
\begin{equation}
\begin{split}\dfrac{1}{H}\partial_{x}\left(\bar{u}(Hu)^{\ast}\right) & =u\partial_{x}u+\underset{\varepsilon\rightarrow0}{\mathcal{O}}(\varepsilon^{2})+\mathcal{O}(\Delta t,{\Delta x})\end{split}
\end{equation}
going back to the semi-discrete divergence free relation (\ref{DivFree}),
one has: 
\begin{equation}
\begin{split}\partial_{xx}\left(Hu\right)=\Delta t\gamma\partial_{xx}\left(H\partial_{x}\hat{\Phi}\right)+\underset{\varepsilon\rightarrow0}{\mathcal{O}}(\varepsilon^{2})\,,\end{split}
\end{equation}
and hence: 
\begin{equation}
\Delta t\alpha\left(\dfrac{\partial_{xx}\left(Hu\right)}{\varepsilon^{2}}\right)=\left(\dfrac{\Delta t}{\varepsilon}\right)^{2}\alpha\gamma\partial_{xx}\left(H\partial_{x}\hat{\Phi}\right)+\mathcal{O}(\Delta t)\,.
\end{equation}
Under the explicit CFL (\ref{CFL_1D}), the first term of the right
hand side is $\mathcal{O}({\Delta x}^{2})$. At last this gives: 
\begin{equation}
\dfrac{u_{i}^{n+1}-u_{i}^{n}}{\Delta t}=-u_{i}\partial_{x}u_{i}-\partial_{x}\hat{\Phi}_{i}+\underset{\varepsilon\rightarrow0}{\mathcal{O}}(\varepsilon^{2})+\mathcal{O}(\Delta t,{\Delta x})\,.
\end{equation}
\end{proof}

\section{Numerical test cases\label{sec:Numerical-test-cases}}

This part is dedicated to the survey of the numerical scheme's global
efficiency at first and second-order, with a particular focus on low
Froude regimes. Theoretical and numerical investigations involving
wet/dry fronts and a complex management of the layers are left for
future works. For the sake of completeness, the second-order extension
in space and time, the adaptive time step used and the time stepping
scheme to incorporate the Coriolis force are given in the Appendix
\ref{sec:Appendix}. We recall here that all the numerical tests were
performed with $\widehat{H}_{K}^{n}=H_{K}^{n}$ in the numerical fluxes
(\ref{fe}, \ref{pien}, \ref{hmu}) (see Remarks \ref{Implicit}
and \ref{Explicit}).

It should be emphasized that it is difficult to carry on qualitative
comparisons on the different numerical approaches in case of multiple
layers. This is mainly due to the very small number of such academic
test cases available in the literature, as it is difficult to derive
analytical solutions. Some reference solutions for the multilayer
shallow water model with the Coriolis force are of course provided
by more sophisticated operational softwares like HYCOM \citep{Bleck2002},
ROMS \citep{Shchepetkin2005} or NEMO \citep{Madec2008}, but with
the inconvenience of not being necessary exactly based on the same
physical model as the one concerned here.

A first academic test case is considered, involving two-dimensional
oscillating layers around a steady state in the linear small amplitude
limit. It is investigated for this test case the inequality conditions
for the stabilization constants $\gamma$ and $\alpha$ ensuring the
linear stability, as well as those which guarantee the strict decrease
of the mechanical energy between each time step. At the end, the \textit{a
priori} best pair verifying the two stability conditions with a minimum
of dissipation will be extracted and the resulting scheme compared
to the classical HLLC approximate Riemann solver (see \citep{Toro2001}
with the wave speed estimates of \citep{Vila1986}). In a second test
case, the scheme's accuracy is investigated for a smooth two-dimensional,
non-stationnary and non-linear solution. In a third test case, we
focus on the well-balanced property, considering an initial jump of
water surface elevation propagating over a non trivial topography.
A more advanced test case is finally studied, the so-called baroclinic
vortex, that can be found in the COMODO benchmark \citep{COMODO},
a test suite set up by the international oceanographic community to
evaluate and compare the numerical solvers efficiency.

\subsection{Linear waves\label{subsec:linear-waves}}

In the present test case we investigate the two-dimensional simulation
of oscillating layers around a steady state of flat layers with a
flat bottom. In case of waves of small amplitude, an approximate analytical
solution can be derived from the linear wave theory. Considering only
one layer, in the limit of small amplitude variation around the layer
depth at rest $\eta_{0}$, the deviation $\zeta_{1}$ ($\eta_{1}=\eta_{0}+\zeta_{1}$)
is solution of the two-dimensional wave equation with the associated
dispersive relation $\omega^{2}=c^{2}\left(k_{x}^{2}+k_{y}^{2}\right)$,
where $k_{x}$ and $k_{y}$ are the wave numbers in the $x$ and $y$
direction respectively. Considering now the same problem for the $L$
layers shallow water model, we obtain $L$ coupled linear wave equations,

\[
\forall i\in\left\llbracket 1,L\right\rrbracket ,\quad{\displaystyle \frac{\partial^{2}\zeta_{i}}{\partial t^{2}}-c_{i}^{2}\sum_{j=1}^{L}\frac{min(\rho_{i},\rho_{j})}{\rho_{i}}\triangle\zeta_{j}=0}\,.
\]
By denoting $\tilde{c}_{i}$ the eigenvalues of the matrix ${\displaystyle A_{ij}=\left(c_{i}^{2}\frac{\min(\rho_{i},\rho_{j})}{\rho_{i}}\right)}$,
the above coupled system of wave equations can be rewritten in $L$
uncoupled linear wave equations,

\[
\forall i\in\left\llbracket 1,L\right\rrbracket ,\quad{\displaystyle \frac{\partial^{2}\tilde{\zeta}_{i}}{\partial t^{2}}-\tilde{c}_{i}^{2}\triangle\tilde{\zeta}_{i}=0}\,,
\]

\noindent where $\tilde{\zeta}$ is the projection of $\zeta$ onto
the diagonal basis using the left eigenvectors matrix. Simulations
are initialized in a $100\ \mathrm{km}$ square box with periodic
boundary conditions, a sea surface at rest $\eta{}_{0}=5000\ \mathrm{m}$,
five evenly spaced layers $h_{i}=1000\ \mathrm{m}$ with densities
following a linear law $\rho_{i}=1000+50\:(i-1)$ and a gravitational
acceleration $g=10\ \mathrm{m}.\mathrm{s}^{-2}$. Note that the density
ratios considered here, large in comparison with those encountered
in more realistic contexts like oceans, have the effect of reducing
the wave phase speed differences and allowing to consider a smaller
time integration to capture the layers interaction. Considering a
deviation $\zeta_{1}=\cos\left(k_{x}x\right)\cos\left(k_{y}y\right)$
only for the first layer with one wavelength in each direction, approximatively
$11$ wave periods can be observed with a simulation time $t=3600\:\mathrm{\mathrm{s}}$
for a maximum wave velocity $\max\left(\tilde{c_{i}}\right)\approx\sqrt{2gh_{0}}\approx316\ \mathrm{m}.\mathrm{s}^{-1}$.
We give here the expression of the discrete mechanical energy:

\begin{equation}
{\normalcolor E^{n}={\displaystyle \frac{1}{2}}{\displaystyle \sum_{K\in\mathbb{T}}\sum_{i=1}^{L}m_{K}\frac{\rho_{i}}{\rho_{L}}\left(h_{K,i}^{n}\left\Vert \mathbf{u}_{K,i}^{n}\right\Vert ^{2}/\epsilon^{2}+g\left(h_{K,i}^{n}\right)^{2}+2\sum_{j=i+1}^{L}gh_{K,i}^{n}h_{K,j}^{n}\right)}}\,,\label{eq:Discrete_Energy}
\end{equation}

\noindent as it will be a useful measurement for the simulations presented
above. Note finally that $\epsilon=2.10^{-4}$ for this test, giving
a very low Froude solution.

\subsubsection{Stability issues - searching for optimal stabilization parameters}

\begin{figure}[!tbh]
\begin{centering}
\includegraphics[width=0.4\textwidth]{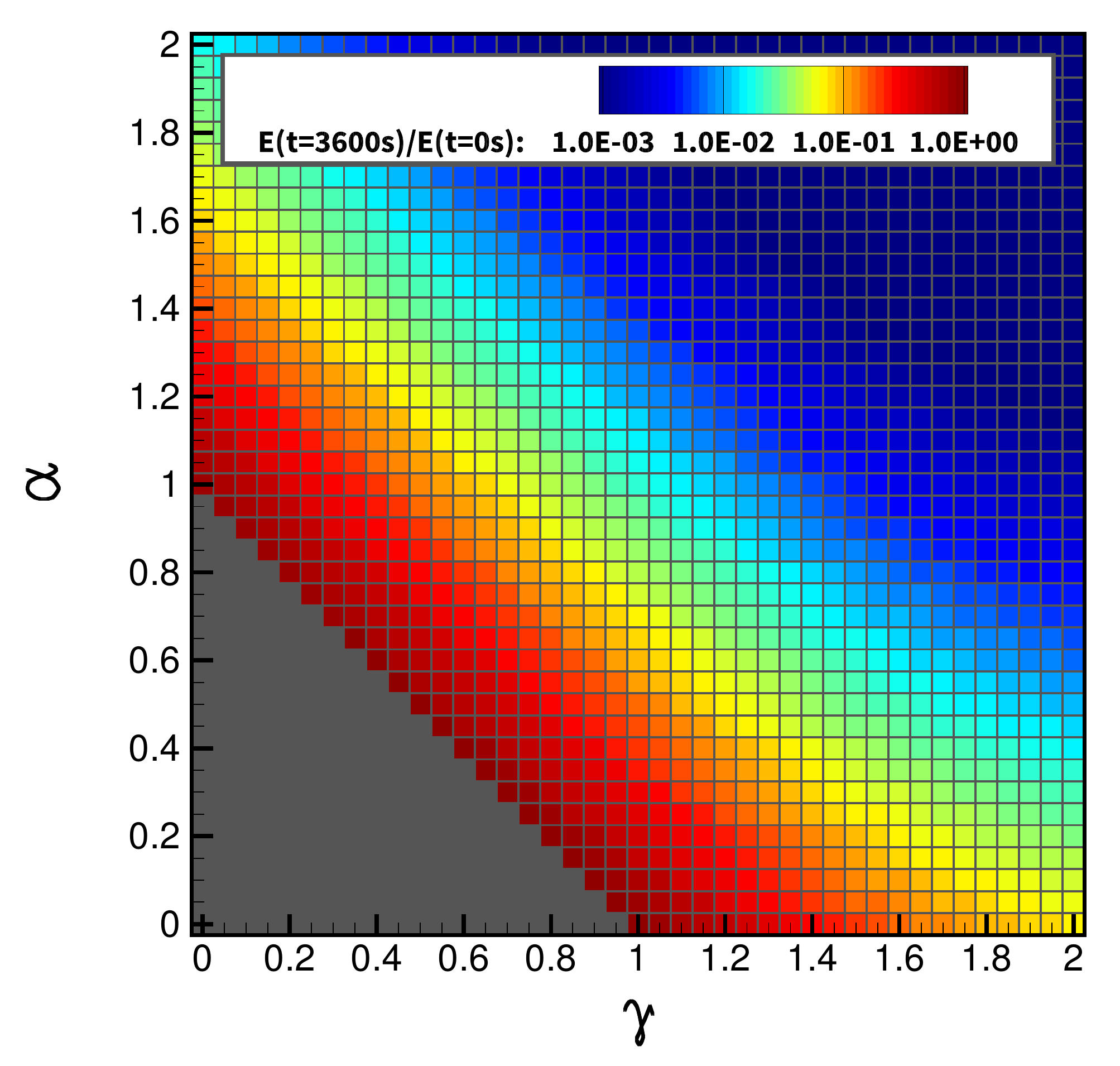}
\includegraphics[width=0.4\textwidth]{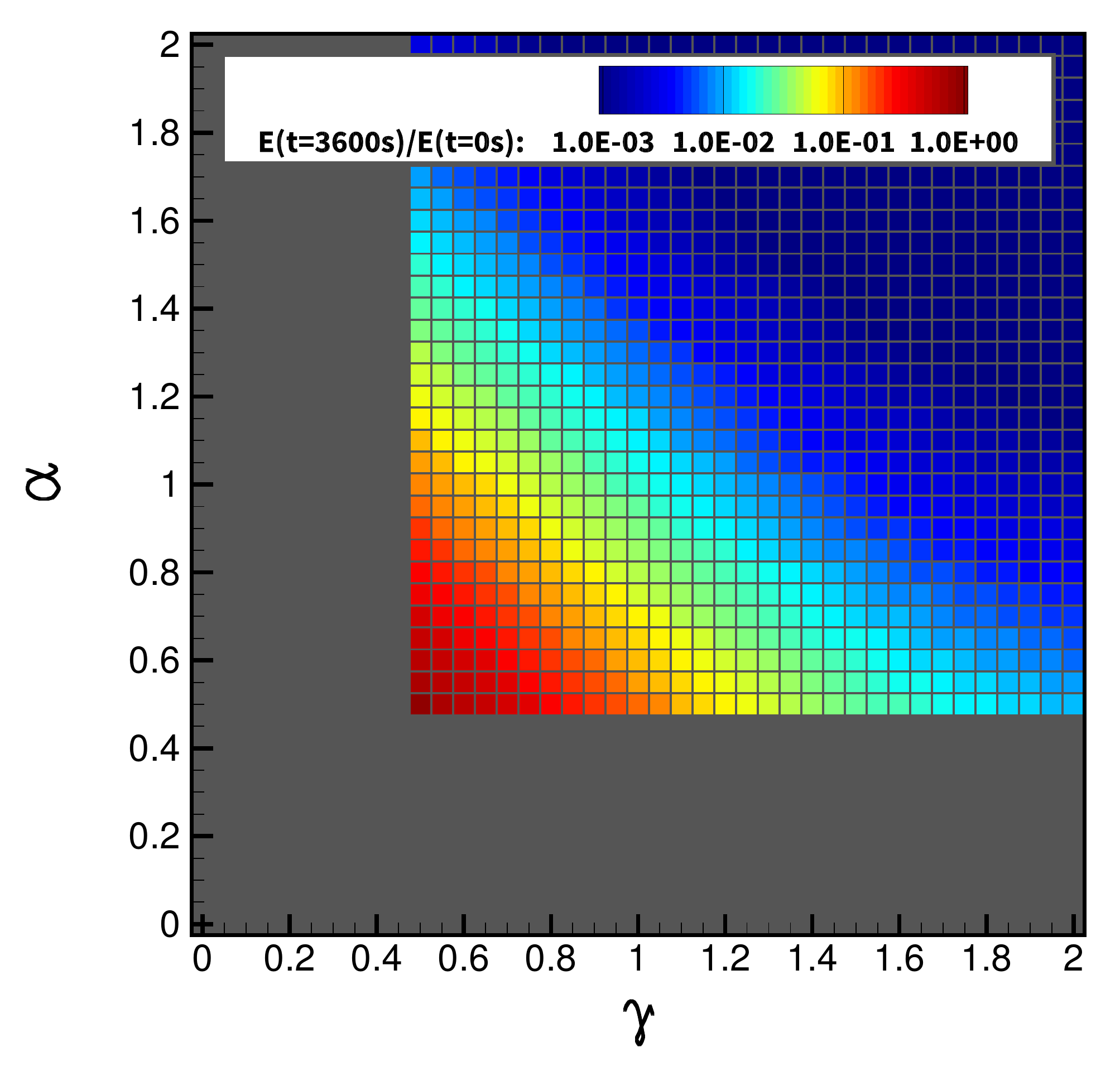} 
\par\end{centering}
\caption{Mechanical energy dissipation according to $\gamma$ and $\alpha$
with a fixed CFL number of $0.5$ in (\ref{eq:time-step}) using the
first-order scheme and a $41\times41$ mesh size ; (\textit{left})
gray zone corresponds to an unstable algorithm ; (\textit{right})
gray zone corresponds to a non monotonically decreasing energy. The
energy ratio $E(t=3600\,\mathrm{s})/E(t=0\,\mathrm{s})$, computed
from (\ref{eq:Discrete_Energy}) and displayed in $log$ scale, highlights
the scheme's dissipation.\label{fig:energy-ratio-gamma-alpha-1}}
\end{figure}

\begin{figure}[!tbh]
\begin{centering}
\includegraphics[width=0.4\textwidth]{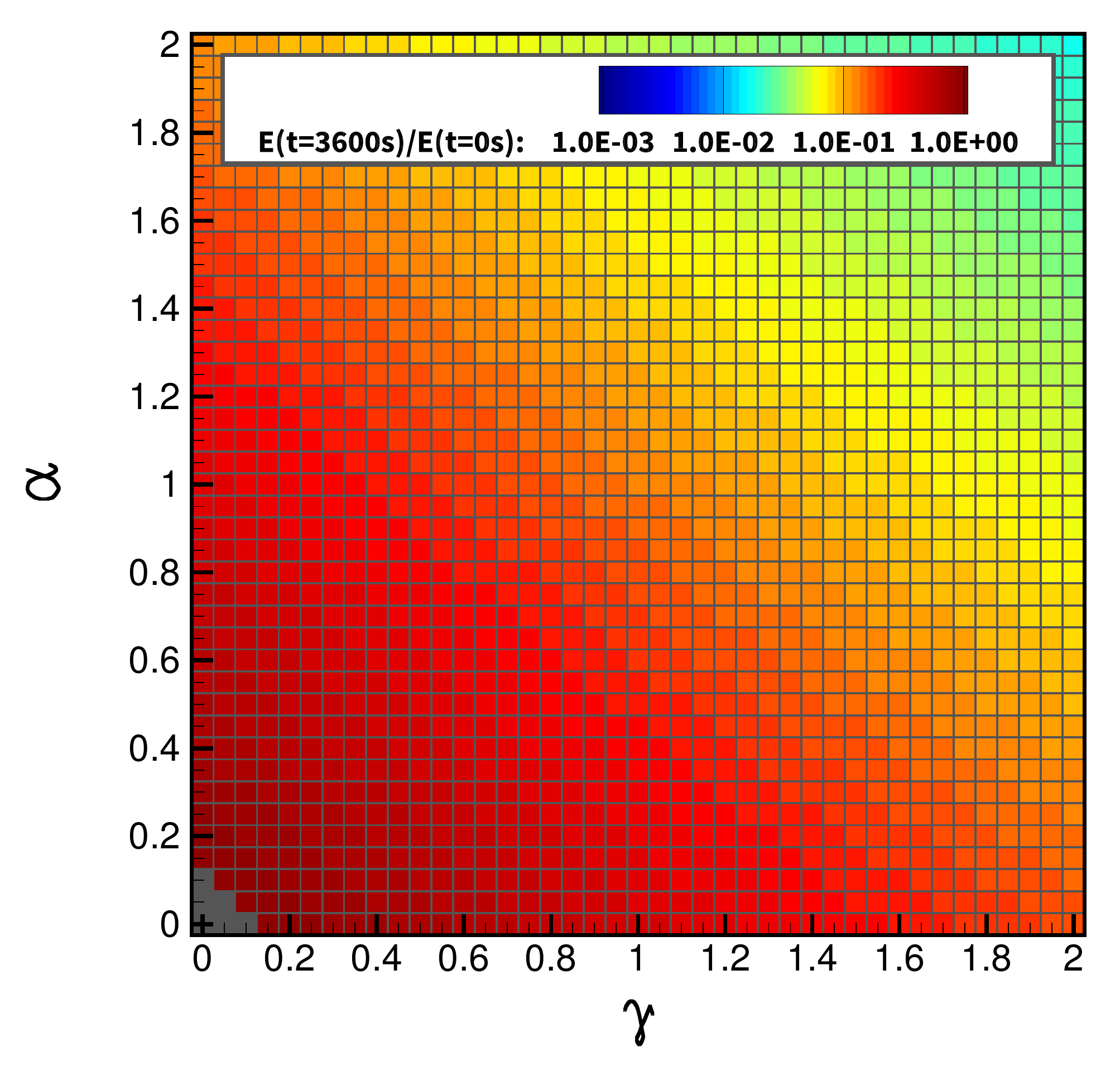}
\includegraphics[width=0.4\textwidth]{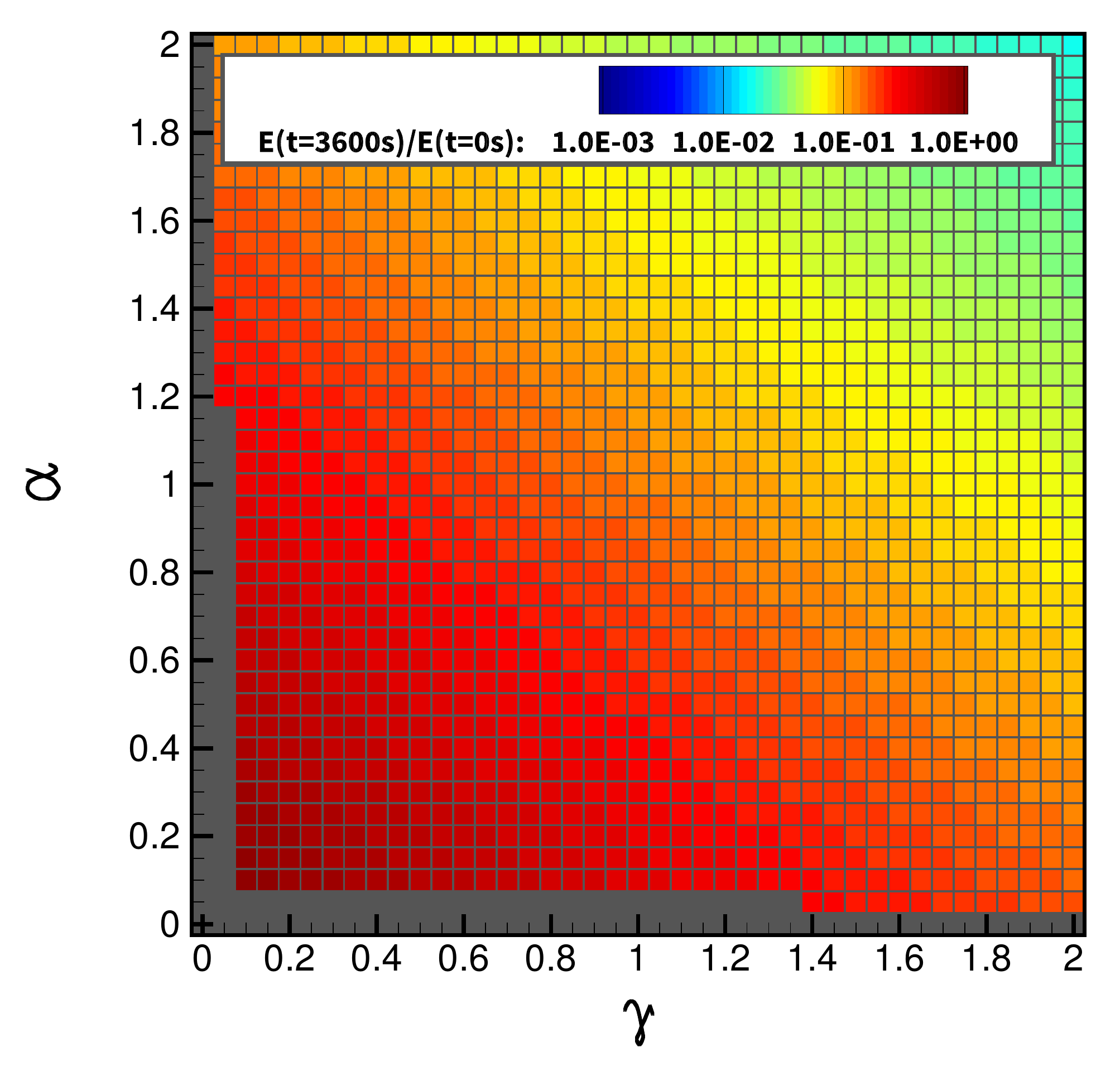} 
\par\end{centering}
\caption{Mechanical energy dissipation according to $\gamma$ and $\alpha$
with a fixed CFL number of $0.5$ in (\ref{eq:time-step}) using the
second-order scheme and a $11\times11$ mesh size; (\textit{left})
gray zone corresponds to an unstable algorithm; (\textit{right}) gray
zone corresponds to a non monotonically decreasing energy. The energy
ratio $E(t=3600\,\mathrm{s})/E(t=0\,\mathrm{s})$, computed from (\ref{eq:Discrete_Energy})
and displayed in $log$ scale, highlights the scheme's dissipation.\label{fig:energy-ratio-gamma-alpha-2}}
\end{figure}

A preliminary goal for this test case is to search numerically a range
for the two stabilization constants $\gamma$ and $\alpha$ that ensures
linear stability, and another that ensures a strict decrease of mechanical
energy, while aiming at minimizing the dissipation, for a CFL number
arbitrarily fixed at $0.5$ in (\ref{eq:time-step}). In order to
address that question, thousands of numerical simulations have been
performed with regular variations of the two constants (with a $0.05$
step), testing for each pair two stopping criteria separately during
the simulation. The first one, intented to detect a possible breaking
point in the linear stability of the scheme, is based on an \textit{a
priori} exponential growth of the mechanical energy and the second
one, more restrictive, checks if the mechanical energy decrease is
violated at each time step. These experiments were carried out using
a $41\times41$ mesh size for the first-order scheme (\ref{mass}-\ref{mom}-\ref{fe}-\ref{phiea})
and a $11\times11$ mesh size for the second-order scheme (Eqs.\ref{eq:exp-scheme-1-1}).
These mesh sizes allow to keep the same order of magnitude for the
mechanical energy diffusion. All the numerical results are summarized
for the first-order scheme in Fig.\ref{fig:energy-ratio-gamma-alpha-1}
and for the second-order scheme in Fig.\ref{fig:energy-ratio-gamma-alpha-2}.

For the first criterion, it can be clearly observed for the first-order
scheme that the sum $\gamma+\alpha$ must be greater than a minimum
value of $1$. This result is perfectly consistent with the linear
stability analysis presented in \S\ref{linear}, in the sense that
the sum ${\normalcolor \gamma+\alpha}$ governs the terms of the amplification
matrix, while the product $\gamma\alpha$ is marginal. As one could
expect, it is also found a minimum of dissipation for this minimal
sum value. Notice that one of the two coefficients can be taken to
zero and that the minimum of dissipation is reached for $\gamma=1$
and $\alpha=0$. Greater sum values introduce quickly and nearly proportionally
large amounts of dissipation. For the second-order scheme, the same
general behaviour is observed again, except that the minimum sum value
found is now $0.15$, really much lower than for the first-order case.
But in contrast with the first-order scheme, this value is correlated
to the given CFL number of $0.5$. This result may be perceived unintuitive
because MUSCL reconstructions tends to reduce the value of the stabilization
terms appearing in the mass flux and the pressure term (\ref{eq:exp-scheme-2-1})
for very regular solutions. We have verified in the linear stability
analysis that this is the Heun's method for time discretization which
mainly explains this reduction, changing profoundly the diffusion
terms nature. The increase of dissipation induced by greater sum values
is also much more limited compared to the first-order scheme.

If we now look to the second criterion, based on the mechanical energy
strict decrease, the two coefficients must be both greater than a
minimum value of $0.5$ for the first-order scheme, and a minimum
value of $0.15$ for the second-order scheme, except for too high
inefficient stabilization constants exhibiting more dissipation. This
experiment confirms an important result: the two stabilization constants
$\gamma$ and $\alpha$ are both necessary to find a strict mechanical
energy decrease.

The stability condition inequalities found for this test case of fast
gravitational waves are summarized in Tab.\ref{tab:stability-criteria-found}.
It is found optimal stabilization constants $\gamma=0.5$ and $\alpha=0.5$
for the first-order scheme and $\gamma=0.1$ and $\alpha=0.1$ for
the second-order scheme if the CFL number is fixed to $0.5$. Many
other simulations were run in other contexts, without bringing any
significant variability on these conditions.

\begin{table}[!tbh]
\begin{centering}
\begin{tabular}{>{\centering}m{0.2\textwidth}>{\centering}m{0.4\textwidth}}
\multicolumn{2}{l}{\textbf{first-order scheme }(independant of the CFL)}\tabularnewline
\midrule 
linear stability  & mechanical energy dissipation\tabularnewline
\midrule 
$\gamma+\alpha\geq1$  & $\gamma\geq0.5$ and $\alpha\geq0.5$\tabularnewline
 & \tabularnewline
\multicolumn{2}{l}{\textbf{second-order scheme }(only for a CFL number of 0.5 in (\ref{eq:time-step}))}\tabularnewline
\midrule 
linear stability  & mechanical energy dissipation\tabularnewline
\midrule 
$\gamma+\alpha\geq0.15$  & $\gamma\geq0.1$ and $\alpha\geq0.1$\tabularnewline
\end{tabular}
\par\end{centering}
\caption{Stability inequalities conditions found by a numerical experiment
of two-dimensional gravity waves for the first and second-order schemes.
The relaxed conditions obtained at second-order highlight the stabilizing
effects of the Heun's time discretization method.\label{tab:stability-criteria-found}}
\end{table}

\begin{figure}[!tbh]
\begin{centering}
\includegraphics[width=0.4\textwidth]{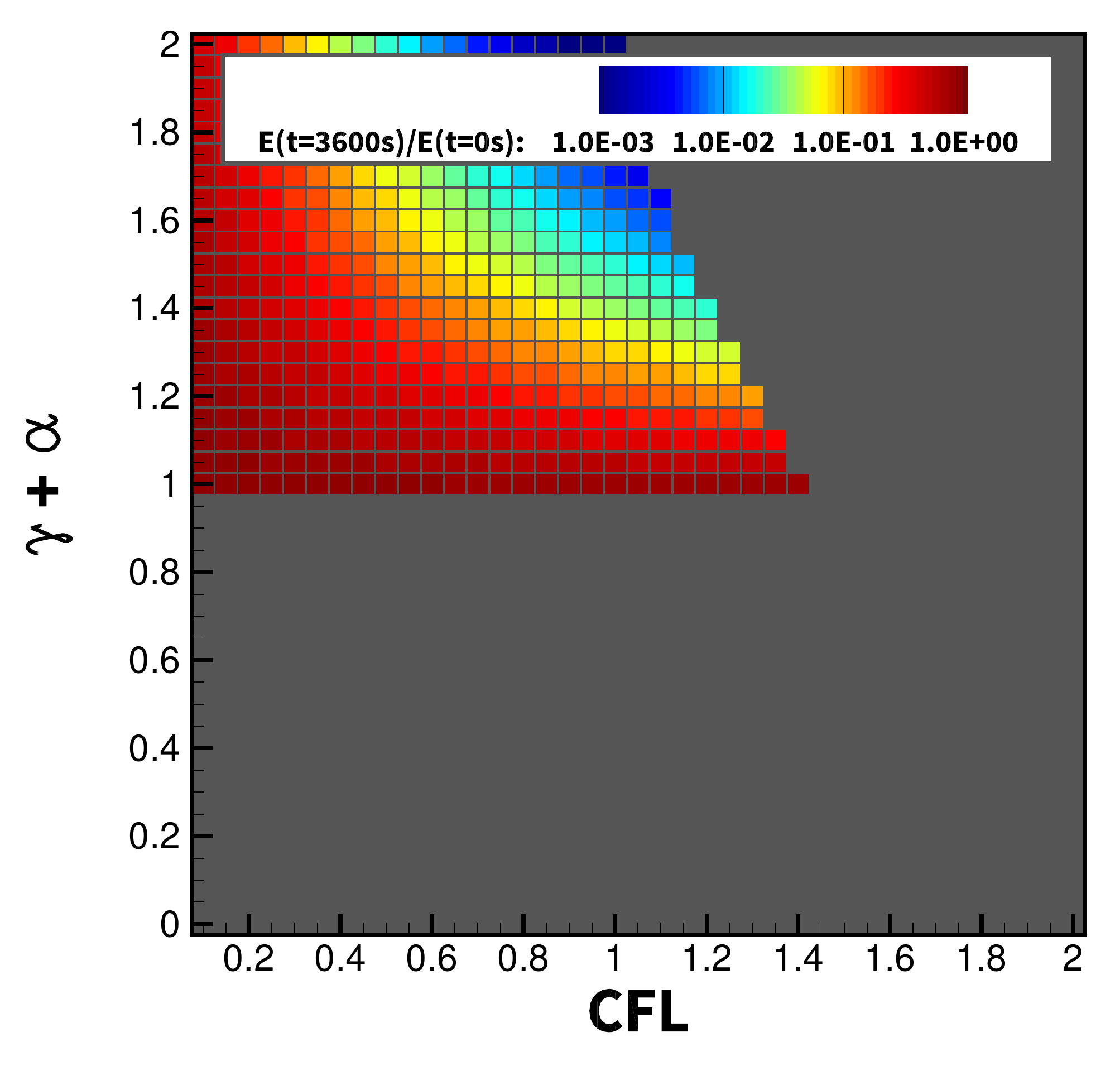}
\includegraphics[width=0.4\textwidth]{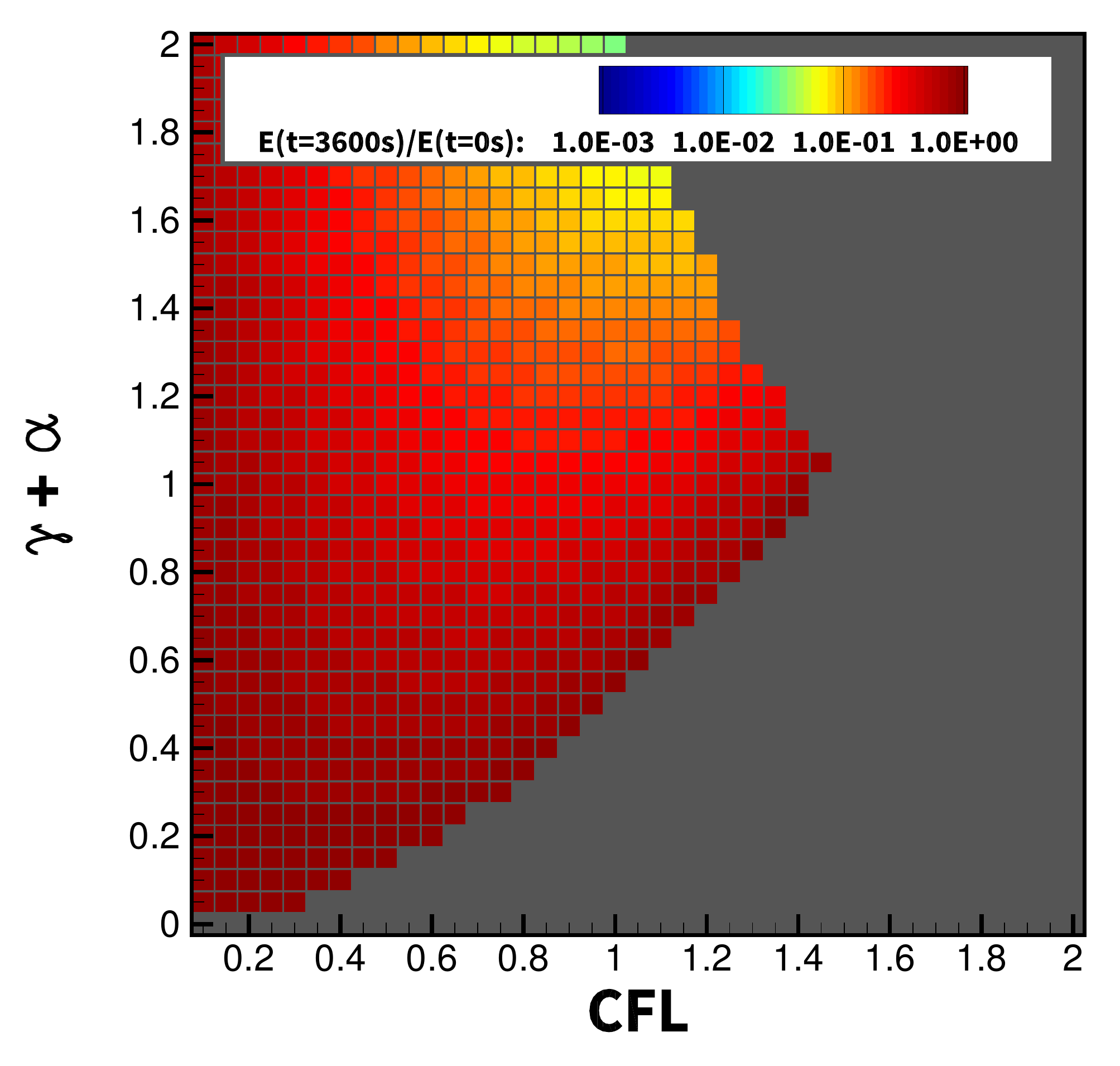} 
\par\end{centering}
\caption{Mechanical energy dissipation according to the sum $\gamma+\alpha$,
fixing the relation $\gamma=\alpha$, and the CFL number. Zones corresponding
to an unstable algorithm and non monotonically decreasing energy are
perfectly overlapping and appear in gray; (\textsl{left}) using the
first-order scheme and a $41\times41$ mesh size; (\textsl{right})
using the second-order scheme and a $11\times11$ mesh size. The energy
ratio $E(t=3600\,\mathrm{s})/E(t=0\,\mathrm{s})$, computed from (\ref{eq:Discrete_Energy})
and displayed in $log$ scale, highlights the scheme's dissipation.\label{fig:energy-ratio-gamma-cfl-1-2}}
\end{figure}

A similar experiment was performed considering varying values for
the sum $\gamma+\alpha$, fixing the relation $\gamma=\alpha$, and
CFL numbers. The numerical results are summarized in Fig.\ref{fig:energy-ratio-gamma-cfl-1-2}.
First, without any surprise, we recover the same patterns as those
from the linear stability analysis (see Figs.\ref{GK1} and \ref{AG1}
rescaling the CFL numbers). Now, an additionnal result is that the
mechanical energy is also dissipated in the domains of linear stability.
Secondly, for the first-order scheme, the dissipation is reduced considering
smaller CFL numbers for a given sum. For the second-order scheme,
the dissipation dependence with respect the CFL number is more complicated
and it is not so clear how to extract an optimal pair of stabilization
constants.

\subsubsection{Comparison with analytical solution}

\begin{figure}[!tbh]
\begin{centering}
\includegraphics[width=0.45\textwidth]{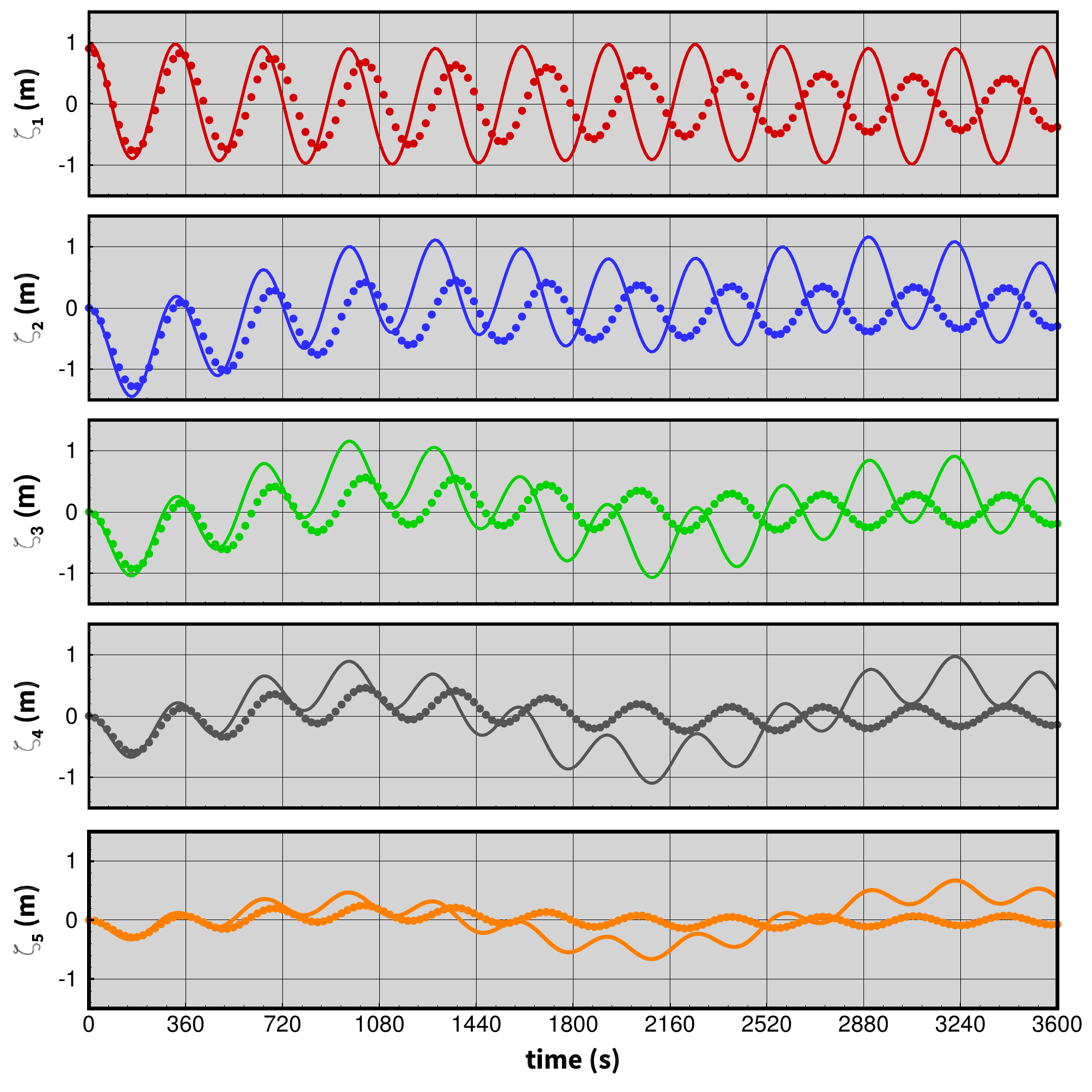}\includegraphics[width=0.45\textwidth]{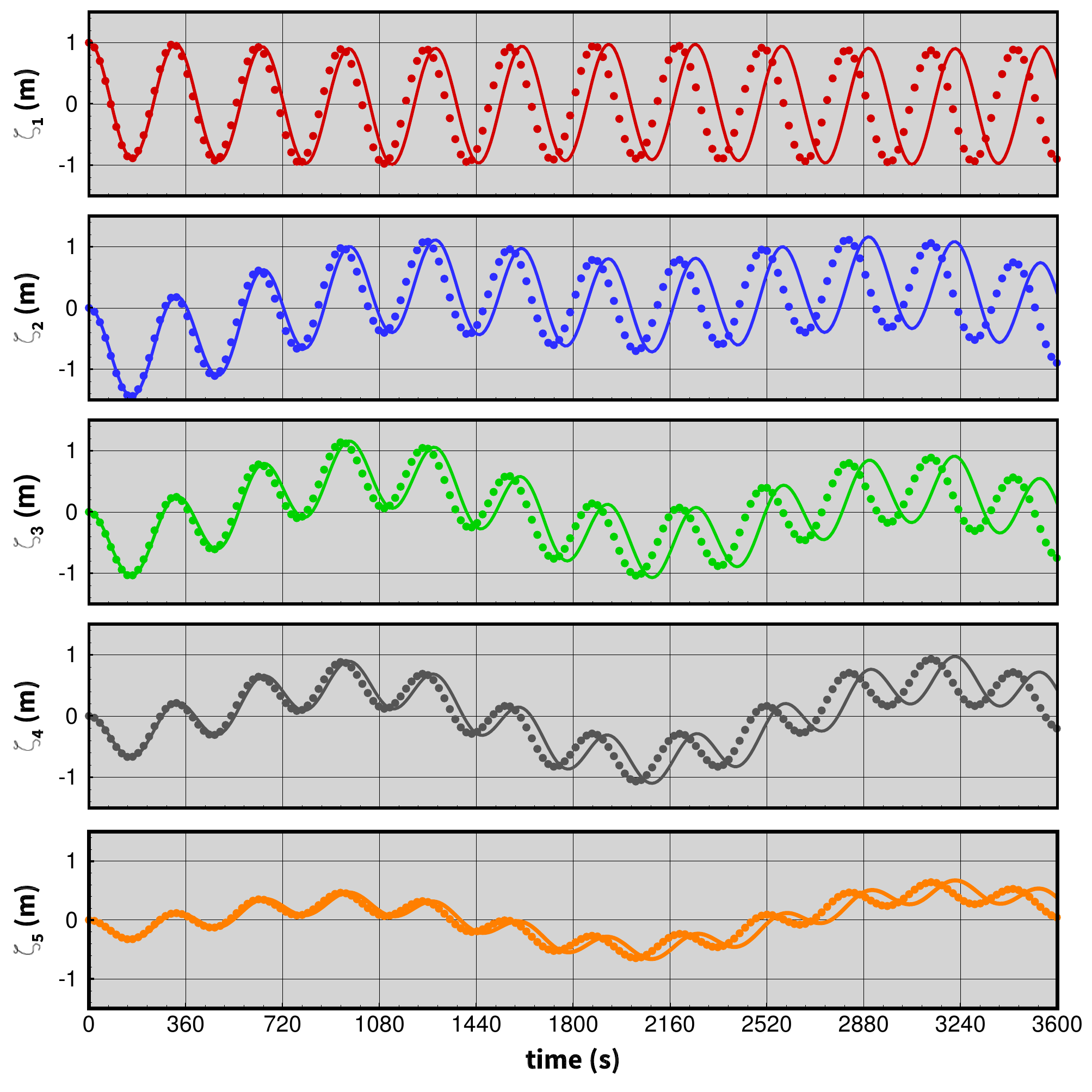} 
\par\end{centering}
\caption{Time evolution of the five surface layers deviation ($\eta_{i}=\eta_{i,0}+\zeta_{i}$)
at the box center computed with a $11\times11$ mesh size. Analytical
solution is given in continuous line and numerical solution in dotted
line; (\textit{left}) using the first-order scheme with $\gamma=\alpha=0.5$
and a CFL number of $0.5$; (\textit{right}) using the second-order
scheme with $\gamma=\alpha=0.1$ and a CFL number of $0.5$.\label{fig:grav-waves-5-layers-11x11}}
\end{figure}

\noindent In Fig.\ref{fig:grav-waves-5-layers-11x11} we propose the
time evolution of the five surface layers deviation using the first-order
scheme with $\gamma=\alpha=0.5$ (\textit{left}) and the second-order
scheme with $\gamma=\alpha=0.1$ (\textit{right}), corresponding to
the two optimal pairs found in the previous section for a $0.5$ CFL
number. The dispersive behaviour of the scheme can clearly be observed
because of the obvious phase shift, although this effect is reduced
by the second-order scheme. Nevertheless, the scheme at first and
second-order reproduces qualitatively very well the multiple interactions
between the layers in light of the $11\times11$ coarse mesh size
used. For this resolution and these stabilization constants, the second-order
scheme does exhibit a minimum of dissipation, only the dispersive
effects can be clearly distinguished.

\subsubsection{Comparison with the HLLC scheme}

We have found by a numerical experiment the optimal pairs of stabilization
constants $\gamma$ and $\alpha$ for the first and second-order schemes.
As the present method also applies to the classical shallow water
equations ($L=1$), it is interesting to illustrate the current approach
efficiency comparing it with other classical Godunov-type solvers.
From this perspective, we reduce the present test to the one layer
case, and employ the HLLC scheme, supplemented with a second-order
MUSCL reconstruction coupled with the Heun's method for time discretization.

Some numerical results are given in Fig.\ref{fig:waves-one-layer}.
As a first remark, the original HLLC scheme totally fails to capture
numerically the oscillations after a few wavelengths. There is no
more mechanical energy at the end of the simulation for the majority
of the mesh sizes considered here. An extreme level of refinement
is needed to asymptotically capture the first-order convergence. The
problem is however significantly reduced employing the second-order
extension in space and time.

With regard to the presented scheme, the results are widely better
than for the HLLC scheme at first and second-order. As a matter of
fact, the first-order scheme is already better than the second-order
HLLC scheme, while bearing in mind that the computational cost is
in addition really smaller. Note that with this level of refinement,
a third order convergence rate is reached for the first-order scheme,
except for most refined meshes, which may indicate a progressive alignment
on the right order of convergence. The present second-order scheme
does not exhibit significant losses of mechanical energy. Only a phase
shift is observed, introduced by the dispersive nature of the flow,
in the same order of magnitude than the HLLC scheme.

\begin{figure}[!tbh]
\centering{}\includegraphics[width=0.49\textwidth]{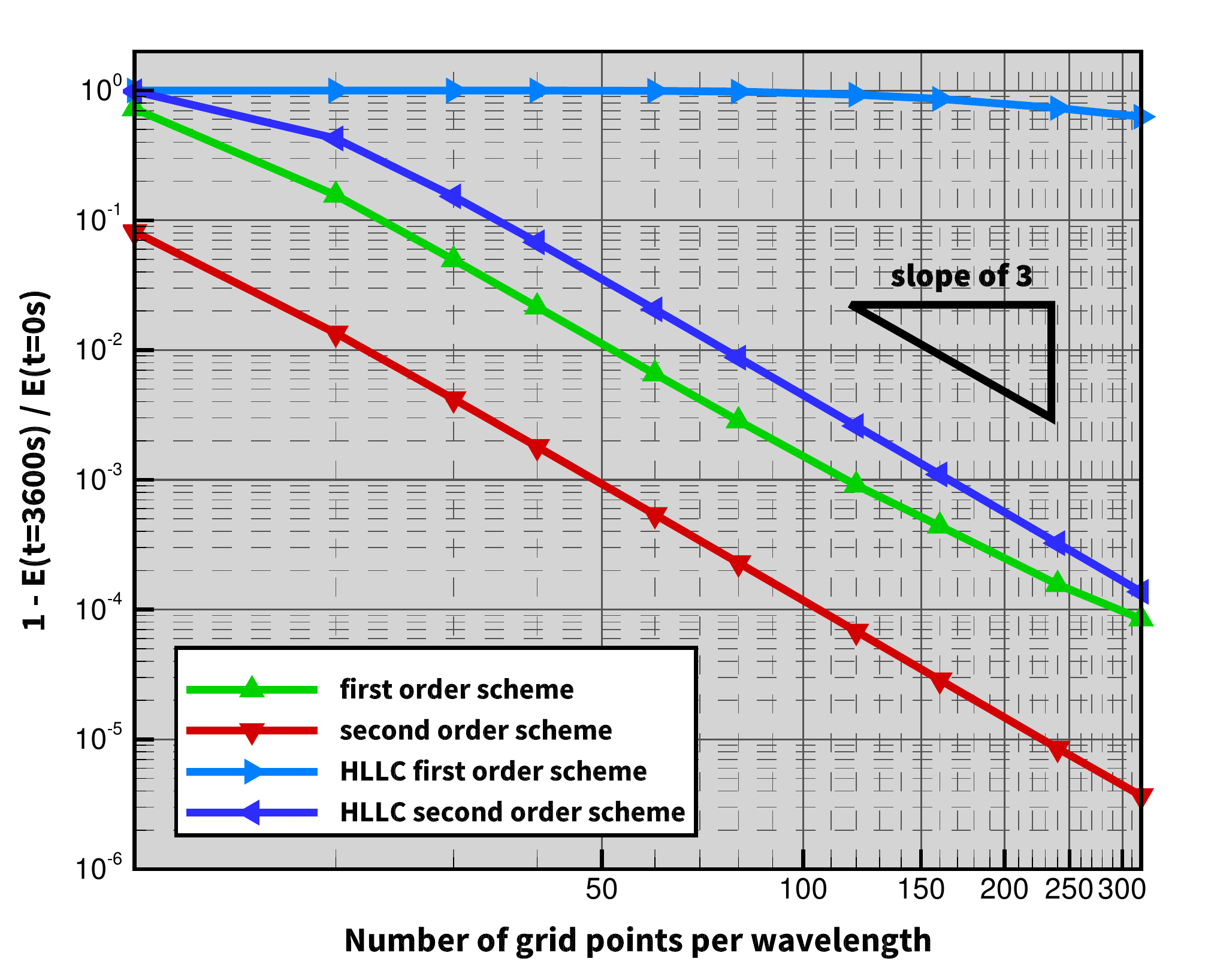} \includegraphics[width=0.49\textwidth]{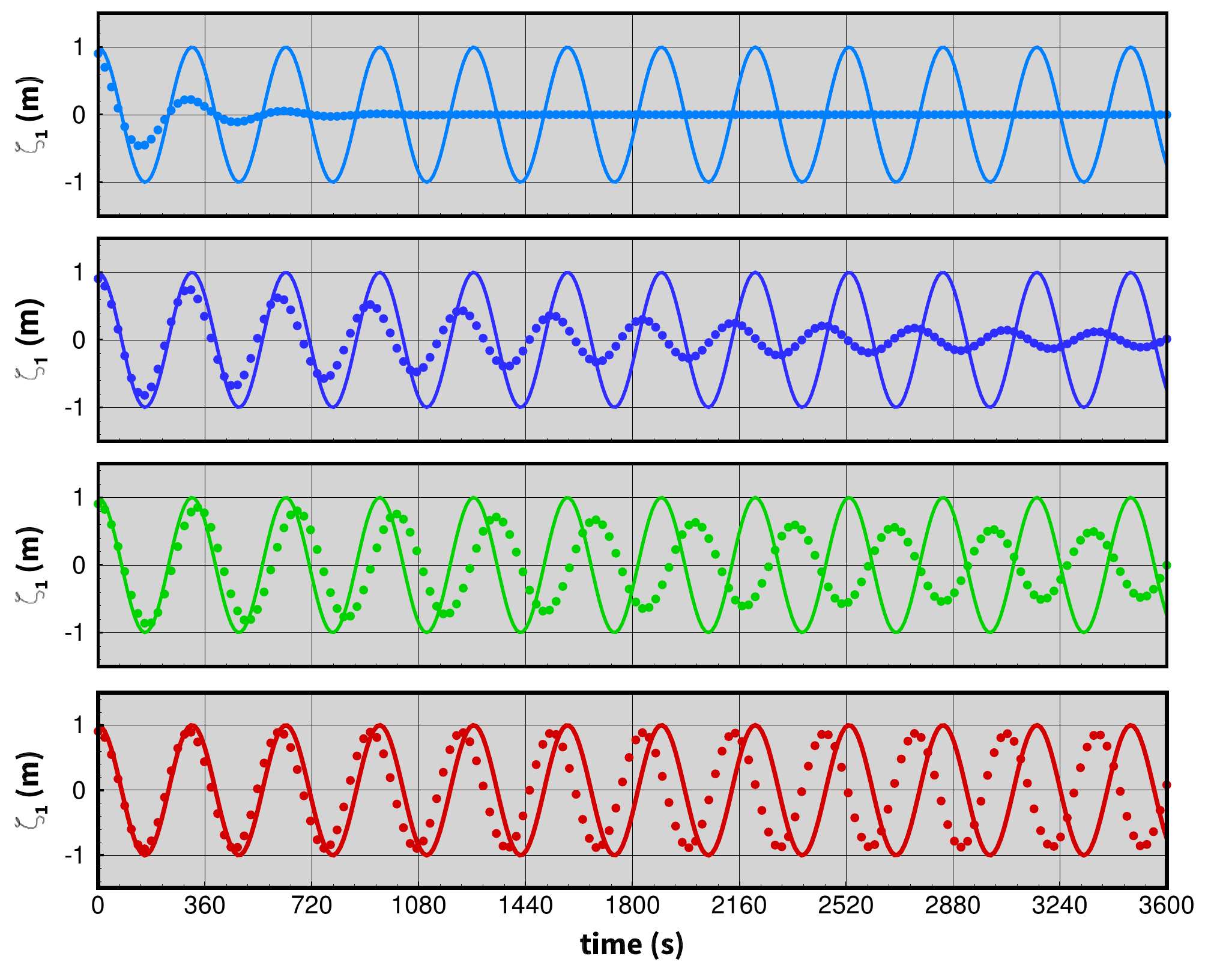}
\caption{Comparison with the HLLC scheme in the one-layer case; (\textit{left})
Mechanical energy dissipation for varying numbers of grid points per
wavelength for the present scheme and the HLLC scheme, at first and
second-order. (\textit{right}) Evolution of surface level at the box
center for a $11\times11$ mesh size with analytical solution in continuous
line and numerical solution in dotted line; \textit{(from top to bottom)}
first-order HLLC scheme, second-order HLLC scheme, present first-order
scheme with $\gamma=\alpha=0.5$ and present second-order scheme with
$\gamma=\alpha=0.1$, all with a $0.5$ CFL number.\textcolor{teal}{\label{fig:waves-one-layer}}}
\end{figure}

\subsection{Smooth surface wave propagation}

We investigate here the numerical scheme's accuracy for a smooth two-dimensional,
non-stationnary and non-linear solution. To this end, a water depth
Gaussian profile is placed in the bottom-left corner of a $500\ km$
square domain with prescribed slip boundaries:

\begin{equation}
\left\{ \begin{array}{l}
h(x,y,t=0)=h_{0}+h_{1}e^{{\displaystyle -r^{2}/2\sigma^{2}}}\\
\mathbf{u}(x,y,t=0)=\mathbf{0}
\end{array}\right.,\label{eq:pb-dambreak}
\end{equation}

\noindent where $r$ is the radial coordinate, $h_{0}=5000\ \mathrm{m}$,
$h_{1}=10\ \mathrm{m}$ and $\sigma=50\ \mathrm{km}$. We consider
a flat bottom and a gravitational acceleration $g=10\ m.s^{-2}$.
Considering a simulation time $t=600\:\mathrm{s}$, a reference solution
is generated using a $2560\times2560$ mesh and the second-order HLLC
scheme. Varying the meshes from $10^{2}$ to $320^{2}$ cells, the
$L^{2}$ absolute error norm between the numerical and reference solutions
(integrating it for each cell of the coarser mesh) are computed at
the end of the simulation. The results given in Tab.\ref{tab:dambreak}
are first showing that the expected orders of convergence are asymptotically
reached. Note the remarkable hierarchy for a given mesh size regarding
the computed error norm: the first-order HLLC scheme, the first-order
present scheme, the HLLC scheme with a Heun/MUSCL second-order extension
and the present second-order scheme with the same extension. The present
second-order scheme provides the smallest error norm independently
of the mesh size, except for the most refined cases where the result
is identical to the second-order HLLC scheme. The asymtotic convergence
to second-order is consequently more rapid for the HLLC scheme for
this test. Finally, the numerical solutions along the radial coordinate
computed with a coarse $20^{2}$ mesh for the four schemes are given
in Fig.\ref{fig:dambreak}, highlighting that our first-order method
is qualitatively as efficient as a second-order HLLC scheme.

\begin{table}[!tbh]
\begin{centering}
\begin{tabular}{>{\raggedleft}p{0.1\textwidth}>{\raggedleft}p{0.15\textwidth}>{\raggedleft}p{0.1\textwidth}>{\raggedleft}p{0.1\textwidth}>{\raggedleft}p{0.1\textwidth}>{\raggedleft}p{0.15\textwidth}>{\raggedleft}p{0.1\textwidth}}
\toprule 
$n_{x}\times n_{y}$ & $\epsilon_{L_{2}}$ & order &  & $n_{x}\times n_{y}$ & $\epsilon_{L_{2}}$ & order\tabularnewline
\midrule 
 &  &  &  &  &  & \tabularnewline
\multicolumn{3}{l}{HLLC first-order scheme} &  & \multicolumn{3}{l}{present first-order scheme ($\alpha=\gamma=0.5$)}\tabularnewline
\cmidrule{1-3} \cmidrule{5-7} 
$10^{2}$ & $3.18\:10^{-1}$ & - &  & $10^{2}$ & $2.25\:10^{-1}$ & -\tabularnewline
$20^{2}$ & $2.27\:10^{-1}$ & $0.49$ &  & $20^{2}$ & $1.11\:10^{-1}$ & $1.02$\tabularnewline
$40^{2}$ & $1.42\:10^{-1}$ & $0.68$ &  & $40^{2}$ & $3.76\:10^{-2}$ & $1.56$\tabularnewline
$80^{2}$ & $8.07\:10^{-2}$ & $0.82$ &  & $80^{2}$ & $1.42\:10^{-2}$ & $1.40$\tabularnewline
$160^{2}$ & $4.34\:10^{-2}$ & $0.90$ &  & $160^{2}$ & $6.25\:10^{-3}$ & $1.18$\tabularnewline
$320^{2}$ & $2.26\:10^{-2}$ & $0.94$ &  & $320^{2}$ & $2.99\:10^{-3}$ & $1.06$\tabularnewline
 &  &  &  &  &  & \tabularnewline
\multicolumn{3}{l}{HLLC second-order scheme} &  & \multicolumn{3}{l}{present second-order scheme ($\alpha=\gamma=0.1$)}\tabularnewline
\cmidrule{1-3} \cmidrule{5-7} 
$10^{2}$ & $1.69\:10^{-1}$ & - &  & $10^{2}$ & $1.16\:10^{-1}$ & -\tabularnewline
$20^{2}$ & $6.64\:10^{-2}$ & $1.35$ &  & $20^{2}$ & $4.70\:10^{-2}$ & $1.30$\tabularnewline
$40^{2}$ & $1.87\:10^{-2}$ & $1.83$ &  & $40^{2}$ & $1.72\:10^{-2}$ & $1.45$\tabularnewline
$80^{2}$ & $4.78\:10^{-3}$ & $1.97$ &  & $80^{2}$ & $4.67\:10^{-3}$ & $1.87$\tabularnewline
$160^{2}$ & $1.21\:10^{-3}$ & $1.98$ &  & $160^{2}$ & $1.21\:10^{-3}$ & $1.96$\tabularnewline
$320^{2}$ & $2.99\:10^{-4}$ & $2.02$ &  & $320^{2}$ & $3.00\:10^{-4}$ & $2.01$\tabularnewline
 &  &  &  &  &  & \tabularnewline
\bottomrule
\end{tabular}
\par\end{centering}
\caption{Numerical convergence results for the radial smooth surface wave propagation.
The errors $\epsilon_{L^{2}}$ refer to the absolute $L^{2}$ norm
between the computed numerical solution obtained with a $n_{x}\times n_{y}$
mesh size and the reference solution computed with a $2560\times2560$
mesh size.\label{tab:dambreak}}
\end{table}

\begin{figure}[!tbh]
\begin{centering}
\includegraphics[width=0.6\textwidth]{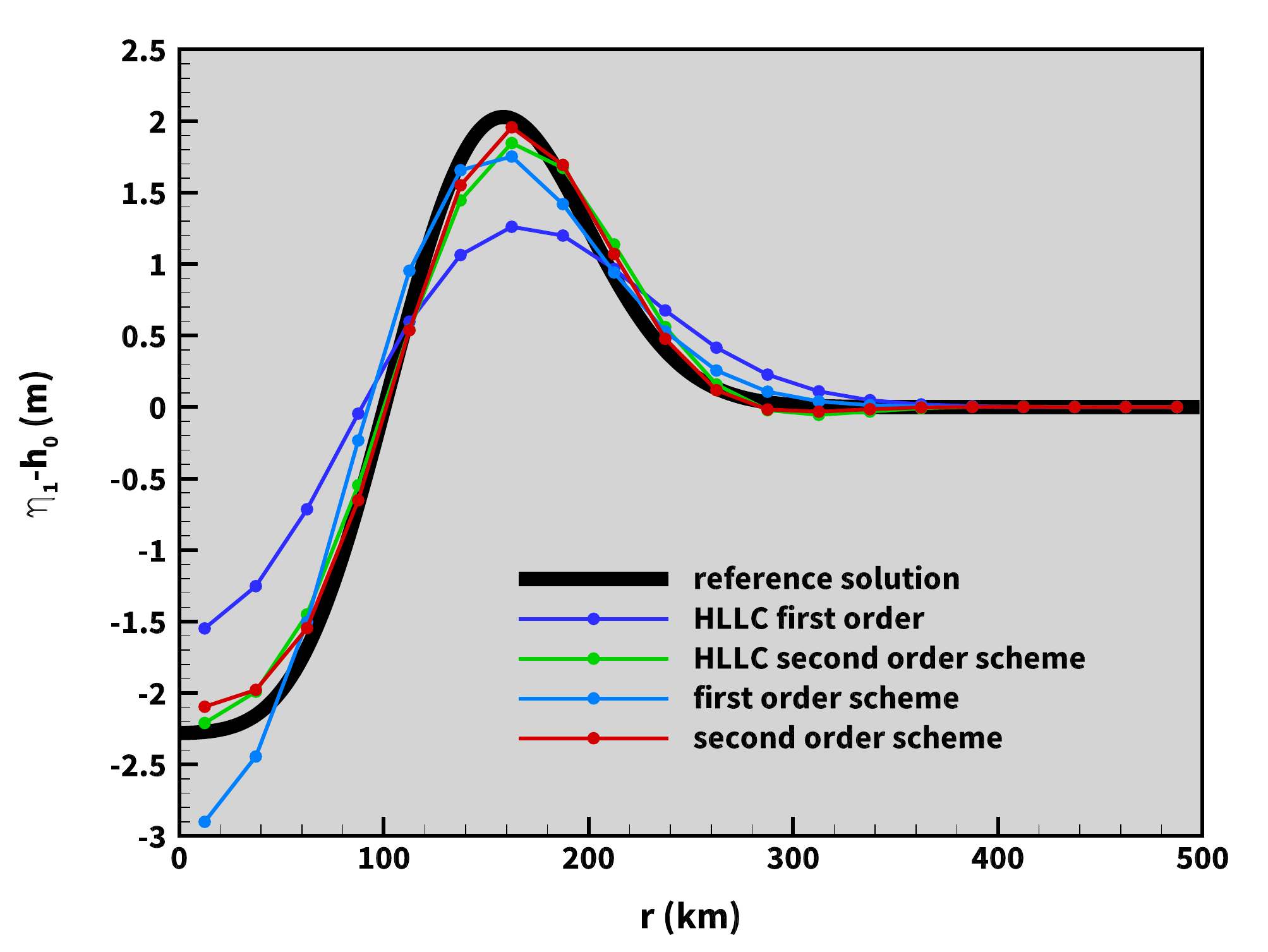}
\par\end{centering}
\caption{Numerical results for the radial smooth surface wave propagation obtained
at time $t=600\:\mathrm{s}$ along one radial axis computed with a
$20\times20$ mesh size. \label{fig:dambreak}}
\end{figure}

\subsection{Small perturbation of a lake at rest\label{subsec:lake}}

This test case proposed in \citep{Leveque1998} and reproduced for
example in \citep{Noelle2006}, \citep{Ricchiuto2013} and \citep{Tavelli2014}
is intended to check the scheme's ability to deal both with the well-balanced
property and the propagation of a jump in the initial water surface
elevation. It should be recalled that the first-order scheme is well-balanced
by construction and that this property easily extends to the second-order
MUSCL reconstruction scheme, as it has been discussed in \S \ref{well-balancing}.

This test involves a two-dimensional rectangular computational domain
$\left[0,2\right]\times\left[0,1\right]$ and a non linear topography
at the bottom:

\begin{equation}
z_{b}=0.8\:e^{{\displaystyle \left(-5\left(x-0.9\right)^{2}-50\left(y-0.5\right)^{2}\right)}}.\label{eq:lake-topo}
\end{equation}

First, considering an initial motionless constant water surface elevation
$\eta_{1}=1$, the solution should stay at rest. At first and second-order
in space, it is found that whatever the simulation time is, the water
surface elevation and velocity error norms are exactly zero because
of the exact flux balance with respect to the discrete potential.

Next, following the original test case in \citep{Leveque1998}, we
consider a jump of water surface elevation:

\begin{equation}
\eta_{1}(x,y,t=0)=\left\{ \begin{array}{ll}
1.01 & \text{if }0.05\leq x\leq0.15\\
1 & \text{otherwise}
\end{array}\right..\label{eq:lake-ini}
\end{equation}

Slip boundaries are prescribed except an idealized outflow at western
boundary, considering an extended computational domain to avoid any
reflexion, as the initial water surface bump generates left- and right-going
waves. Considering a simulation time $t=0.46\:s$, some snapshots
are given in Fig.\ref{fig:lake_results} for the present scheme, at
first and second-order for a relatively coarse $300\times100$ mesh
(\textit{bottom}). Using the same resolution, these results can be
compared with the second order HLLC scheme, and an highly resolved
solution, serving as reference (\textit{top}). A Barth limiter \citep{Barth2004}
has been used for the reconstructed water surface elevation to prevent
from too much dispersive solutions. Our scheme is reproducing qualitatively
very well the complex flow dynamics. However, notably due to the discontinuous
nature of the initial solution, the stabilization constants must be
taken higher than the optimal ones found in the previous test case
(Tab.\ref{tab:stability-criteria-found}) to avoid spurious oscillations.
Cross sections of the final solution are displayed in Fig.\ref{fig:lake_results-2},
showing again a low level of numerical diffusion in comparison with
the classical HLLC scheme. In conclusion, our scheme can be succesfully
employed for this kind of complex flows, implying an initial jump
and non trivial topography. These observations also tend to indicate
that the scheme's efficiency can be significantly improved with an
adjustment of the constants $\gamma$ and $\alpha$, according to
the local regularity of the discrete solution. Additional theoretical
and numerical investigations are currently in progress in that direction.

\begin{figure}[!tbh]
\begin{centering}
\includegraphics[width=0.49\textwidth]{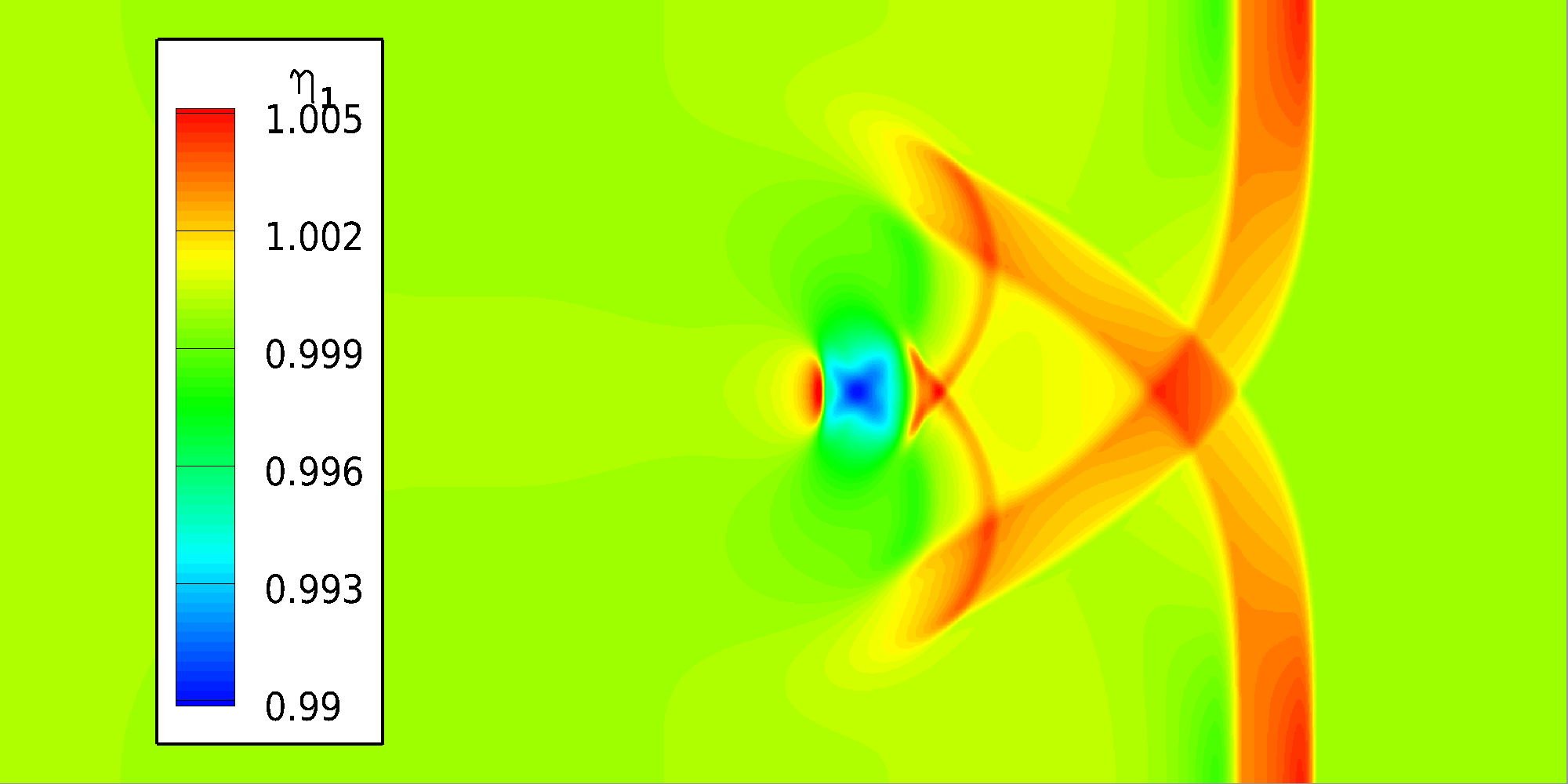}\includegraphics[width=0.49\textwidth]{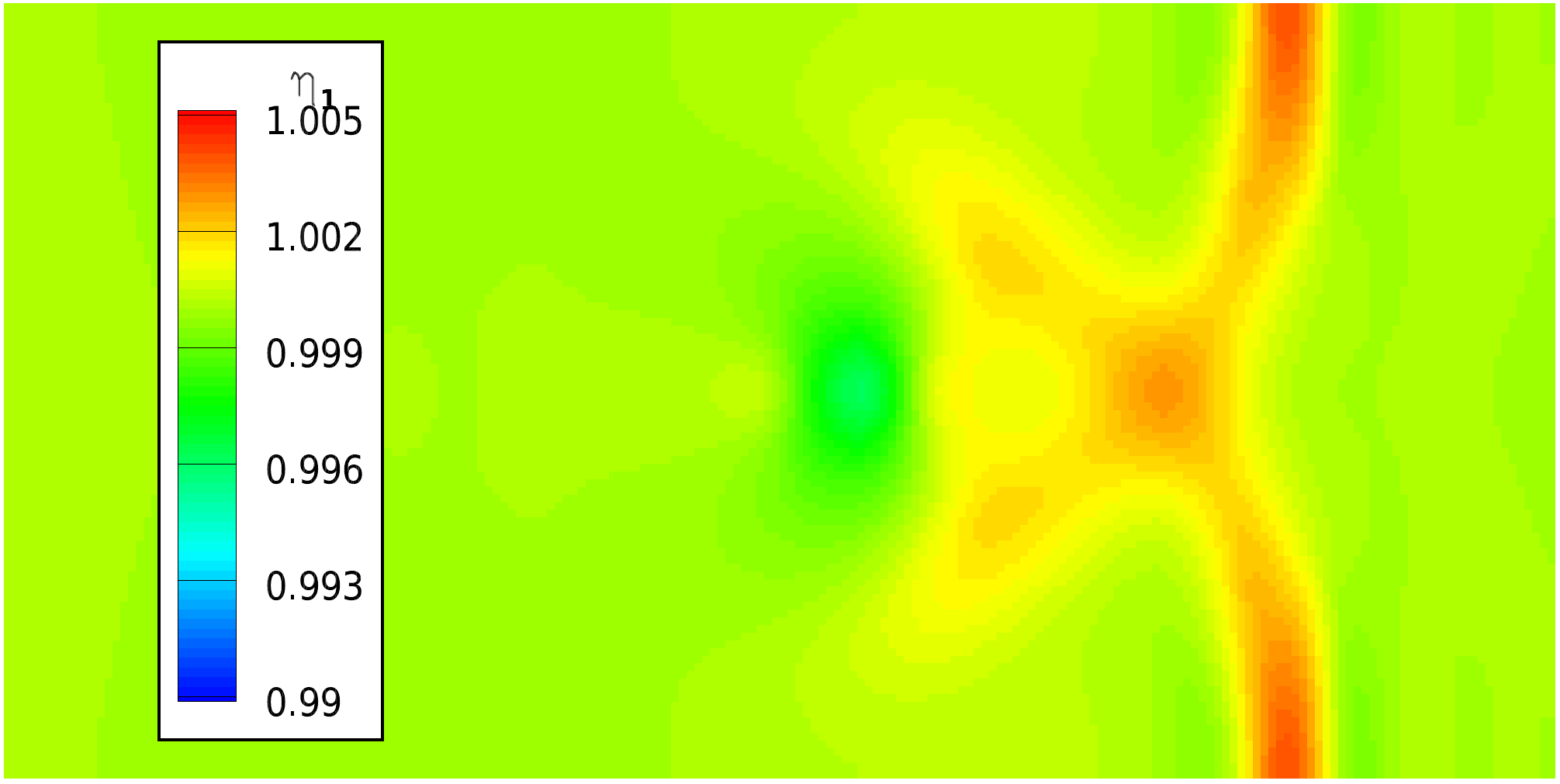}
\par\end{centering}
\begin{centering}
\includegraphics[width=0.49\textwidth]{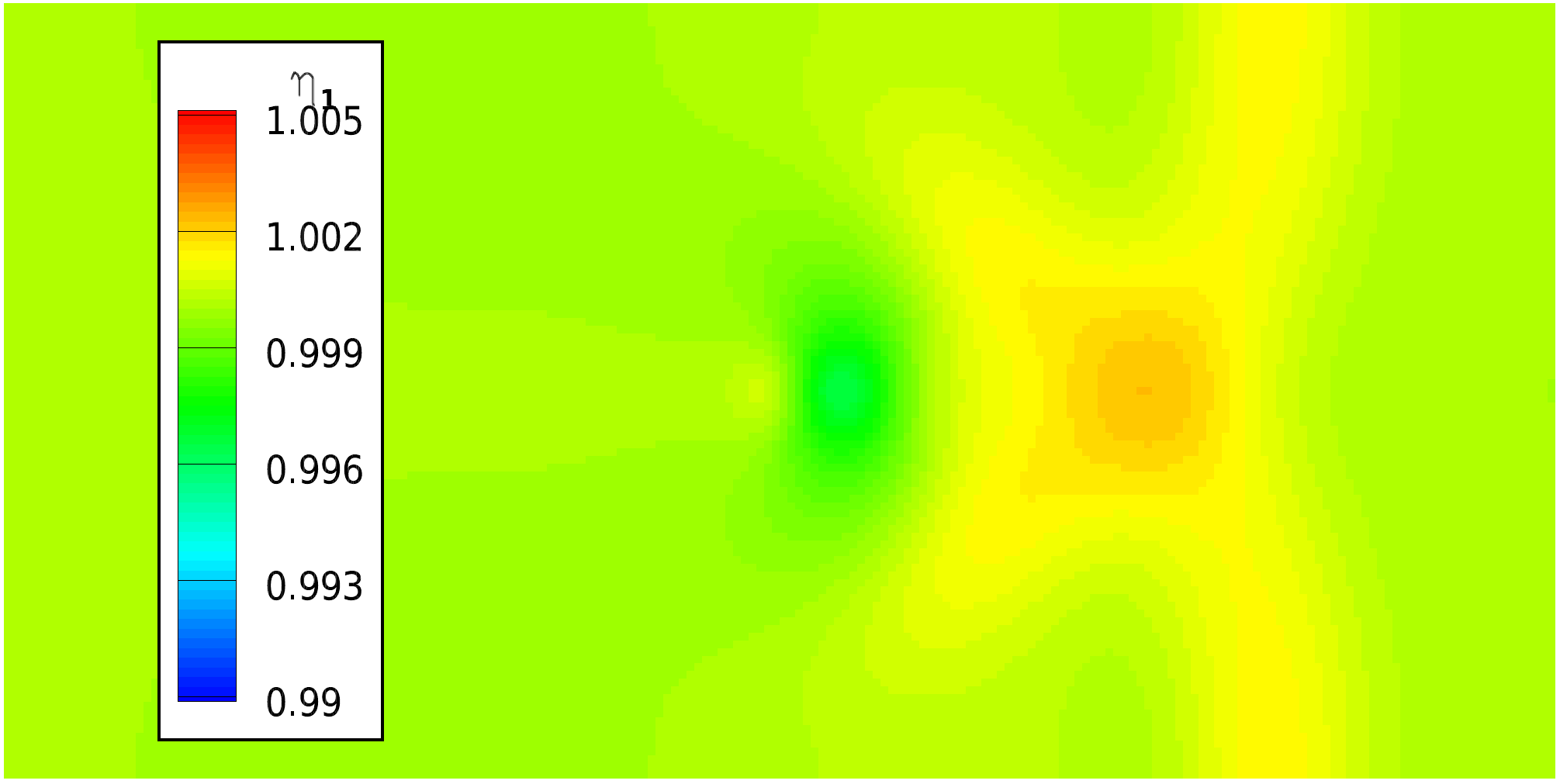}\includegraphics[width=0.49\textwidth]{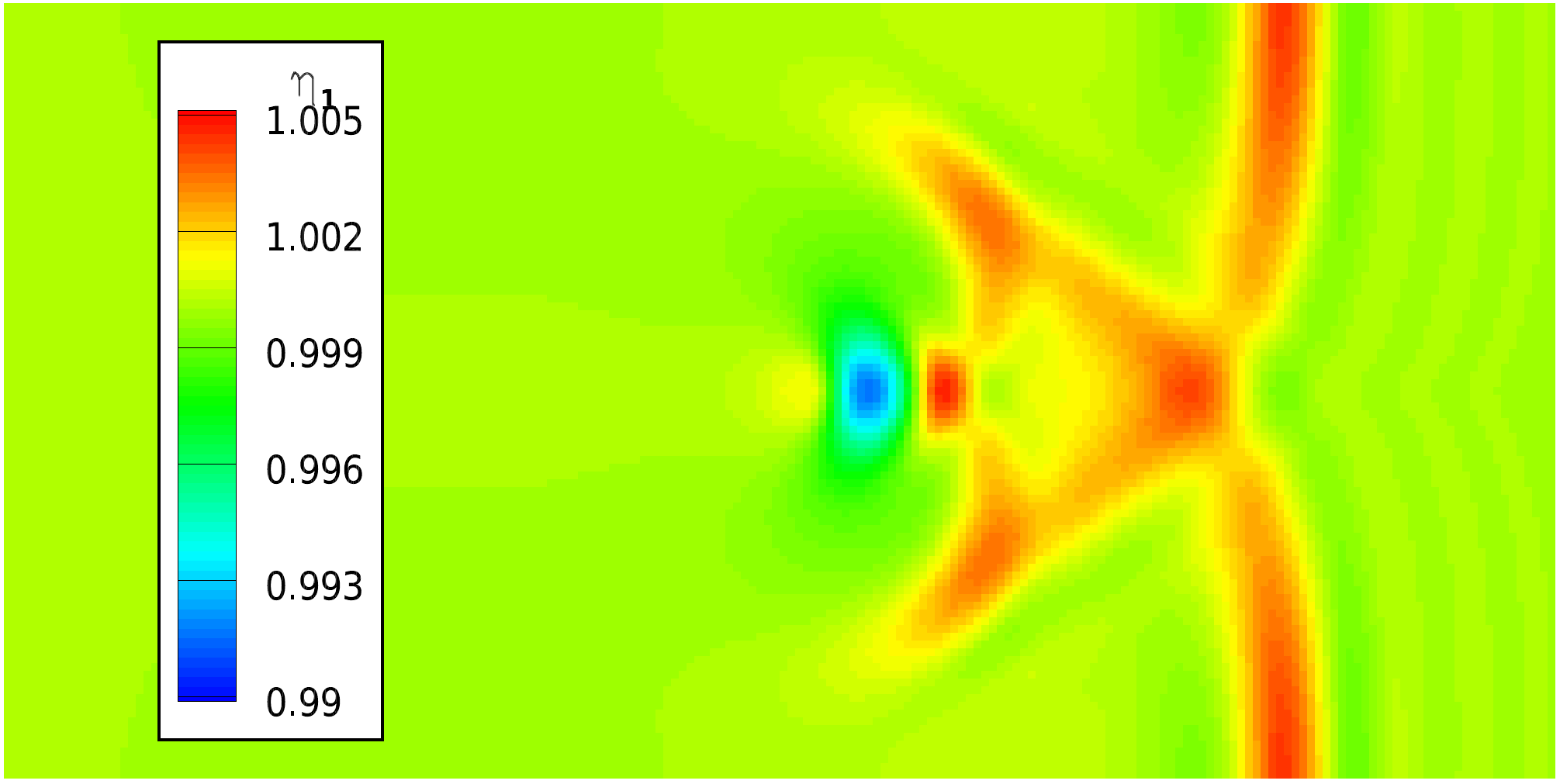}
\par\end{centering}
\caption{Numerical simulations of a propagating wave considering the bottom
topography (\ref{eq:lake-topo}) and the initial condtion (\ref{eq:lake-ini});
water surface top view at time $t=0.46\:s$; (\textit{top}) using
the second-order HLLC scheme and a $1600\times800$ mesh size (\textit{left}),
and a $300\times100$ mesh size \textit{(right}); (\textit{bottom})
using a $300\times100$ mesh size and the present first-order scheme
with $\gamma=\alpha=1$ \textit{(left)} and the present second-order
scheme with $\gamma=\alpha=0.5$ \textit{(right)}.\label{fig:lake_results}}
\end{figure}

\begin{figure}[!tbh]
\begin{centering}
\includegraphics[width=0.6\textwidth]{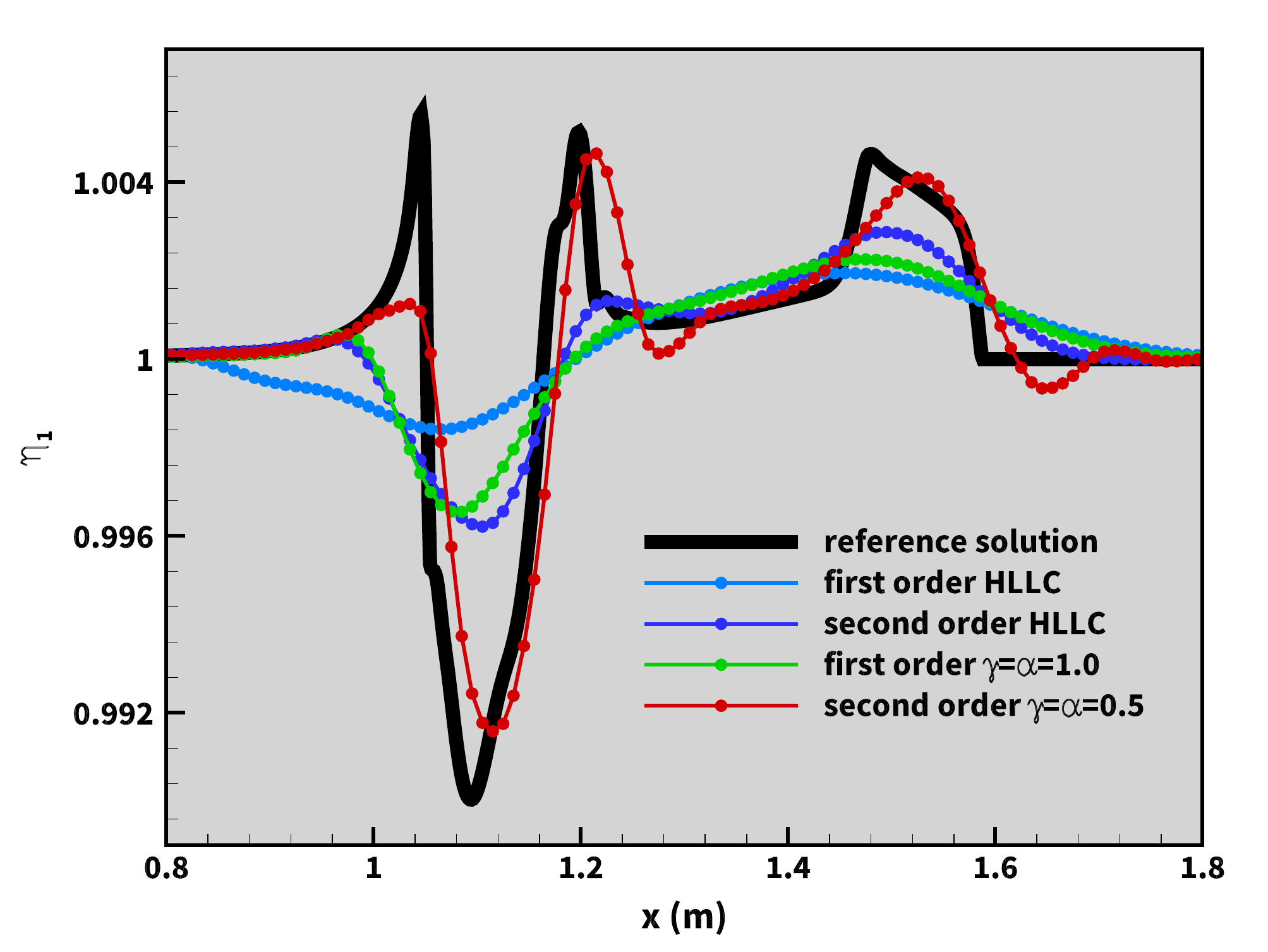}
\par\end{centering}
\caption{Slides view corresponding to the simulations in Fig.\ref{fig:lake_results}
along the horizontal axis at the middle of the domain, at time $t=0.46\:s$
and for a $300\times100$ mesh. A Barth limiter \citep{Barth2004}
has been used for the MUSCL reconstructed water surface elevation.\label{fig:lake_results-2}}
\end{figure}

\subsection{Baroclinic vortex\label{subsec:baroclinic-vortex}}

Based on the COMODO benchmark \citep{COMODO}, we study here an idealized
axisymmetric and anticyclonic baroclinic vortex initially centered,
propagating south-westward due to a $\beta$-plane approximation,
following the numerical experiment proposed in \citep{Penven2006}.
The vortex is expected to approximately retains its axisymmetric shape
with a progressive decrease of energy along its trajectory, mainly
in the wave of emissions of weak-amplitude Rossby-waves. This last
test case represents a good indicator of the scheme's accuracy in
the frame of a complex flow with several layers, the principal difficulty
lying in the capability to describe accurately the vortex motion.
Indeed, the numerical diffusion and dispersion induced by unsuitable
schemes can quickly break the cyclostrophic balance and subsequently
deteriorate the vortex trajectory.

\subsubsection{Initialization}

A vortex is placed at the center of the box $\left[-900\ \mathrm{km},900\ \mathrm{km}\right]^{2}$
with boundary walls according to an axisymmetric Gaussian pressure
profile:

\begin{equation}
\begin{array}{l}
\eta_{1}={\displaystyle \frac{P_{0}}{g\rho_{0}}}e^{\nicefrac{-r^{2}}{2\lambda^{2}}}\end{array}\,,
\end{equation}

\noindent where $\rho_{0}=1024.4\ \mathrm{kg}.\mathrm{m}^{-3}$ is
the density at sea surface, $g=9.81\ \mathrm{m}.\mathrm{s}^{-2}$
is the gravitational acceleration, $\lambda=60\ \mathrm{km}$ and
$P_{0}=\rho_{0}f_{0}u_{max}\lambda\sqrt{e}$ is a pressure defined
from a maximum velocity $u_{max}=0.8\ \mathrm{m}.\mathrm{s}^{-1}$,
giving an anticyclonic vortex. In each layer $i$, the vortex at cyclostrophic
equilibrium respects an axisymmetric balance between centripetal acceleration
$v_{i,\theta}$, pressure $p_{i}$ and Coriolis force:

\begin{equation}
-{\displaystyle \frac{v_{i,\theta}^{2}}{r}}-fv_{i,\theta}+{\displaystyle \frac{dp_{i}}{dr}}=0\,.
\end{equation}

Eliminating the unphysical solution, we obtain the final velocity
expression in each layer as a function of the layer pressure gradient
in cylindrical coordinates:

\begin{equation}
v_{i,\theta}=-{\displaystyle \frac{fr}{2}}\left(1-\sqrt{1+{\displaystyle \frac{4{\displaystyle \frac{dp_{i}}{dr}}}{rf^{2}}}}\right)\,.\label{eq:cyclo_vel}
\end{equation}
Note that with simulations initialized with a velocity at geostrophic
equilibrium $v_{i,\theta}={\displaystyle -\frac{1}{f}\frac{dp_{i}}{dr}}$
as prescribed in the original test case \citep{COMODO}, numerical
approximations generates too much undesirable small scale waves because
of the initial imbalance at the continuous level (ignoring the $\beta$-plane
approximation). As their wavelength decreases with the mesh size,
and as we have seen that our scheme does not dissipate high frequencies,
it involves an improper convergence. As a consequence, the $u_{max}$
considered here is smaller than in the original test case to ensure
the positivity of the term in the square root in (\ref{eq:cyclo_vel}).

A $\beta$-plane approximation is made for the Coriolis force:

\begin{equation}
\begin{array}{l}
f=f_{0}+\beta y\end{array}\,,
\end{equation}

\noindent with a latitude $\theta=38.5\text{\textdegree}$, giving
the two constants $f_{0}=2\Omega\sin\left(\theta\right)\simeq9,054\ 10^{-5}$
and $\beta=2\Omega\cos\left(\theta\right)/R_{earth}\simeq1,788\ 10^{-11}$.
The density distribution involves ten layers at rest, evenly sized,
following the linear law:

\begin{equation}
\begin{array}{l}
\rho_{i}=\rho_{0}\left(1-{\displaystyle \frac{N^{2}}{g}z_{i}}\right)\end{array}\quad\text{with}\quad z_{i}={\displaystyle \frac{h_{0}\left(i-\frac{1}{2}\right)}{N}}\,,
\end{equation}

\noindent where $N=3.10^{-3}\ \mathrm{s}^{-1}$ is the Brunt-Väisälä
frequency and $h_{0}=5000\ \mathrm{m}$ is the unperturbed sea surface
height. No motion is prescribed under a level $h_{1}=2500\ \mathrm{m}$
in order to prevent from fast barotropic modes as prescribed in the
original test case \citep{COMODO}. It is derived here a formal way
to nullify the velocity starting from the 6th layer. For a $L$ layers
system, the potential in the layer $i$ can be written:

\begin{equation}
\Phi_{i}={\displaystyle \frac{g}{\rho_{i}}}\left({\displaystyle \rho_{1}\eta_{1}+\sum_{k=2}^{i}\left(\rho_{k}-\rho_{k-1}\right)\eta_{k}}\right)\,.
\end{equation}

If we suppose $\nabla\Phi_{i}=0$, it can be found that:

\begin{equation}
\rho_{1}\nabla\eta_{1}+{\displaystyle \sum_{k=2}^{i}\left(\rho_{k}-\rho_{k-1}\right)\nabla\eta_{k}}=0\,.
\end{equation}

If we suppose in addition that $\forall k>i\:,\:\nabla\eta_{k}=0$,
then we find also $\forall k>i\:,\:\nabla\Phi_{k}=0$. Suppose now
that $\nabla\eta_{k+1}=r\nabla\eta_{k}$, then:

\begin{equation}
\rho_{1}\nabla\eta_{1}+\nabla\eta_{2}{\displaystyle \sum_{k=2}^{i}\left(\rho_{k}-\rho_{k-1}\right)r^{k-2}}=0\,,
\end{equation}

giving the final expression for the water surface elevation gradient:

\begin{equation}
\nabla\eta_{i}=\frac{-\rho_{1}r^{i-2}\nabla\eta_{1}}{{\displaystyle \sum_{k=2}^{i}\left(\rho_{k}-\rho_{k-1}\right)r^{k-2}}}\,.
\end{equation}

from which we extract the final water surface elevation distribution
with $r=1$ adding the layer level at rest as a constant. Let us notice
the inverse sign of the internal layer gradients compared to the sea
surface gradient $\nabla\eta_{1}$ (since $\rho_{k}>\rho_{k-1}$),
implying a pressure gradient decrease. Finally, we recall that we
do not consider any viscosity or bottom friction effects in this test
case. Note finally that $\epsilon\thickapprox3.6\:10^{-3}$ for this
test case.

\subsubsection{Simulations}

\begin{figure}[!tbh]
\begin{centering}
\includegraphics[width=0.45\textwidth]{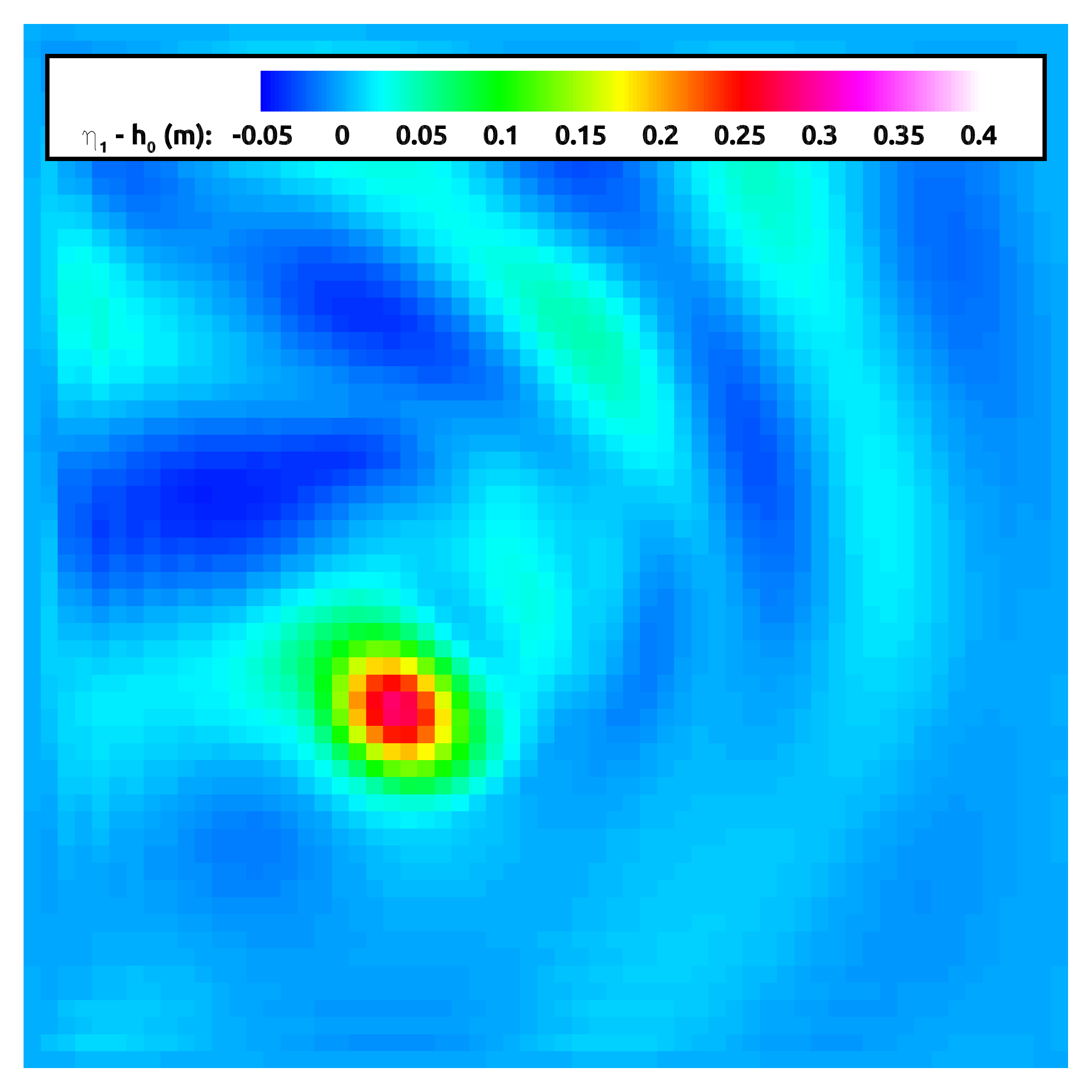}\includegraphics[width=0.45\textwidth]{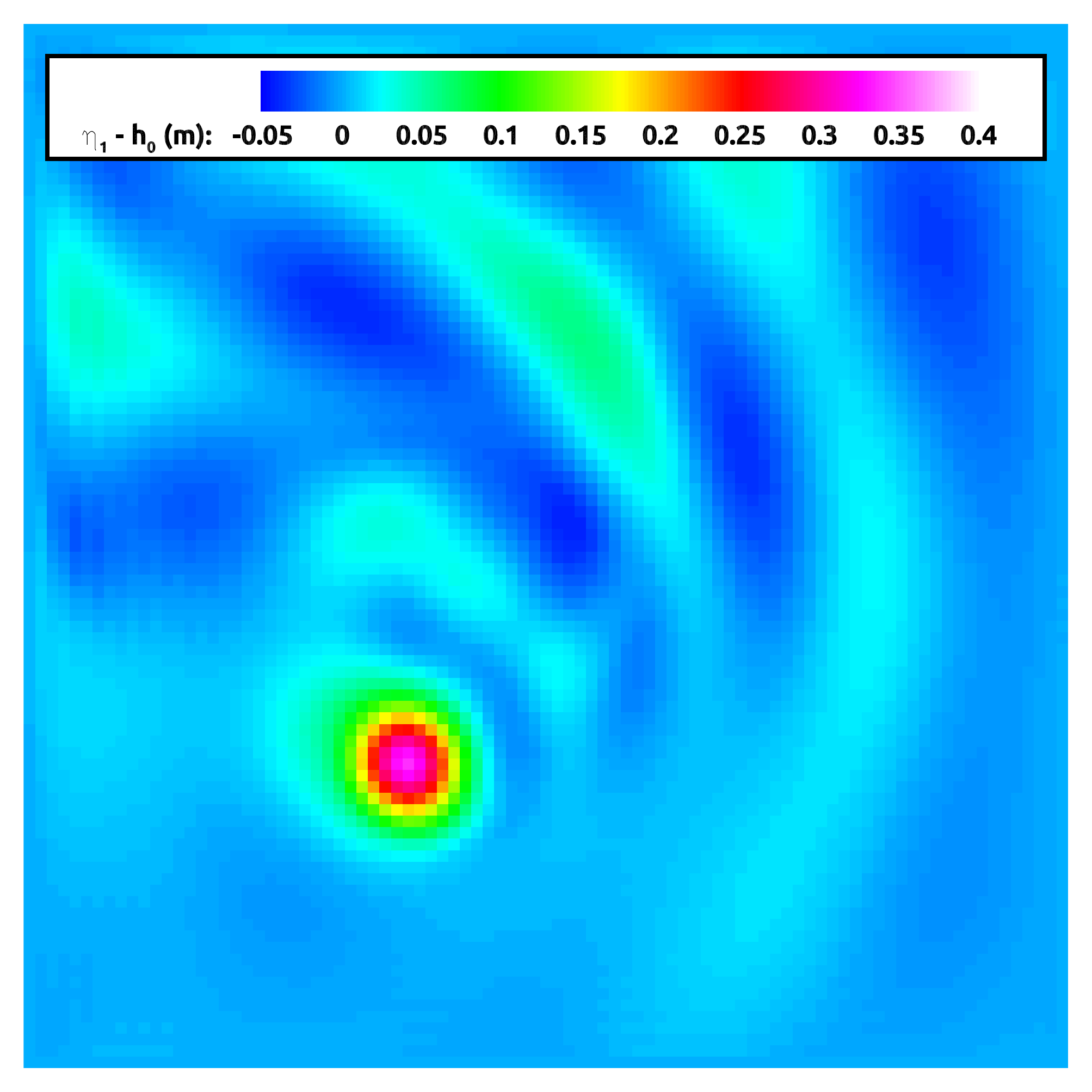} 
\par\end{centering}
\begin{centering}
\includegraphics[width=0.45\textwidth]{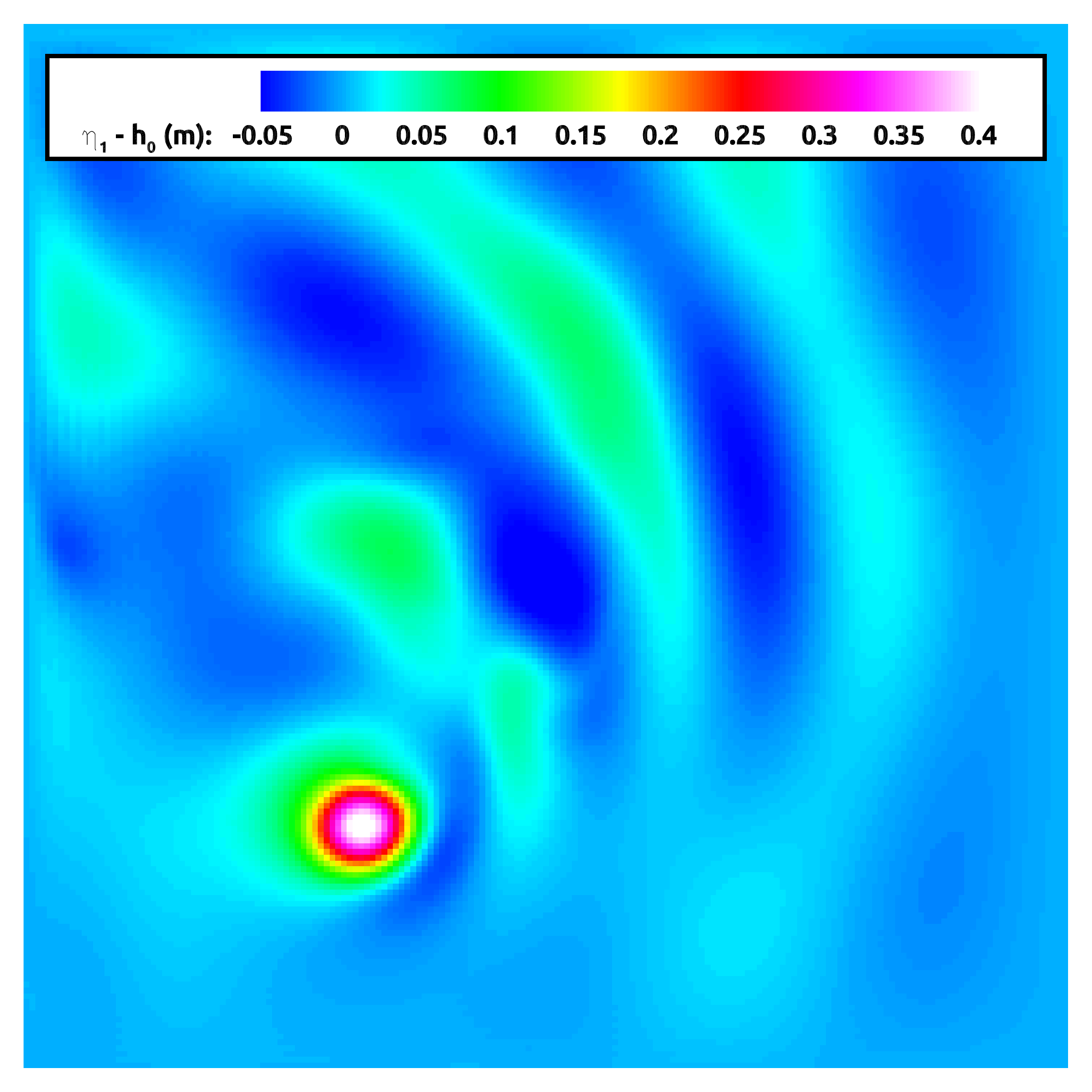}\includegraphics[width=0.45\textwidth]{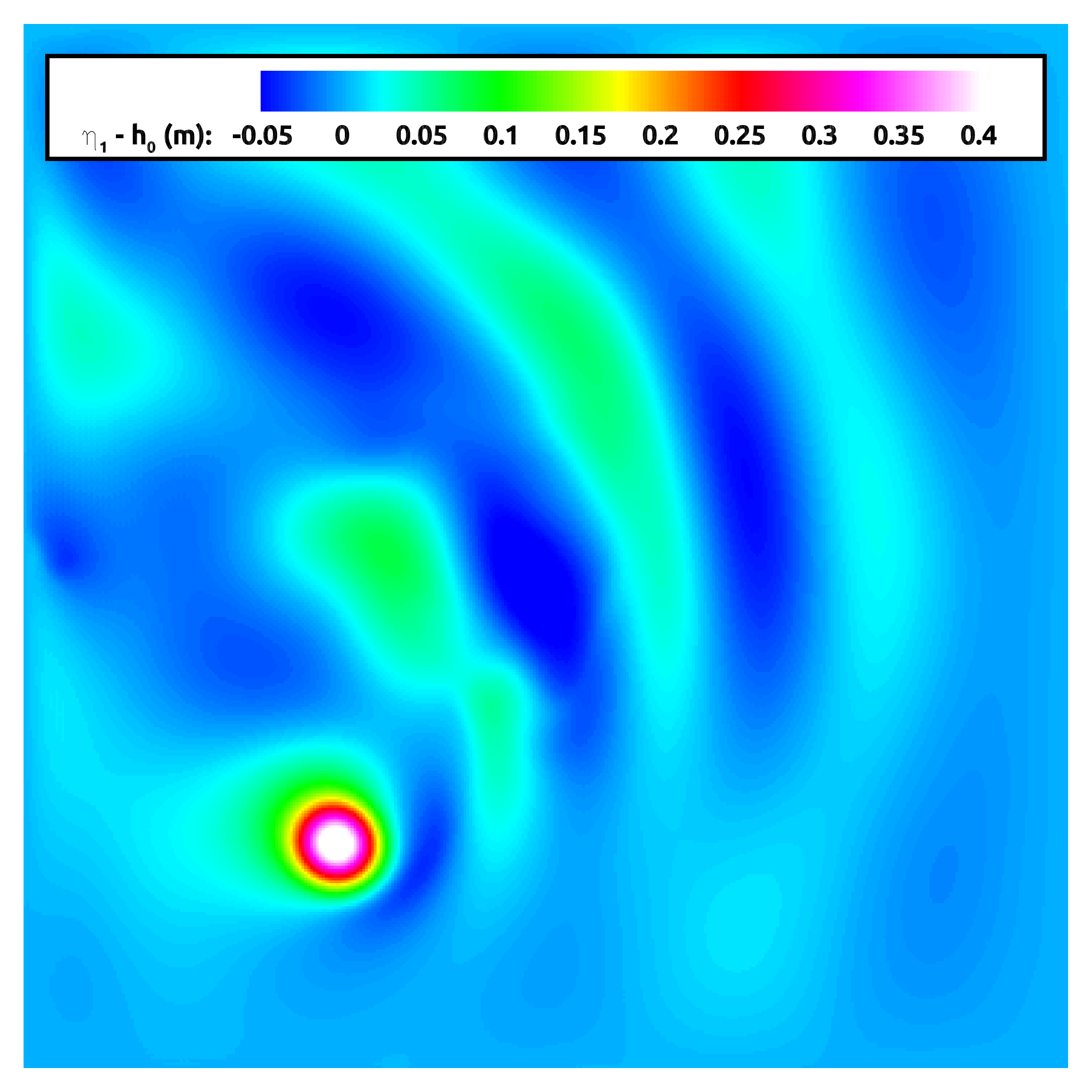} 
\par\end{centering}
\caption{Sea surface height for the baroclinic vortex test case obtained with
the present second-order scheme. After 100 days of simulation, the
vortex, initially centered, has moved to the southwest and small amplitudes
Rossby waves emission can be observed in the trajectory wake; (\textit{top-left})
$\Delta x=30\ \mathrm{km}$; (\textit{top-right}) $\Delta x=20\ \mathrm{km}$;
(\textit{bottom-left}) $\Delta x=10\ \mathrm{km}$; (\textit{bottom-right})
$\Delta x=5\ \mathrm{km}$.\label{fig:vortex-surfaces}}
\end{figure}

Simulations have been performed using the second-order scheme presented
in the Appendix \ref{subsec:Second-order-extension} with a time integration
period of $100$ days with five space resolutions $\Delta x=30\,\mathrm{km}\,,\,20\,\mathrm{km}\,,\,10\,\mathrm{km}\,,\,5\,\mathrm{km}\;\text{and}\;2\,\mathrm{km}$
corresponding respectively to discretizations of space domain with
$60\times60\times10\,,\,90\times90\times10\,,\,180\times180\times10\,,\,360\times360\times10\;\text{and}\;900\times900\times10$
cells and layers. It has been chosen $\gamma=0.2$ and $\alpha=0$
for the stabilization constants coefficients, with a CFL number of
$0.5$. We have seen before that this set of parameters is sufficient
to ensure the linear stability of the numerical scheme \S\ref{linear}.

The sea surface height for the first four resolutions are given in
Fig.\ref{fig:vortex-surfaces}. It can be roughly observed a relative
rapid convergence since the solutions for the $10\,\mathrm{km}$ and
$5\,\mathrm{km}$ resolutions are already very close. The vortex final
shape as well as the position and amplitude of the Rossby waves in
the trajectory wake are very similar, excepted maybe for very fine
structures. For the lower resolutions of $30\,\mathrm{km}$ and $20\,\mathrm{km}$,
the final axisymmetric vortex shape has not been completely broken,
resulting to relatively acceptable simulations. The large structures
of the emitted Rossby waves are correctly captured, especially the
two bands in the northeast. However, the vortex has clearly lost an
important energy as its maximum amplitude is lower than for the more
refined meshes.

Going further in the convergence analysis, it is given in Fig.\ref{fig:vortex-conv}
the time evolution of the vortex $y$-deviation (computed from the
maximum amplitude with bilinear interpolation), the vortex maximum
amplitude, the kinetic and mechanical energies, obtained from (\ref{eq:Discrete_Energy})
(substracting to the potential energy the unperturbed state contribution,
the mechanical energy has been rescaled to the initial value). The
overall results for the $5\,\mathrm{km}$ and $2\,\mathrm{km}$ are
sufficiently close to consider that the convergence has been very
nearly reached. The $2\,\mathrm{km}$ resolution exhibits a really
minimum of dissipation and will stand for a reference solution. We
give in Tab.\ref{tab:vortex} the associated $L\text{\texttwosuperior}$
error norms in time using this solution as reference. An asymptotic
convergence of 2 seems to be reached for all the diagnostic quantities.
Considering the $10\,\mathrm{km}$ resolution (an average mesh resolution
for oceanic simulations in practice) all the results are in very good
agreement with the reference solution. Towards the end of the simulation,
the kinetic energy loss starts to move away the vortex trajectory
from the converged one. For the two lower resolutions, the kinetic
energy is lost at the beginning of the simulation because of an initial
numerical imbalance between centripetal acceleration, pressure and
Coriolis forces. A lower decrease can be observed afterwards, highlighting
a good accuracy for long time simulations.

\begin{figure}[!tbh]
\begin{centering}
\includegraphics[width=0.9\textwidth]{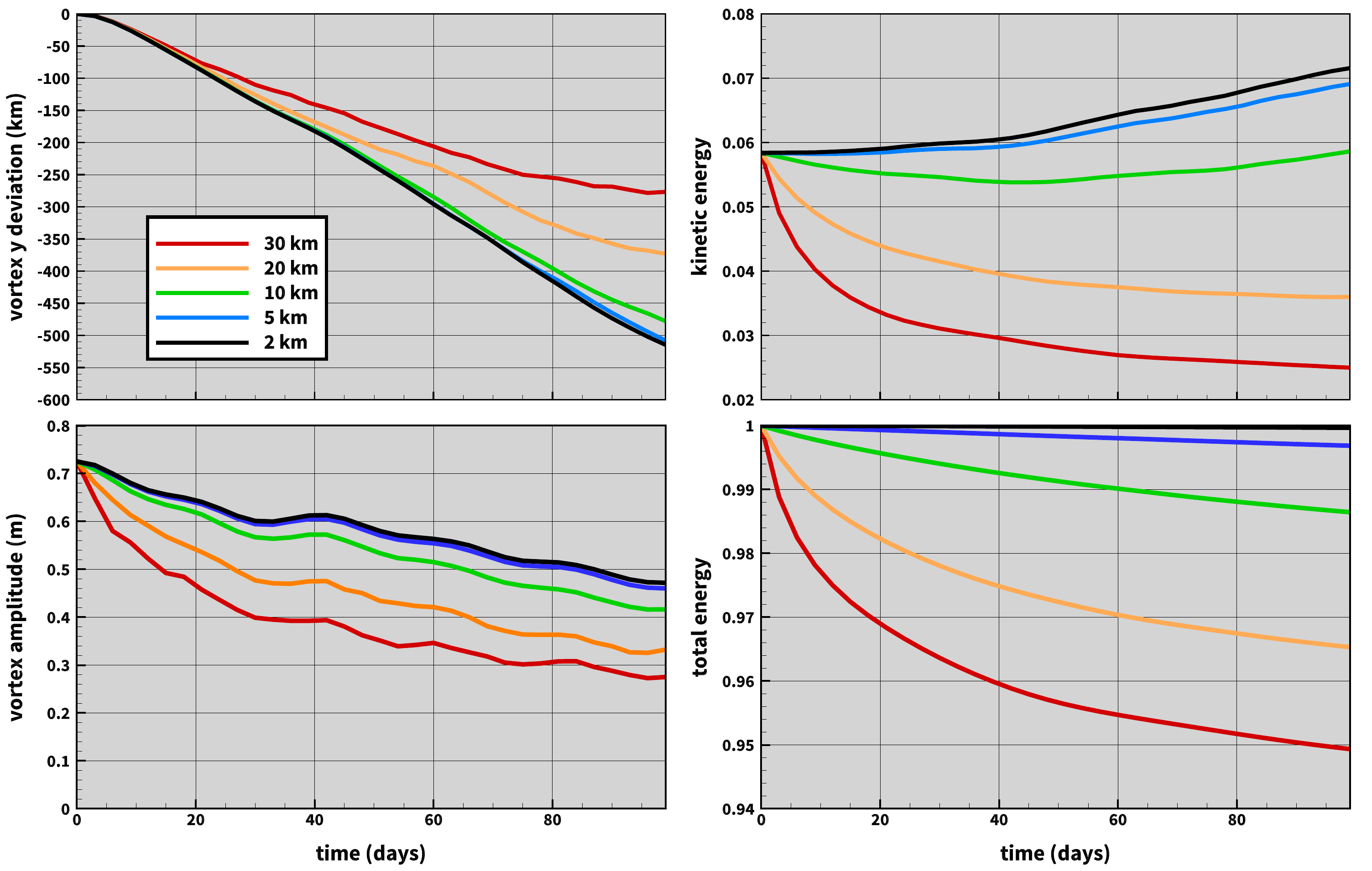} 
\par\end{centering}
\caption{Time evolution of some revelant diagnostic quantities for the baroclinic
vortex test case using the present second-order scheme. Simulations
stabilization constants are $\gamma=0.2$ and $\alpha=0$ with the
second-order scheme. Kinetic and total energies are computed from
(\ref{eq:Discrete_Energy}).\label{fig:vortex-conv}}
\end{figure}

\begin{table}[!tbh]
\begin{centering}
\begin{tabular}{>{\raggedleft}p{0.1\textwidth}>{\raggedleft}p{0.15\textwidth}>{\raggedleft}p{0.1\textwidth}>{\raggedleft}p{0.1\textwidth}>{\raggedleft}p{0.1\textwidth}>{\raggedleft}p{0.15\textwidth}>{\raggedleft}p{0.1\textwidth}}
\toprule 
$\Delta x$ & $\epsilon_{L_{2}}$ & order &  & $\Delta x$ & $\epsilon_{L_{2}}$ & order\tabularnewline
\midrule 
 &  &  &  &  &  & \tabularnewline
\multicolumn{3}{l}{Vortex amplitude (m)} &  & \multicolumn{3}{l}{Vortex y-deviation (km)}\tabularnewline
\cmidrule{1-3} \cmidrule{5-7} 
$30\,\mathrm{\mathrm{km}}$ & $1.99\:10^{-1}$ & - &  & $30\,\mathrm{\mathrm{km}}$ & $1.11\:10^{2}$ & -\tabularnewline
$20\,\mathrm{\mathrm{km}}$ & $1.31\:10^{-1}$ & $1.03$ &  & $20\,\mathrm{\mathrm{km}}$ & $6.37\:10^{1}$ & $1.37$\tabularnewline
$10\,\mathrm{\mathrm{km}}$ & $4.37\:10^{-2}$ & $1.58$ &  & $10\,\mathrm{\mathrm{km}}$ & $1.48\:10^{1}$ & $2.11$\tabularnewline
$5\,\mathrm{\mathrm{km}}$ & $8.13\:10^{-3}$ & $2.43$ &  & $5\,\mathrm{km}$ & $3.65\:10^{0}$ & $2.02$\tabularnewline
 &  &  &  &  &  & \tabularnewline
\multicolumn{3}{l}{Kinetic Energy} &  & \multicolumn{3}{l}{Mechanical Energy}\tabularnewline
\cmidrule{1-3} \cmidrule{5-7} 
$30\,\mathrm{\mathrm{km}}$ & $3.43\:10^{-2}$ & - &  & $30\,\mathrm{\mathrm{km}}$ & $4.07\:10^{-2}$ & -\tabularnewline
$20\,\mathrm{\mathrm{km}}$ & $2.45\:10^{-2}$ & $0.83$ &  & $20\,\mathrm{\mathrm{km}}$ & $2.63\:10^{-2}$ & $1.08$\tabularnewline
$10\,\mathrm{\mathrm{km}}$ & $8.57\:10^{-3}$ & $1.52$ &  & $10\,\mathrm{\mathrm{km}}$ & $8.80\:10^{-3}$ & $1.58$\tabularnewline
$5\,\mathrm{\mathrm{km}}$ & $1.59\:10^{-3}$ & $2.43$ &  & $5\,\mathrm{\mathrm{km}}$ & $1.70\:10^{-3}$ & $2.37$\tabularnewline
 &  &  &  &  &  & \tabularnewline
\bottomrule
\end{tabular}
\par\end{centering}
\caption{Numerical convergence results for the baroclinic vortex using the
present second-order scheme. The errors $\epsilon_{L^{2}}$ refer
to the $L^{2}$ error norm in time between the numerical solution
and the reference solution computed with $\Delta x=2\,\mathrm{km}$.\label{tab:vortex}}
\end{table}

From a numerical stablity point of view, it can be observed for all
the resolutions a strict decrease of the mechanical energy for the
chosen pair of stabilization constants $\gamma$ and $\alpha$. It
appears that the pressure stabilization term (\ref{alpen}) is not
required here to ensure a strict mechanical energy decrease, although
the simulated flow is very complex. Since this term is proportional
to the velocity divergence, it could be explained by a flow always
very close to the incompressible condition. Another explanation could
be the introduction of the Coriolis force that may have an impact
on the stability conditions. We also performed another series of simulations
for the $10\,\mathrm{km}$ resolution, with $\alpha=0.05\,,\,0.10\,,\,0.15$
and $0.20$ keeping the same other simulation parameters. The results
given in Fig.\ref{fig:vortex-alpha-influence} show a quick deterioration
for $\alpha$ increasing values. All the diagnostic quantities are
approximately in the range of the $20\,\mathrm{\mathrm{km}}$ and
$30\,\mathrm{\mathrm{km}}$ resolution results killing this stabilization
term. The pressure term is impacted by a more important initial numerical
imbalance. It can be easily verified looking at the initial kinetic
energy decrease.

\begin{figure}[!tbh]
\begin{centering}
\includegraphics[width=0.9\textwidth]{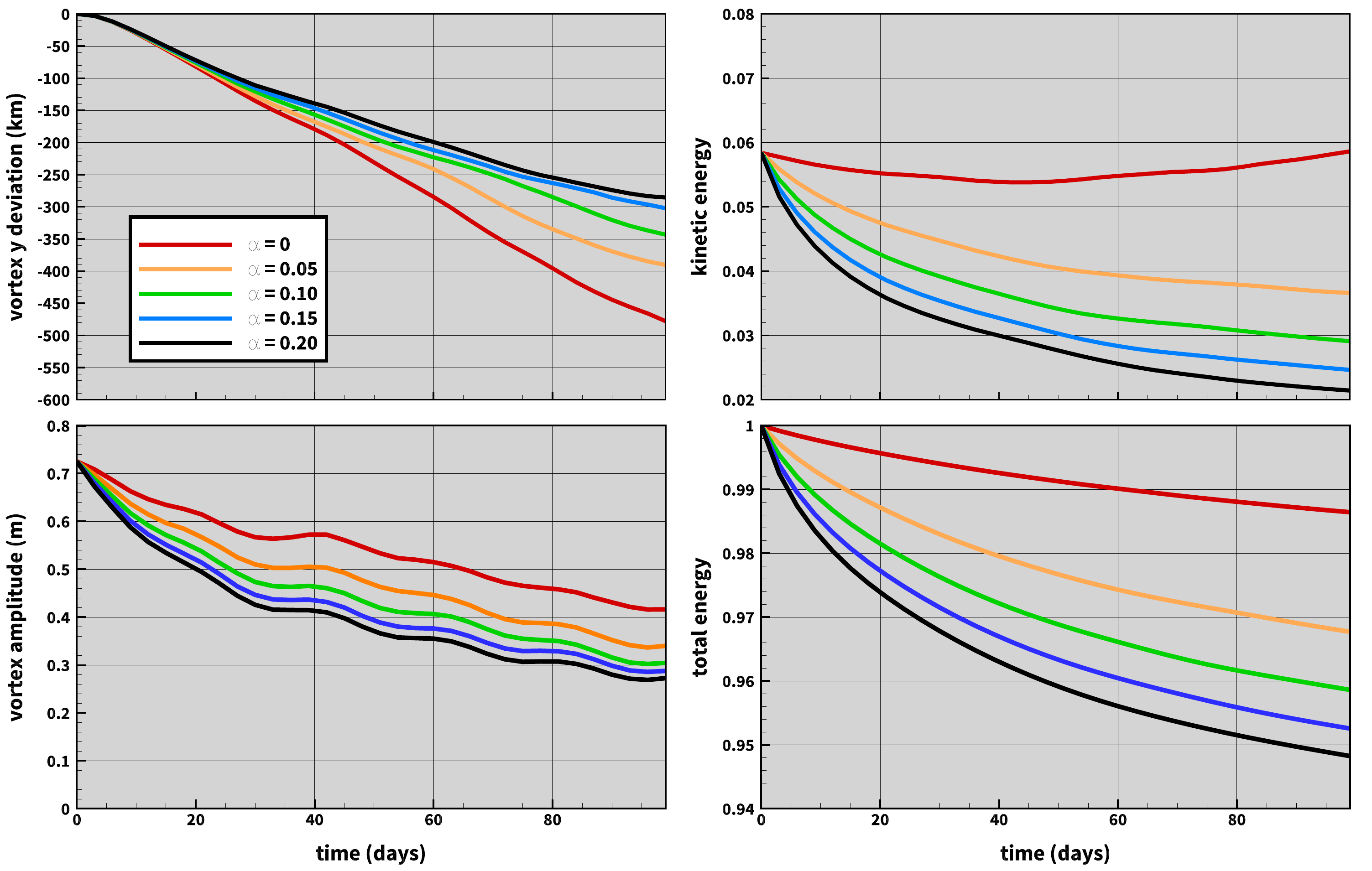} 
\par\end{centering}
\caption{Time evolution of some revelant diagnostic quantities for the baroclinic
vortex test case using the present second-order scheme. Simulations
stabilization constants are $\gamma=0.2$ and some varying $\alpha$
with the second-order scheme and a fixed mesh size $\Delta x=10\,\mathrm{km}$.\label{fig:vortex-alpha-influence}}
\end{figure}

\section{Conclusion}

In this paper we have introduced an explicit numerical scheme on unstructured
meshes for the two-dimensional multilayer shallow water system with
density stratification. The main characteristic of the numerical approach
stands in its ability to deal with the non conservative pressure\textit{
}term with strong stability properties, and without the need of evaluating
the system eigenvalues. The formalism is particularly adapted to deal
with well-balancing issues, and a positivity result is also exhibited.
Assuming a classical explicit CFL condition, the dissipation of the
mechanical energy has been demonstrated under sufficient inequality
conditions on a pair of stabilization constants, as well as the consistency
with respect to the low Froude regimes at different time scales, which
stand for two fundamental and challenging criteria in the context
of large-scale oceanic or estuary flows. The non linear study has
been complemented through a complete linear stability analysis for
the the first and second-order schemes, for the one and two-dimensional
problems. In particular, it has been observed that the calibration
of the stabilization constants could be significantly relaxed at second-order
with the use of an appropriate time scheme. The practical consequences
are undeniable since it allows to considerably limit the diffusive
losses in the numerical simulations. In view of these results, a more
advanced high order space and time analysis is currently in progress,
including an eventual extension to a general finite elements frame
as we believe that the proposed numerical method gives a solid framework
to derive high-order explicit schemes. As it is still confirmed by
our numerical experiments, these stability properties make the approach
particularly well suited to large-scale oceanic circulation, and competitive
with other softwares developed within the oceanographic community. 

In addition to high order space and time extensions, many other perspectives
are driven by the present developments. First, the explicit scheme's
efficiency must be compared with its semi-implicit version \citep{Parisot2015},
which accepts bigger time steps, but at the price of a more important
computational cost (due to the resolution of a nonlinear system) and
the difficulty to derive high order time and space extensions having
the same strong stability properties. Thus, to date, the time benefits
brought by the semi-implicit version are not so clear, especially
since the use of bigger time steps tends to rapidly deteriorate the
scheme's accuracy. Appropriate high order schemes need to be used
in order to limit this drawback. The global stability analysis of
the numerical scheme taking into account the Coriolis force with or
without time stepping also needs to be performed. In addition, and
in view of very promising preliminary results, the present approach
is currently oriented toward other crucial operational contexts such
as river flows or coastal applications. These works need futher investigations
to handle hydraulic jumps or wetting and drying areas, with the management
of disappearing layers or emerging topographies. Also, in light of
the numerical results, it appears crucial to study the possibility
of computing the two adimensional stabilization constants locally,
according notably to the discrete solution local regularity. This
flexibility may substantially improve the overall accuracy of the
method.

\section*{Acknowledgements}

This work was granted access to the HPC resources of CALMIP supercomputing
center under the allocation 2016-P1234.

\section{Appendix\label{sec:Appendix}}

The first part of this Appendix presents the main steps leading to
the control of the total energy production (proof of Theorem \ref{dissipation}).
We then give an interpretation of the numerical model in terms convex
combination of 1d schemes, as mentioned in Section \ref{sec:Asymptotic-regimes}.
Some technical aspects for implementation purposes are also proposed,
including the MUSCL reconstruction scheme (supplemented by a formal
extension of energy dissipation results), treatment of Coriolis force
and the fully explicit formula used for the time step selection.

\subsection{Stability results for the first-order scheme\label{Proofs}}

\subsubsection{Kinetic energy\label{App-propK}}

We begin by the kinetic energy, and set: 
\[
{\mathcal{K}_{K,i}^{n}}=\frac{1}{2}{H_{K,i}^{n}}\norm{\textbf{u}_{K,i}^{n}}^{2}\,.
\]
We have the following result:

\begin{proposition}\label{propK}\textit{Estimation of the kinetic
energy production}

\[
{\mathcal{K}_{K,i}^{n+1}}-{\mathcal{K}_{K,i}^{n}}+\dfrac{\Delta t}{m_{K}}{\sum_{e\in\partial K}}\big({\mathcal{G}_{\mathcal{K},e,i}^{n}}.{\textbf{n}_{e,K}}\big){m_{e}}+\mathcal{Q}_{\mathcal{K},K,i}\leq\mathcal{R}_{\mathcal{K},K,i}+\mathcal{H}_{\mathcal{K},K,i}-\mathcal{A}_{\mathcal{K},K,i}+\tilde{\mathcal{A}}_{\mathcal{K},K,i}\,,
\]
with
\begin{align}
{\mathcal{G}_{\mathcal{K},e,i}^{n}}.{\textbf{n}_{e,K}} & =\frac{1}{2}\norm{\textbf{u}_{K,i}^{n}}^{2}{\big(\mathcal{F}_{e,i}^{n}.\textbf{n}_{e,K}\big)^{+}}+\frac{1}{2}\norm{\textbf{u}_{K_{e},i}^{n}}^{2}{\big(\mathcal{F}_{e,i}^{n}.\textbf{n}_{e,K}\big)^{-}}\,,\nonumber \\
\mathcal{Q}_{\mathcal{K},K,i} & =\dfrac{\Delta t}{m_{K}}{H_{K,i}^{n}}\textbf{u}_{K,i}^{n}.{\sum_{e\in\partial K}}\dfrac{{\boldsymbol{\delta}\Phi_{e,i}^{n}}}{\varepsilon^{2}}{m_{e}}\,,\label{Qkk}\\
\mathcal{H}_{\mathcal{K},K,i} & =\dfrac{\Delta t}{m_{K}}{\sum_{e\in\partial K}}{\overline{H\textbf{u}}_{e,i}^{n}}.\dfrac{{\Lambda_{e,i}^{n}}}{\varepsilon^{2}}\textbf{n}_{e,K}{m_{e}}\,,\nonumber \\
\mathcal{A}_{\mathcal{K},K,i} & =\dfrac{\Delta t}{m_{K}}{\sum_{e\in\partial K}}\dfrac{{\Lambda_{e,i}^{n}}}{\varepsilon^{2}}\dfrac{1}{2}({H_{K_{e},i}^{n}}{\textbf{u}_{K_{e},i}^{n}}-{H_{K,i}^{n}}\textbf{u}_{K,i}^{n}).\textbf{n}_{e,K}{m_{e}}\,,\label{Akk}\\
\tilde{\mathcal{A}}_{\mathcal{K},K,i} & =2\left(\dfrac{\Delta t}{m_{K}}\right)^{2}\dfrac{({H_{K,i}^{n}})^{2}}{{H_{K,i}^{n+1}}}{m_{\partial K}}{\sum_{e\in\partial K}}\left(\dfrac{{\Lambda_{e,i}^{n}}}{\varepsilon^{2}}\right)^{2}{m_{e}}\,,\label{tAkk}\\
\mathcal{R}_{\mathcal{K},K,i} & =\left(\dfrac{\Delta t}{m_{K}}\right)^{2}\dfrac{({H_{K,i}^{n}})^{2}}{{H_{K,i}^{n+1}}}{m_{\partial K}}{\sum_{e\in\partial K}}\norm{\dfrac{{\boldsymbol{\delta}\Phi_{e,i}^{n}}}{\varepsilon^{2}}}^{2}{m_{e}}\,.
\label{Rkk}
\end{align}
\end{proposition}

\begin{proof}We drop the subscript \textit{``i''} for a better
readability. We first use the equation on $\textbf{u}$ (\ref{u1}):
\[
\begin{split}{H_{K}^{n+1}}(\textbf{u}_{K}^{n+1}-\textbf{u}_{K}^{n}).\textbf{u}_{K}^{n}= & -\dfrac{\Delta t}{m_{K}}{\sum_{e\in\partial K}}({\textbf{u}_{K_{e}}^{n}}-\textbf{u}_{K}^{n}).\textbf{u}_{K}^{n}{\big(\mathcal{F}_{e}^{n}.\textbf{n}_{e,K}\big)^{-}}{m_{e}}\\
 & -\dfrac{\Delta t}{m_{K}}{H_{K}^{n}}\textbf{u}_{K}^{n}.{\sum_{e\in\partial K}}\dfrac{{\Phi_{e}^{n,\ast}}}{\varepsilon^{2}}\textbf{n}_{e,K}{m_{e}}\,.
\end{split}
\]
Then, using the relation $(\textbf{a}-\textbf{b}).\textbf{b}=\dfrac{1}{2}\norm{\textbf{a}}^{2}-\dfrac{1}{2}\norm{\textbf{b}}^{2}-\dfrac{1}{2}\norm{\textbf{a}-\textbf{b}}^{2}$
: 
\[
\begin{split}{H_{K}^{n+1}}\Big(\dfrac{1}{2}\norm{\textbf{u}_{K}^{n+1}}^{2} & -\dfrac{1}{2}\norm{\textbf{u}_{K}^{n}}^{2}-\dfrac{1}{2}\norm{\textbf{u}_{K}^{n+1}-\textbf{u}_{K}^{n}}^{2}\Big)\\
= & -\dfrac{\Delta t}{m_{K}}{\sum_{e\in\partial K}}\Big(\dfrac{1}{2}\norm{\textbf{u}_{K_{e}}^{n}}^{2}-\dfrac{1}{2}\norm{\textbf{u}_{K}^{n}}^{2}-\dfrac{1}{2}\norm{{\textbf{u}_{K_{e}}^{n}}-\textbf{u}_{K}^{n}}^{2}\Big){\big(\mathcal{F}_{e}^{n}.\textbf{n}_{e,K}\big)^{-}}{m_{e}}\\
 & -\dfrac{\Delta t}{m_{K}}{H_{K}^{n}}\textbf{u}_{K}^{n}.{\sum_{e\in\partial K}}\dfrac{{\Phi_{e}^{n,\ast}}}{\varepsilon^{2}}\textbf{n}_{e,K}{m_{e}}\,.
\end{split}
\]
previous equality and invoking the mass equation (\ref{mass}), we
have:
\begin{equation}
\begin{split}{\tilde{\mathcal{K}}_{K}^{n+1}}-{\tilde{\mathcal{K}}_{K}^{n}}= & -\dfrac{\Delta t}{m_{K}}{\sum_{e\in\partial K}}\Big(\frac{1}{2}\norm{\textbf{u}_{K}^{n}}^{2}{\big(\mathcal{F}_{e}^{n}.\textbf{n}_{e,K}\big)^{+}}+\frac{1}{2}\norm{\textbf{u}_{K_{e}}^{n}}^{2}{\big(\mathcal{F}_{e}^{n}.\textbf{n}_{e,K}\big)^{-}}\Big){m_{e}}\\
 & +\dfrac{1}{2}{H_{K}^{n+1}}\norm{\textbf{u}_{K}^{n+1}-\textbf{u}_{K}^{n}}^{2}+\dfrac{\Delta t}{m_{K}}{\sum_{e\in\partial K}}\dfrac{1}{2}\norm{{\textbf{u}_{K_{e}}^{n}}-\textbf{u}_{K}^{n}}^{2}{\big(\mathcal{F}_{e}^{n}.\textbf{n}_{e,K}\big)^{-}}{m_{e}}\\
 & -\dfrac{\Delta t}{m_{K}}{H_{K}^{n}}\textbf{u}_{K}^{n}.{\sum_{e\in\partial K}}\dfrac{{\Phi_{e}^{n,\ast}}}{\varepsilon^{2}}\textbf{n}_{e,K}{m_{e}}\,.
\end{split}
\label{Bilan_K1}
\end{equation}
We now denote: 
\[
S_{K}=\dfrac{1}{2}{H_{K}^{n+1}}\norm{\textbf{u}_{K}^{n+1}-\textbf{u}_{K}^{n}}^{2}+\dfrac{\Delta t}{m_{K}}{\sum_{e\in\partial K}}\dfrac{1}{2}\norm{{\textbf{u}_{K_{e}}^{n}}-\textbf{u}_{K}^{n}}^{2}{\big(\mathcal{F}_{e}^{n}.\textbf{n}_{e,K}\big)^{-}}{m_{e}}\,,
\]
focus on the first term of $S_{K}$. We first use Jensen's inequality
with the weights $1/4,1/2,1/4$ to obtain a control of the form:
\[
\begin{split}\dfrac{1}{2}{H_{K}^{n+1}}\norm{\textbf{u}_{K}^{n+1}-\textbf{u}_{K}^{n}}^{2}\leq & \quad\dfrac{\left({H_{K}^{n}}\right)^{2}}{{H_{K}^{n+1}}}\left(\dfrac{\Delta t}{m_{K}}\right)^{2}\norm{{\sum_{e\in\partial K}}\dfrac{{\boldsymbol{\delta}\Phi_{e}^{n}}}{\varepsilon^{2}}\textbf{n}_{e,K}{m_{e}}}^{2}\\
 & +2\dfrac{\left({H_{K}^{n}}\right)^{2}}{{H_{K}^{n+1}}}\left(\dfrac{\Delta t}{m_{K}}\right)^{2}\norm{{\sum_{e\in\partial K}}\dfrac{{\Lambda_{e}^{n}}}{\varepsilon^{2}}\textbf{n}_{e,K}{m_{e}}}^{2}\\
 & +\dfrac{2}{{H_{K}^{n+1}}}\left(\dfrac{\Delta t}{m_{K}}\right)^{2}\norm{{\sum_{e\in\partial K}}\left({\textbf{u}_{K_{e}}^{n}}-\textbf{u}_{K}^{n}\right){\big(\mathcal{F}_{e}^{n}.\textbf{n}_{e,K}\big)^{-}}{m_{e}}}^{2}\,.
\end{split}
\]
We now carry on a separate analysis of each of the resulting terms.
Using again Jensen's inequality:
\begin{equation}
\begin{split}\norm{{\sum_{e\in\partial K}}\dfrac{{\overline{\Phi}_{e}^{n}}}{\varepsilon^{2}}\textbf{n}_{e,K}{m_{e}}}^{2} & \leq{m_{\partial K}}\left({\sum_{e\in\partial K}}\norm{\dfrac{{\boldsymbol{\delta}\Phi_{e}^{n}}}{\varepsilon^{2}}}^{2}{m_{e}}\right)\,,\\
\norm{{\sum_{e\in\partial K}}\dfrac{{\Lambda_{e}^{n}}}{\varepsilon^{2}}\textbf{n}_{e,K}{m_{e}}}^{2} & \leq{m_{\partial K}}\left({\sum_{e\in\partial K}}\left(\dfrac{{\Lambda_{e}^{n}}}{\varepsilon^{2}}\right)^{2}{m_{e}}\right)\,.
\end{split}
\label{JensenK}
\end{equation}
On the other hand, the Cauchy-Schwarz inequality gives: 
\[
\begin{split}\norm{{\sum_{e\in\partial K}}({\textbf{u}_{K_{e}}^{n}}-\textbf{u}_{K}^{n}){\big(\mathcal{F}_{e}^{n}.\textbf{n}_{e,K}\big)^{-}}{m_{e}}}^{2}\leq & \Big({\sum_{e\in\partial K}}\norm{{\textbf{u}_{K_{e}}^{n}}-\textbf{u}_{K}^{n}}^{2}{\big(\mathcal{F}_{e}^{n}.\textbf{n}_{e,K}\big)^{-}}{m_{e}}\Big)\Big({\sum_{e\in\partial K}}{\big(\mathcal{F}_{e}^{n}.\textbf{n}_{e,K}\big)^{-}}{m_{e}}\Big)\,.\end{split}
\]
Thus: 
\begin{equation}
\begin{split}S_{K}\leq & \left(\dfrac{\Delta t}{m_{K}}\right)^{2}\dfrac{({H_{K}^{n}})^{2}}{{H_{K}^{n+1}}}{m_{\partial K}}\left({\sum_{e\in\partial K}}\norm{\dfrac{{\boldsymbol{\delta}\Phi_{e}^{n}}}{\varepsilon^{2}}}^{2}{m_{e}}\right)+2\left(\dfrac{\Delta t}{m_{K}}\right)^{2}\dfrac{({H_{K}^{n}})^{2}}{{H_{K}^{n+1}}}{m_{\partial K}}\left({\sum_{e\in\partial K}}\left(\dfrac{{\Lambda_{e}^{n}}}{\varepsilon^{2}}\right)^{2}{m_{e}}\right)\\
+ & \dfrac{1}{2}\dfrac{\Delta t}{m_{K}}{\sum_{e\in\partial K}}\norm{{\textbf{u}_{K_{e}}^{n}}-\textbf{u}_{K}^{n}}^{2}{\big(\mathcal{F}_{e}^{n}.\textbf{n}_{e,K}\big)^{-}}{m_{e}}\times\Big[1-4\dfrac{\Delta t}{m_{K}}{\sum_{e\in\partial K}}\frac{-{\big(\mathcal{F}_{e}^{n}.\textbf{n}_{e,K}\big)^{-}}}{{H_{K}^{n+1}}}{m_{e}}\Big]\,.
\end{split}
\label{SK}
\end{equation}
The third term being assumed negative according to Remark \ref{Rob2}
(condition (\ref{Robustesse2}) with $\beta=1/4$), this yields the
remainder $\mathcal{R}_{\mathcal{K},K,i}$ (\ref{Rkk}) and the contribution
$\tilde{\mathcal{A}}_{\mathcal{K},K,i}$ (\ref{tAkk}). Finally, using
again (\ref{RelPhi}) the term involving the potential forces in (\ref{Bilan_K1})
is rewritten as: 
\[
\begin{split}\dfrac{\Delta t}{m_{K}}{H_{K}^{n}}\textbf{u}_{K}^{n}.{\sum_{e\in\partial K}}\dfrac{{\Phi_{e}^{n,\ast}}}{\varepsilon^{2}}\textbf{n}_{e,K}{m_{e}} & =\dfrac{\Delta t}{m_{K}}{H_{K}^{n}}\textbf{u}_{K}^{n}.{\sum_{e\in\partial K}}\dfrac{{\boldsymbol{\delta}\Phi_{e}^{n}}}{\varepsilon^{2}}{m_{e}}-\dfrac{\Delta t}{m_{K}}{H_{K}^{n}}\textbf{u}_{K}^{n}.{\sum_{e\in\partial K}}\dfrac{{\Lambda_{e}^{n}}}{\varepsilon^{2}}\textbf{n}_{e,K}{m_{e}}\,.\end{split}
\]
We use the relation ${H_{K}^{n}}\textbf{u}_{K}^{n}={\overline{H\textbf{u}}_{e}^{n}}+\dfrac{1}{2}({H_{K}^{n}}\textbf{u}_{K}^{n}-{H_{K_{e}}^{n}}{\textbf{u}_{K_{e}}^{n}})$
on the second member of the right hand side in the previous equality,
to finally obtain $\mathcal{Q}_{\mathcal{K},K,i}$, $\mathcal{H}_{\mathcal{K},K,i}$
and $\mathcal{A}_{\mathcal{K},K,i}$.\end{proof}

\subsubsection{Potential energy\label{App-propE}}

We now turn to the potential part, and denote ${\mathcal{E}_{K}^{n}}$
the potential energy on the cell $K$ at time $n$. We have the following
result:

\begin{proposition}\label{propE} \textit{Estimation of the potential
energy production:}

\[
{\mathcal{E}_{K}^{n+1}}-{\mathcal{E}_{K}^{n}}+\dfrac{\Delta t}{m_{K}}{\sum_{i=1}^{L}}{\sum_{e\in\partial K}}\big({\mathcal{G}_{\mathcal{E},e,i}^{n}}.{\textbf{n}_{e,K}}\big){m_{e}}-\mathcal{Q}_{\mathcal{E},K}\leq-\mathcal{R}_{\mathcal{E},K}+\mathcal{H}_{\mathcal{E},K}+\mathcal{A}_{\mathcal{E},K}+\tilde{\mathcal{R}}_{\mathcal{E},K}\,,
\]
with 
\begin{align}
{\mathcal{G}_{\mathcal{E},e,i}^{n}}.{\textbf{n}_{e,K}} & ={\overline{\Phi}_{e,i}^{n}}\mathcal{F}_{e,i}^{n}.\textbf{n}_{e,K}\,,\nonumber \\
\mathcal{Q}_{\mathcal{E},K} & =\dfrac{\Delta t}{m_{K}}{\sum_{i=1}^{L}}{H_{K,i}^{n}}\textbf{u}_{K,i}^{n}.{\sum_{e\in\partial K}}{\boldsymbol{\delta}\Phi_{e,i}^{n}}{m_{e}}\,,\label{Qek}\\
\mathcal{R}_{\mathcal{E},K} & =\dfrac{\Delta t}{m_{K}}{\sum_{i=1}^{L}}{\sum_{e\in\partial K}}\Pi_{e,i}^{n}.{\boldsymbol{\delta}\Phi_{e,i}^{n}}{m_{e}}\,,\label{Rek}\\
\mathcal{H}_{\mathcal{E},K} & =\dfrac{\Delta t}{m_{K}}{\sum_{i=1}^{L}}{\sum_{e\in\partial K}}\left(\dfrac{{H_{K_{e},i}^{n}}{\textbf{u}_{K_{e},i}^{n}}-{H_{K,i}^{n}}\textbf{u}_{K,i}^{n}}{2}\right).{\boldsymbol{\delta}\Phi_{e,i}^{n}}{m_{e}}\,,
\nonumber 
\end{align}
and the Taylor's residuals: 
\begin{align}
\tilde{\mathcal{R}}_{\mathcal{E},K} & ={C_{{\boldsymbol{\mathcal{H}}}}}\left(\dfrac{\Delta t}{m_{K}}\right)^{2}{m_{\partial K}}{\sum_{i=1}^{L}}{\sum_{e\in\partial K}}(\Pi_{e,i}^{n}.\textbf{n}_{e,K})^{2}{m_{e}}\,,\label{tRek}\\
\mathcal{A}_{\mathcal{E},K} & ={C_{{\boldsymbol{\mathcal{H}}}}}\left(\dfrac{\Delta t}{m_{K}}\right)^{2}{m_{\partial K}}{\sum_{i=1}^{L}}{\sum_{e\in\partial K}}({{\delta}(H\textbf{u})_{e,i}^{n}})^{2}{m_{e}}\,.\label{Aek}
\end{align}
\end{proposition}

\begin{proof}Using Taylor's formula between time steps $n$ and $n+1$,
we have for a certain $s\in\left[0,1\right]$: 
\[
{\mathcal{E}_{K}^{n+1}}-{\mathcal{E}_{K}^{n}}=-\dfrac{\Delta t}{m_{K}}{\sum_{i=1}^{L}}{\sum_{e\in\partial K}}{\Phi_{K,i}^{n}}\mathcal{F}_{e,i}^{n}.\textbf{n}_{e,K}{m_{e}}+\dfrac{1}{2}{\sum_{i=1}^{L}}{\sum_{j=1}^{L}}\big({H_{K,i}^{n+1}}-{H_{K,i}^{n}}\big){\boldsymbol{\mathcal{H}}}_{ij,K}^{n+s}\big(H_{K,j}^{n+1}-H_{K,j}^{n}\big)\,,
\]
where ${\boldsymbol{\mathcal{H}}}_{ij,K}^{n+s}={\boldsymbol{\mathcal{H}}}_{ij}\left(s{\boldsymbol{H}}_{K}^{n+1}+(1-s){\boldsymbol{H}}_{K}^{n},\mathbf{x}_{K}\right)$,
where we recall that $\boldsymbol{{\boldsymbol{H}}}_{K}^{n}=\leftidx{^{t}}{\left(H_{K,1}^{n},\cdots,H_{K,L}^{n}\right)}$.
Then we call the following decomposition: 

\[
\begin{split}{\Phi_{K,i}^{n}}\mathcal{F}_{e,i}^{n}.\textbf{n}_{e,K}{m_{e}} & ={\overline{\Phi}_{e,i}^{n}}\mathcal{F}_{e,i}^{n}.\textbf{n}_{e,K}{m_{e}}+({\Phi_{K,i}^{n}}-{\overline{\Phi}_{e,i}^{n}})\mathcal{F}_{e,i}^{n}.\textbf{n}_{e,K}{m_{e}}\\
 & ={\overline{\Phi}_{e,i}^{n}}\mathcal{F}_{e,i}^{n}.\textbf{n}_{e,K}{m_{e}}-{\overline{H\textbf{u}}_{e,i}^{n}}.\boldsymbol{\delta}\Phi_{e,i}^{n}{m_{e}}+\Pi_{e,i}^{n}.\boldsymbol{\delta}\Phi_{e,i}^{n}{m_{e}}\,.
\end{split}
\]
\\
Expanding ${\overline{H\textbf{u}}_{e,i}^{n}}={H_{K,i}^{n}}\textbf{u}_{K,i}^{n}+\left(\dfrac{{H_{K_{e},i}^{n}}{\textbf{u}_{K_{e},i}^{n}}-{H_{K,i}^{n}}\textbf{u}_{K,i}^{n}}{2}\right)$
we recover the symmetric fluxes ${\mathcal{G}_{\mathcal{E},e,i}^{n}}.{\textbf{n}_{e,K}}$
and the residuals $\mathcal{Q}_{\mathcal{E},K}$, $\mathcal{R}_{\mathcal{E},K}$,
$\mathcal{H}_{\mathcal{E},K}$. Concerning now the Taylor's residual,
we have, according to (\ref{HL2}): 
\begin{equation}
\mathcal{W}_{\mathcal{E},K}:=\dfrac{1}{2}{\sum_{i=1}^{L}}{\sum_{j=1}^{L}}\big({H_{K,i}^{n+1}}-{H_{K,i}^{n}}\big){\boldsymbol{\mathcal{H}}}_{ij,K}^{n+s}\big(H_{K,j}^{n+1}-H_{K,j}^{n}\big)\leq\dfrac{1}{2}{C_{{\boldsymbol{\mathcal{H}}}}}{\sum_{i=1}^{L}}\left({H_{K,i}^{n+1}}-{H_{K,i}^{n}}\right)^{2}\,.\label{WL2}
\end{equation}
We then reformulate (\ref{mass}): 
\[
\begin{split}{H_{K,i}^{n+1}}-{H_{K,i}^{n}} & =-\dfrac{\Delta t}{m_{K}}{\sum_{e\in\partial K}}\mathcal{F}_{e,i}^{n}.\textbf{n}_{e,K}{m_{e}}=-\dfrac{\Delta t}{m_{K}}{\sum_{e\in\partial K}}{\overline{H\textbf{u}}_{e,i}^{n}}.\textbf{n}_{e,K}{m_{e}}+\dfrac{\Delta t}{m_{K}}{\sum_{e\in\partial K}}\Pi_{e,i}^{n}.\textbf{n}_{e,K}{m_{e}}\,,\\
 & =-\dfrac{\Delta t}{m_{K}}{\sum_{e\in\partial K}}{{\delta}(H\textbf{u})_{e,i}^{n}}{m_{e}}+\dfrac{\Delta t}{m_{K}}{\sum_{e\in\partial K}}\Pi_{e,i}^{n}.\textbf{n}_{e,K}{m_{e}}\,,
\end{split}
\]
where we recall that ${{\delta}(H\textbf{u})_{e,i}^{n}}=\dfrac{1}{2}({H_{K_{e},i}^{n}}{\textbf{u}_{K_{e},i}^{n}}-{H_{K,i}^{n}}\textbf{u}_{K,i}^{n}).\textbf{n}_{e,K}$.
Injecting this in (\ref{WL2}), we use Jensen's inequality to obtain:
\begin{equation}
\begin{split}\mathcal{W}_{\mathcal{E},K}\leq & {C_{{\boldsymbol{\mathcal{H}}}}}{\sum_{i=1}^{L}}\left(\dfrac{\Delta t}{m_{K}}{\sum_{e\in\partial K}}{{\delta}(H\textbf{u})_{e,i}^{n}}{m_{e}}\right)^{2}+{C_{{\boldsymbol{\mathcal{H}}}}}{\sum_{i=1}^{L}}\left(\dfrac{\Delta t}{m_{K}}{\sum_{e\in\partial K}}\Pi_{e,i}^{n}.\textbf{n}_{e,K}{m_{e}}\right)^{2}\\
\leq & {C_{{\boldsymbol{\mathcal{H}}}}}\left(\dfrac{\Delta t}{m_{K}}\right)^{2}{m_{\partial K}}{\sum_{i=1}^{L}}{\sum_{e\in\partial K}}({{\delta}(H\textbf{u})_{e,i}^{n}})^{2}{m_{e}}+{C_{{\boldsymbol{\mathcal{H}}}}}\left(\dfrac{\Delta t}{m_{K}}\right)^{2}{m_{\partial K}}{\sum_{i=1}^{L}}{\sum_{e\in\partial K}}(\Pi_{e,i}^{n}.\textbf{n}_{e,K})^{2}{m_{e}}\,,
\end{split}
\label{JensenP}
\end{equation}
and fall on the two remaining terms of the estimation.\end{proof}

\subsubsection{Total energy\label{TotalE}}

Let's now consider ${E^{n}}={\displaystyle \sum_{K\in\mathbb{T}}{m_{K}}\left({\mathcal{E}_{K}^{n}}/\varepsilon^{2}+{\displaystyle {\sum_{i=1}^{L}}{\mathcal{K}_{K,i}^{n}}}\right)}$
the discrete mechanical energy, and focus on the non-antisymmetric
terms. We first observe an exact balance between the terms (\ref{Qkk})
and (\ref{Qek}) arising from the kinetic and potential parts. In
consequence the effort is put on a simultaneous control of the terms
$\mathcal{R}$ and $\mathcal{A}$ appearing in the kinetic and potential
energy budgets.\\

\textbf{\textit{Estimate 1 :}}

We gather the contributions issuing from the estimations on the kinetic
and potential discrete energies, i.e. (\ref{Rkk}) and (\ref{Rek},
\ref{tRek}) respectively: 
\[
\begin{split}{m_{K}}{\sum_{i=1}^{L}}\mathcal{R}_{\mathcal{K},K,i} & =\left(\Delta t\right)^{2}{\sum_{i=1}^{L}}\left(\dfrac{({H_{K,i}^{n}})^{2}}{{H_{K,i}^{n+1}}}\dfrac{{m_{\partial K}}}{{m_{K}}}\right){\sum_{e\in\partial K}}\norm{\dfrac{{\boldsymbol{\delta}\Phi_{e,i}^{n}}}{\varepsilon^{2}}}^{2}{m_{e}}\,,\\
-{m_{K}}\mathcal{R}_{\mathcal{E},K}/\varepsilon^{2} & =-\Delta t{\sum_{i=1}^{L}}{\sum_{e\in\partial K}}\Pi_{e,i}^{n}.\dfrac{{\boldsymbol{\delta}\Phi_{e,i}^{n}}}{\varepsilon^{2}}{m_{e}}\,,\\
{m_{K}}\tilde{\mathcal{R}}_{\mathcal{E},K}/\varepsilon^{2} & =\left(\Delta t\right)^{2}{C_{{\boldsymbol{\mathcal{H}}}}}\left(\dfrac{{m_{\partial K}}}{{m_{K}}}\right){\sum_{i=1}^{L}}{\sum_{e\in\partial K}}\left(\dfrac{\Pi_{e,i}^{n}.\textbf{n}_{e,K}}{\varepsilon}\right)^{2}{m_{e}}\,.
\end{split}
\]
As a preliminary step, we define:
\begin{equation}
\widehat{H}_{K,i}^{n}:=\dfrac{({H_{K,i}^{n}})^{2}}{{H_{K,i}^{n+1}}}=H_{K,i}^{n}+\mathcal{O}(\Delta t)\,.\label{hath}
\end{equation}
 Then, using $\left(\dfrac{({H_{K,i}^{n}})^{2}}{{H_{K,i}^{n+1}}}\dfrac{{m_{\partial K}}}{{m_{K}}}\right)=2\left(\dfrac{\widehat{H}}{{\Delta}}\right)_{K,i}^{n}=\left(\left(\dfrac{\widehat{H}}{{\Delta}}\right)_{K,i}^{n}+\left(\dfrac{\widehat{H}}{{\Delta}}\right)_{K_{e},i}^{n}\right)+\left(\left(\dfrac{\widehat{H}}{{\Delta}}\right)_{K,i}^{n}-\left(\dfrac{\widehat{H}}{{\Delta}}\right)_{K_{e},i}^{n}\right)$,
we split the first contribution ${m_{K}}{\sum_{i=1}^{L}}\mathcal{R}_{\mathcal{K},K,i}$
in a sum of symmetric and antisymmetric parts: 
\[
\begin{split}{m_{K}}{\sum_{i=1}^{L}}\mathcal{R}_{\mathcal{K},K,i} & =\left(\Delta t\right)^{2}{\sum_{i=1}^{L}}{\sum_{e\in\partial K}}2\left(\dfrac{\widehat{H}}{{\Delta}}\right)_{e,i}^{n}\norm{\dfrac{{\boldsymbol{\delta}\Phi_{e,i}^{n}}}{\varepsilon^{2}}}^{2}{m_{e}}\,\\
 & +\left(\Delta t\right)^{2}{\sum_{i=1}^{L}}{\sum_{e\in\partial K}}\frac{1}{2}\left(\left(\dfrac{\widehat{H}}{{\Delta}}\right)_{K,i}^{n}-\left(\dfrac{\widehat{H}}{{\Delta}}\right)_{K_{e},i}^{n}\right)\norm{\dfrac{{\boldsymbol{\delta}\Phi_{e,i}^{n}}}{\varepsilon^{2}}}^{2}{m_{e}}\,.
\end{split}
\]
In a similar way, with $\dfrac{1}{\Delta_{K}}=\dfrac{1}{2}\left(\dfrac{1}{\Delta_{K}}+\dfrac{1}{\Delta_{K_{e}}}\right)+\dfrac{1}{2}\left(\dfrac{1}{\Delta_{K}}-\dfrac{1}{\Delta_{K_{e}}}\right)$,
the term ${m_{K}}\tilde{\mathcal{R}}_{\mathcal{E},K}/\varepsilon^{2}$
reads: 
\[
\begin{split}{m_{K}}\tilde{\mathcal{R}}_{\mathcal{E},K}/\varepsilon^{2} & =\left(\Delta t\right)^{2}{\sum_{i=1}^{L}}{\sum_{e\in\partial K}}\frac{{C_{{\boldsymbol{\mathcal{H}}}}}}{\Delta_{e}}\left(\dfrac{\Pi_{e,i}^{n}.\textbf{n}_{e,K}}{\varepsilon}\right)^{2}{m_{e}}\,\\
 & +\left(\Delta t\right)^{2}{\sum_{i=1}^{L}}{\sum_{e\in\partial K}}\frac{{C_{{\boldsymbol{\mathcal{H}}}}}}{2}\left(\frac{1}{\Delta_{K}}-\frac{1}{\Delta_{K_{e}}}\right)\left(\dfrac{\Pi_{e,i}^{n}.\textbf{n}_{e,K}}{\varepsilon}\right)^{2}{m_{e}}\,.
\end{split}
\]
Dropping the antisymmetric terms, which vanish after global summation,
we use (\ref{alpen}): 
\[
\Pi_{e,i}^{n}=\gamma\Delta t\left(\dfrac{\widehat{H}}{{\Delta}}\right)_{e,i}^{n}\dfrac{{\boldsymbol{\delta}\Phi_{e,i}^{n}}}{\varepsilon^{2}}\,\,\,,\,\,\gamma>0\,,
\]
to write the total contribution as: 
\begin{equation}
\begin{split}\sum_{K\in\mathbb{T}}{m_{K}}\left({\sum_{i=1}^{L}}\mathcal{R}_{\mathcal{K},K,i}\tilde{\mathcal{R}}_{\mathcal{E},K}/\varepsilon^{2}-\mathcal{R}_{\mathcal{E},K}/\varepsilon^{2}+\tilde{\mathcal{R}}_{\mathcal{E},K}/\varepsilon^{2}\right)\qquad\qquad\qquad\\
=\left(\Delta t\right)^{2}\sum_{K\in\mathbb{T}}{\sum_{i=1}^{L}}{\sum_{e\in\partial K}}\left[2+\gamma^{2}\left(\dfrac{\left(\Delta t\right)^{2}}{\varepsilon^{2}}\frac{{C_{{\boldsymbol{\mathcal{H}}}}}}{\Delta_{e}}\left(\dfrac{\widehat{H}}{{\Delta}}\right)_{e,i}^{n}\right)-\gamma\right] & \left(\dfrac{\widehat{H}}{{\Delta}}\right)_{e,i}^{n}\norm{\dfrac{{\boldsymbol{\delta}\Phi_{e,i}^{n}}}{\varepsilon^{2}}}^{2}{m_{e}}\,.
\end{split}
\label{gamma}
\end{equation}
Defining the quantity $\rho_{\varepsilon}$ such that: 
\begin{equation}
\rho_{\varepsilon}^{2}=2\dfrac{\left(\Delta t\right)^{2}}{\varepsilon^{2}}\frac{{C_{{\boldsymbol{\mathcal{H}}}}}}{\Delta_{e}}\left(\dfrac{\widehat{H}}{{\Delta}}\right)_{e,i}^{n}\,,\label{re}
\end{equation}
the negativity of (\ref{gamma}) reduces to: 
\begin{equation}
p(\gamma)=\frac{1}{2}\rho_{\varepsilon}^{2}\gamma^{2}-\gamma+2\leq0\,.\label{p}
\end{equation}
Based on the positivity of the discriminant (that is $\rho_{\varepsilon}\leq\dfrac{1}{2}$)
and the roots of $p$: $\gamma^{\pm}=\dfrac{1\pm\sqrt{1-4\rho_{\varepsilon}^{2}}}{\rho_{\varepsilon}^{2}}$,
one can establish that the value $\gamma=4$ ensures the negativity
of $p$.\\

\textbf{\textit{Estimate 2:}}

We consider the three remaining terms involved in the energy budget
(\ref{Akk}), (\ref{tAkk}) and (\ref{Aek}): 
\[
\begin{split}-{m_{K}}{\sum_{i=1}^{L}}\mathcal{A}_{\mathcal{K},K,i} & =-\Delta t{\sum_{i=1}^{L}}{\sum_{e\in\partial K}}\dfrac{{\Lambda_{e,i}^{n}}}{\varepsilon^{2}}{{\delta}(H\textbf{u})_{e,i}^{n}}{m_{e}}\,,\\
{m_{K}}{\sum_{i=1}^{L}}\tilde{\mathcal{A}}_{\mathcal{K},K,i} & =2\left(\Delta t\right)^{2}{\sum_{i=1}^{L}}\left(\dfrac{({H_{K,i}^{n}})^{2}}{{H_{K,i}^{n+1}}}\dfrac{{m_{\partial K}}}{{m_{K}}}\right){\sum_{e\in\partial K}}\left(\dfrac{{\Lambda_{e,i}^{n}}}{\varepsilon^{2}}\right)^{2}{m_{e}}\,,\\
{m_{K}}\mathcal{A}_{\mathcal{E},K}/\varepsilon^{2} & =\left(\Delta t\right)^{2}{C_{{\boldsymbol{\mathcal{H}}}}}\left(\dfrac{{m_{\partial K}}}{{m_{K}}}\right){\sum_{i=1}^{L}}{\sum_{e\in\partial K}}({{\delta}(H\textbf{u})_{e,i}^{n}}/\varepsilon)^{2}{m_{e}}\,.
\end{split}
\]
In the spirit of the previous analysis we decompose ${m_{K}}{\sum_{i=1}^{L}}\tilde{\mathcal{A}}_{\mathcal{K},K,i}$
and ${m_{K}}{\sum_{i=1}^{L}}\mathcal{A}_{\mathcal{E},K}$ as follows:
\[
\begin{split}{m_{K}}{\sum_{i=1}^{L}}\tilde{\mathcal{A}}_{\mathcal{K},K,i} & =4\left(\Delta t\right)^{2}{\sum_{i=1}^{L}}{\sum_{e\in\partial K}}\left(\dfrac{\widehat{H}}{{\Delta}}\right)_{e,i}^{n}\left(\dfrac{{\Lambda_{e,i}^{n}}}{\varepsilon^{2}}\right)^{2}{m_{e}}\,\\
 & +\left(\Delta t\right)^{2}{\sum_{i=1}^{L}}{\sum_{e\in\partial K}}\left(\left(\dfrac{\widehat{H}}{{\Delta}}\right)_{K,i}^{n}-\left(\dfrac{\widehat{H}}{{\Delta}}\right)_{K_{e},i}^{n}\right)\left(\dfrac{{\Lambda_{e,i}^{n}}}{\varepsilon^{2}}\right)^{2}{m_{e}}\,\\
{m_{K}}\mathcal{A}_{\mathcal{E},K}/\varepsilon^{2} & =\left(\Delta t\right)^{2}{\sum_{i=1}^{L}}{\sum_{e\in\partial K}}\frac{{C_{{\boldsymbol{\mathcal{H}}}}}}{\Delta_{e}}({{\delta}(H\textbf{u})_{e,i}^{n}}/\varepsilon)^{2}{m_{e}}\,\\
 & +\left(\Delta t\right)^{2}{\sum_{i=1}^{L}}{\sum_{e\in\partial K}}\frac{{C_{{\boldsymbol{\mathcal{H}}}}}}{2}\left(\frac{1}{\Delta_{K}}-\frac{1}{\Delta_{K_{e}}}\right)({{\delta}(H\textbf{u})_{e,i}^{n}}/\varepsilon)^{2}{m_{e}}\,.
\end{split}
\]
Again we neglect the antisymmetric terms, and consider (\ref{alpen}):
\[
{\Lambda_{e,i}^{n}}=\alpha{C_{{\boldsymbol{\mathcal{H}}}}}\Delta t\frac{{{\delta}(H\textbf{u})_{e,i}^{n}}}{\Delta_{e}}\quad,\quad\alpha>0\,.
\]
The total contribution attached to these terms becomes: 
\begin{equation}
\begin{split}\sum_{K\in\mathbb{T}}{\sum_{e\in\partial K}}{m_{K}}\left(-{\sum_{i=1}^{L}}\tilde{\mathcal{A}}_{\mathcal{K},K,i}+{\sum_{i=1}^{L}}\tilde{\mathcal{A}}_{\mathcal{K},K,i}+\mathcal{A}_{\mathcal{E},K}/\varepsilon^{2}\right)\qquad\qquad\\
=\left(\Delta t\right)^{2}\sum_{K\in\mathbb{T}}{\sum_{i=1}^{L}}{\sum_{e\in\partial K}}\left[-\alpha+\alpha^{2}\left(4\dfrac{\left(\Delta t\right)^{2}}{\varepsilon^{2}}\frac{{C_{{\boldsymbol{\mathcal{H}}}}}}{\Delta_{e}}\left(\dfrac{\widehat{H}}{{\Delta}}\right)_{e,i}^{n}\right)+1\right] & \frac{{C_{{\boldsymbol{\mathcal{H}}}}}}{\Delta_{e}}\left(\dfrac{{{\delta}(H\textbf{u})_{e,i}^{n}}}{\varepsilon}\right)^{2}{m_{e}}\,.
\end{split}
\label{alpha}
\end{equation}
Using the same notations as previously, we are this time left with
the study of the second-order polynomial: 
\begin{equation}
q(\alpha)=2\rho_{\varepsilon}^{2}\alpha^{2}-\alpha+1\leq0\,.\label{q}
\end{equation}
Supposing that $\rho_{\varepsilon}\leq\dfrac{1}{2\sqrt{2}}$, the
real roots are $\alpha^{\pm}=\dfrac{1\pm\sqrt{1-8\rho_{\varepsilon}^{2}}}{4\rho_{\varepsilon}^{2}}$,
from which we extract the value $\alpha=2$.

\subsection{Reformulation as convex combination of 1d schemes\label{subsec:Reformulation-as-convex}}

Following the ideas of \citep{Berthon2006} (see also \citep{Duran2015}
for an application to the Shallow Water equations), each cell $K$
is divided in a subgrid made of triangles $T_{K,e}$, connecting the
edges $e\in\partial K$ to the mass center of $K$ (see Fig.\ref{fig:convex}).
\begin{figure}[!tbh]
\begin{centering}
\includegraphics[width=0.35\textwidth]{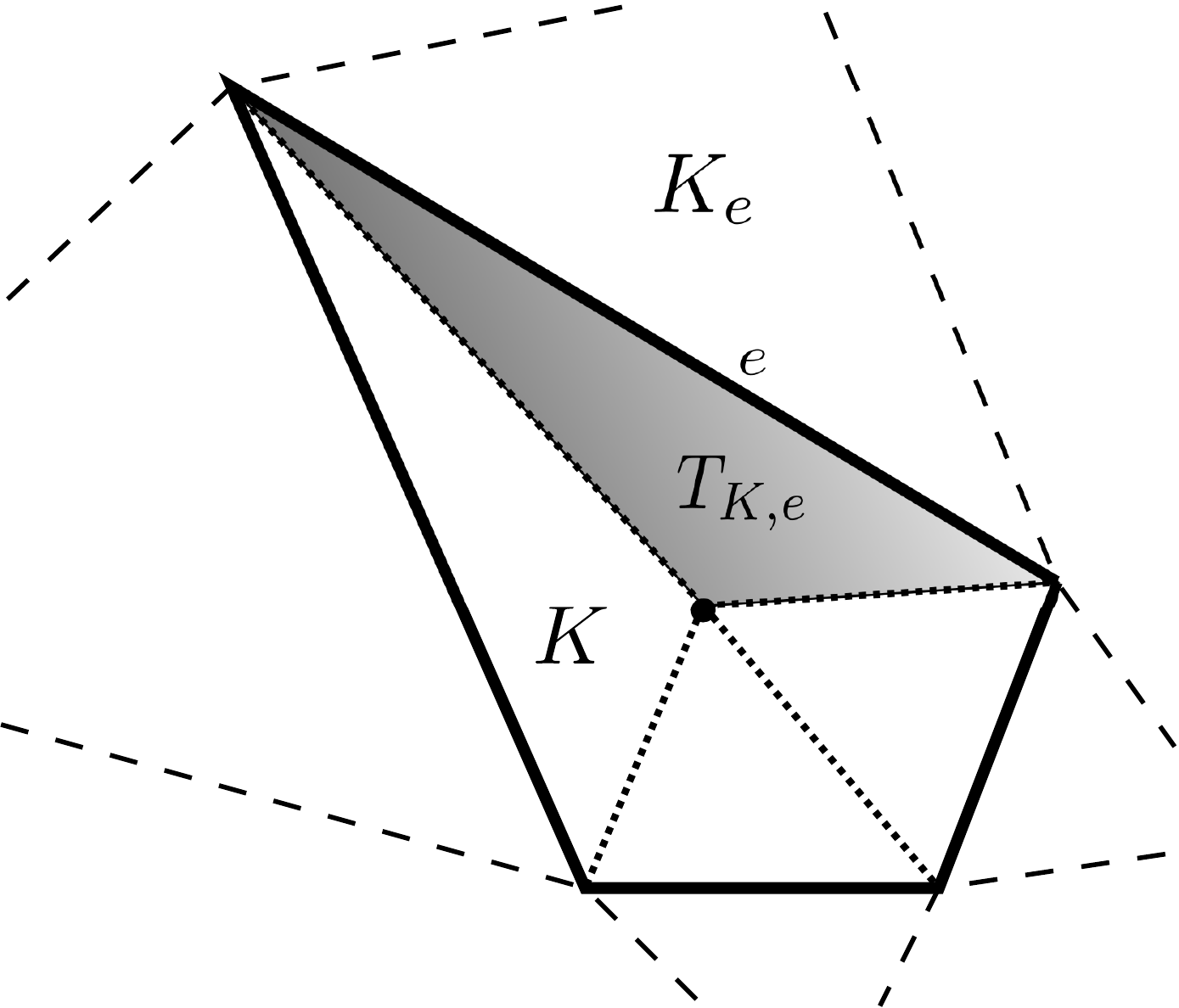}
\par\end{centering}
\caption{Mesh subgrid associated with an element $K$. Focus on the interface
$e$: the triangle $T_{K,e}$ connects $e$ to the mass center of
$K$. \label{fig:convex}}
\end{figure}
Gathering the discrete variables of the model in the vectors ${\mathbf{W}}_{K}$,
the mass and momentum fluxes involved in the scheme (\ref{mass},
\ref{mom}), together with the discrete gradient pressure, can be
reformulated in terms of functions of ${\mathbf{W}}_{K},\,{\mathbf{W}}_{K_{e}}$
and $n_{e,K}$ , through the following notations (we drop the subscript
\textit{``i''} to alleviate the notations):
\begin{align}
\mathcal{F}_{e}^{n}.\mathbf{n}_{e,K} & =\mathcal{F}({\mathbf{W}}_{K}^{n},\,{\mathbf{W}}_{K_{e}}^{n},\,{\mathbf{n}}_{e,K})\nonumber \\
\mathcal{G}_{e}^{n}.\mathbf{n}_{e,K} & =\mathcal{G}({\mathbf{W}}_{K}^{n},\,{\mathbf{W}}_{K_{e}}^{n},\,{\mathbf{n}}_{e,K})=\mathbf{u}_{K}^{n}\left(\mathcal{F}_{e}^{n}.\mathbf{n}_{e,K}\right)^{+}+\mathbf{u}_{K_{e}}^{n}\left(\mathcal{F}_{e}^{n}.\mathbf{n}_{e,K_{e}}\right)^{-}\qquad.\label{eq:fluxes}\\
\mathcal{P}_{e}^{n} & ={\Phi_{e}^{n,\ast}}{\mathbf{n}}_{e,K}=\mathcal{P}({\mathbf{W}}_{K}^{n},\,{\mathbf{W}}_{K_{e}}^{n},\,{\mathbf{n}}_{e,K})\,\nonumber 
\end{align}
Then, denoting $m_{T_{K,e}}$ the area of $T_{K,e}$, the scheme (\ref{mass},
\ref{mom}) can be written as a convex combination of one-dimensional
schemes:

\begin{subequations}
\begin{empheq}[left=\empheqlbrace,right=\text{\quad ,}]{align}
H_{K}^{n+1}& =  {\displaystyle \sum_{e \in \partial K } \dfrac{m_{T_{K,e}}}{m_K} H_{e}^{n+1} }\label{mass1d}
\\ 
H_{K}^{n+1} \mathbf{u}_{K}^{n+1}  &=  {\displaystyle  \sum_{e \in \partial K }  \dfrac{m_{T_{K,e}}}{m_K} H_{e}^{n+1} \mathbf{u}_{e}^{n+1}   }\,  \label{mom1d}
\end{empheq} \end{subequations}where we have introduced the auxiliary variables:

\begin{subequations}
\begin{empheq}[left=\empheqlbrace,right=\text{\quad ,}]{align}
H_{e}^{n+1}& =  H_{K}^{n}  -  {\displaystyle \dfrac{\Delta t}{\Delta x_e} 
\left[\mathcal{F}({\mathbf{W}}_{K},\,{\mathbf{W}}_{K_{e}},\,{\mathbf{n}}_{e,K})  -  \mathcal{F}({\mathbf{W}}_{K},\,{\mathbf{W}}_{K},\,{\mathbf{n}}_{e,K})  \right]}\label{mass1d}
\\ H_{e}^{n+1} \mathbf{u}_{e}^{n+1}  &=  H_{K}^{n} \mathbf{u}_{K}^{n}  -  {\displaystyle \dfrac{\Delta t}{\Delta x_e}
\left[\mathcal{G}({\mathbf{W}}_{K},\,{\mathbf{W}}_{K_{e}},\,{\mathbf{n}}_{e,K}) - \mathcal{G}({\mathbf{W}}_{K},\,{\mathbf{W}}_{K},\,{\mathbf{n}}_{e,K})\right]} \label{mom1d}
\\  &   \qquad \qquad - {\displaystyle \dfrac{\Delta t}{\Delta x_e} H_{K}^{n} \left[\mathcal{P}({\mathbf{W}}_{K},\,{\mathbf{W}}_{K_{e}},\,{\mathbf{n}}_{e,K}) - \mathcal{P}({\mathbf{W}}_{K},\,{\mathbf{W}}_{K},\,{\mathbf{n}}_{e,K}) \right] } \,\nonumber     \end{empheq}  \end{subequations} and the geometric constant $\Delta x_{e}=\dfrac{m_{T_{K,e}}}{m_{e}}$.

\subsection{Second-order extension\label{subsec:Second-order-extension}}

\subsubsection{MUSCL reconstructions\label{subsec:MUSCL}}

We consider in this work a monoslope second-order MUSCL scheme, which
consists of local linear reconstructions by computing a vectorial
slope $\left[\mathbf{\nabla}\mathbf{W}_{K}\right]_{m}$ in each cell
$K$ and for each primitive variable $m$, such that the two reconstructed
primitive variables vectors $\mathbf{W}_{e,K}$ and $\mathbf{W}_{e,K_{e}}$
are evaluated at each side of edge $e$ by:

\begin{equation}
\begin{array}{c}
\mathbf{W}_{e,K_{\phantom{e}}}=\mathbf{W}_{K_{\phantom{e}}}+\mathbf{\nabla}\mathbf{W}_{K_{\phantom{e}}}.\mathbf{x}_{K_{\phantom{e}}}\mathbf{x}_{e}\\
\mathbf{W}_{e,K_{e}}=\mathbf{W}_{K_{e}}+\mathbf{\nabla}\mathbf{W}_{K_{e}}.\mathbf{x}_{K_{e}}\mathbf{x}_{e}
\end{array}\,.\label{eq:muscl_recons_vector}
\end{equation}

These quantities are intended to replace the primitive variables in
the first-order scheme (Eqs.\ref{mass}-\ref{mom}-\ref{phiea}-\ref{fe})
to evaluate the numerical flux $\mathcal{F}_{e}^{n}$ and the pressure
$\Phi_{e}^{n,\ast}$ at the edge $e$. Classically, with such a linear
reconstruction, one can expect a scheme with a second-order accuracy
in space for sufficient regular solutions. To this end, a least square
method is employed to compute the vectorial slopes for each primitive
variable $h_{K}^{n}$, $u_{K}^{n}$ and $v_{K}^{n}$. More explicitly,
the following sums of squares

\begin{equation}
\begin{array}{c}
E_{m}\left(\left[\mathbf{\nabla}\mathbf{W}_{K}\right]_{m}\right)={\displaystyle \sum_{e\in\partial K}\left(\left[\mathbf{W}_{K_{e}}\right]_{m}-\left(\left[\mathbf{W}_{K}\right]_{m}+\left[\mathbf{\nabla}\mathbf{W}_{K}\right]_{m}.\mathbf{x}_{K}\mathbf{x}_{K_{e}}\right)\right)^{2}}\end{array}\,,\label{eq:ls_muscl}
\end{equation}

\noindent are minimized by setting the gradients to zero solution
of simple 2 x 2 linear systems. This method represents a good alternative
among others to find the hyperplane because of its accuracy and robustness,
independently from the number of neighbours. No limitation is imposed
to the computed vectorial slope because most of the numerical solutions
considered in this work are largely sufficiently regular and far from
wet/dry conditions to ensure numerical stability (except a Barth limiter
\citep{Barth2004} for the lake test case \S\ref{subsec:lake}).

\subsubsection{Second-order scheme\label{subsec:second-order-scheme}}

With the two reconstructed primitive variables vectors $\mathbf{W}_{e,K}^{n}$
and $\mathbf{W}_{e,K_{e}}^{n}$ at each side of the edge $e$, interface
terms are simply replaced in the original first-order scheme. In the
general $L$ layer case, and omitting the subscript ``i'' referring
to the layer numbering for the sake of clarity, this leads to the
scheme:

\begin{equation}
\left\{ \begin{array}{lcll}
H_{K}^{n+1} & = & H_{K}^{n} & -{\displaystyle \frac{\Delta t}{m_{K}}}{\displaystyle \sum_{e\in\partial K}}\left(\mathcal{F}_{e}^{n}.\mathbf{n}_{e,K}\right)m_{e}\\
\\
H_{K}^{n+1}\mathbf{u}_{K}^{n+1} & = & H_{K}^{n}\mathbf{u}_{K}^{n} & -{\displaystyle \frac{\Delta t}{m_{K}}}{\displaystyle \sum_{e\in\partial K}}\left({\color{red}{\normalcolor \mathbf{u}_{e,K}^{n}}}\left(\mathcal{F}_{e}^{n}.\mathbf{n}_{e,K}\right)^{+}+{\color{red}{\normalcolor \mathbf{u}_{e,K_{e}}^{n}}}\left(\mathcal{F}_{e}^{n}.\mathbf{n}_{e,K}\right)^{-}\right)m_{e}\\
\\
 &  &  & -{\displaystyle \frac{\Delta t}{m_{K}}}H_{K}^{n}{\displaystyle \sum_{e\in\partial K}}\left(\dfrac{\Phi_{e}^{n,\ast}}{\varepsilon^{2}}\,\mathbf{n}_{e,K}\right)m_{e}
\end{array}\right.\,,\label{eq:exp-scheme-1-1}
\end{equation}

with

\begin{equation}
\left\{ \begin{array}{l}
\mathcal{F}_{e}^{n}={\displaystyle \frac{{\color{red}{\normalcolor H_{e,K}^{n}}}{\color{red}{\normalcolor \mathbf{u}_{e,K}^{n}}}+{\color{red}{\normalcolor H_{e,K_{e}}^{n}}}{\color{red}{\normalcolor \mathbf{u}_{e,K_{e}}^{n}}}}{2}-\frac{\gamma\Delta t}{4}\left({\color{red}{\normalcolor H_{e,K}^{n}}}\frac{m_{\partial K}}{m_{K}}+{\color{red}{\normalcolor H_{e,K_{e}}^{n}}}\frac{m_{\partial K_{e}}}{m_{K_{e}}}\right)\left(\frac{{\color{red}{\normalcolor \Phi_{e,K_{e}}^{n}}}-{\color{red}{\normalcolor \Phi_{e,K}^{n}}}}{2\varepsilon^{2}}\right)\mathbf{n}_{e,K}}\\
\\
\Phi_{e}^{n,\ast}={\displaystyle \frac{{\color{red}{\normalcolor \Phi_{e,K}^{n}}}+{\color{red}{\normalcolor \Phi_{e,K_{e}}^{n}}}}{2}-\frac{\alpha\Delta t}{2}gL\left(\frac{m_{\partial K}}{m_{K}}+\frac{m_{\partial K_{e}}}{m_{K_{e}}}\right)\left(\dfrac{{\color{red}{\normalcolor H_{e,K_{e}}^{n}}}{\color{red}{\normalcolor \mathbf{u}_{e,K_{e}}^{n}}}-{\color{red}{\normalcolor H_{e,K}^{n}}}{\color{red}{\normalcolor \mathbf{u}_{e,K}^{n}}}}{2}\right).\mathbf{n}_{e,K}}
\end{array}\right.\,.\label{eq:exp-scheme-2-1}
\end{equation}

For $\mathcal{F}_{e}^{n}$ we use the fully explicit version of the
numerical fluxes, following comments of \S \ref{Robustness} and
Remark \ref{Explicit}. As concerns the corrected potential, $\Phi_{e}^{n,\ast}$,
to make things more concrete, the constant ${C_{{\boldsymbol{\mathcal{H}}}}}$
relying on the $L^{2}-$norm of $\mathcal{H}$ (\ref{Hess}) has been
roughly estimated by $gL/\rho$, $\rho$ standing for the density
of the considered layer. Of course, a more accurate estimate of $\vert\vert\vert{\boldsymbol{\mathcal{H}}}(\boldsymbol{H},\mathbf{x})\vert\vert\vert_{L^{2}}$
can be used, according to Remark \ref{Hess}, but this does not affect
the numerical results . All the vectorial slopes are first computed
and the reconstructed primitive variables ${\normalcolor h_{e,K}^{n}}$,
${\color{red}{\normalcolor \mathbf{u}_{e,K}^{n}}}$ are subsequently
extracted at each edge side. The numerical scheme can afterwards be
supplemented by a Heun scheme for time integration in order to derive
a full second-order scheme is space and time, stable under a classical
CFL number.

\subsubsection{Entropy stability of MUSCL extension\label{subsec:Formal-second-order}}

The following section is intented to give some insights into the general
strategy adopted to extend the energy dissipation to MUSCL schemes.
We consider the case of a regular cartesian mesh for the sake of simplicity,
and note $\Delta x$ the space step (meaning that $\ m_{e}=\Delta x\;\text{\,}\forall e\in\mathbb{F}\text{\,,}$
$\mathbb{F}$ collecting the edges of the mesh). Again for simplicity
reasons, we propose here a formal proof, in which the constants will
be generically denoted $C$. Note that we allow some of these constants
to imply several $L^{\infty}$ norms of the flow variables, which
is ultimately equivalent to suppose the water heights bounded and
far from zero. Following \citep{Vila1989}, and denotig $h$ a characteristic
length of the mesh, we proceed to a complementary restriction on the
reconstructed variables (\ref{eq:muscl_recons_vector}), assuming
\begin{equation}
\left\Vert \nabla\mathbf{W}_{K}^{n}\right\Vert <Ch^{1-r}\label{restriction}
\end{equation}
 with $0<r<1$, and $C>0$, in order to control the slope in the regions
close to discontinuities. Note that such a limitation does not occur
in smooth areas since we expect $\left\Vert \nabla\mathbf{W}_{K}^{n}\right\Vert <C.$
Let $\mathbf{V}=\mathbf{E}_{\mathbf{W}}(\mathbf{W})$ be the set of
entropy variables. Denoting $\mathbf{W}_{K}^{n}$ and $\mathbf{\bar{{W}}}_{K}^{n}$
the solutions of the MUSCL and first-order schemes respectively, and
according to the convexity of $\mathbf{E}_{\mathbf{}}$ , we have
the local estimation:
\[
\mathbf{E}_{K}^{n+1}\leq\mathbf{E}(\mathbf{\bar{{W}}}_{K}^{n+1})+\mathbf{V}_{K}^{n+1}.\left(\mathbf{W}_{K}^{n+1}-\mathbf{\bar{{W}}}_{K}^{n+1}\right)\text{\,}.
\]
Formally, according to (\ref{Strict_decrease}), we can find a constant
$C>0$ such that :

\begin{equation}
\mathbf{E}^{n+1}+C\left(\Delta t\right)^{2}\sum_{K\in\mathbb{T},e\in\partial K}\left(\left\Vert \delta\Phi_{e}^{n}\right\Vert ^{2}+\left\Vert \delta H\mathbf{u}_{e}^{n}\right\Vert ^{2}\right)\leq\mathbf{E}^{n}+\left(\Delta x\right)^{2}\sum_{K\in\mathbb{T}}\mathbf{V}_{K}^{n+1}.\left(\mathbf{W}_{K}^{n+1}-\mathbf{\bar{{W}}}_{K}^{n+1}\right)\text{\,}.\label{eq:second_order_energy}
\end{equation}
We express the difference between the second and first-order solutions
at time $n+1$ as:

\begin{equation}
\mathbf{W}_{K}^{n+1}-\mathbf{\bar{{W}}}_{K}^{n+1}\text{=\ensuremath{\left(\begin{array}{c}
 \frac{\Delta t}{\Delta x}\sum_{e\in\partial K}\delta\mathcal{F}_{e,K}^{n}\\
 \frac{\Delta t}{\Delta x}\sum_{e\in\partial K}\delta\mathcal{G}_{e,K}^{n}+\frac{\Delta t}{\Delta x}H_{K}^{n}\sum_{e\in\partial K}\delta\mathcal{P}_{e,K}^{n} 
\end{array}\right)}\,},\label{Wdelta}
\end{equation}
where
\begin{align*}
\delta\mathcal{F}_{e,K}^{n} & =\mathcal{F}({\mathbf{W}}_{e,K}^{n},\,{\mathbf{W}}_{e,K_{e}}^{n},\,{\mathbf{n}}_{e,K})-\mathcal{F}({\mathbf{W}}_{K}^{n},\,{\mathbf{W}}_{K_{e}}^{n},\,{\mathbf{n}}_{e,K})\\
\delta\mathcal{G}_{e,K}^{n} & =\mathcal{G}({\mathbf{W}}_{e,K}^{n},\,{\mathbf{W}}_{e,K_{e}}^{n},\,{\mathbf{n}}_{e,K})-\mathcal{G}({\mathbf{W}}_{K}^{n},\,{\mathbf{W}}_{K_{e}}^{n},\,{\mathbf{n}}_{e,K})\qquad,\\
\delta\mathcal{P}_{e,K}^{n} & =\mathcal{P}({\mathbf{W}}_{e,K}^{n},\,{\mathbf{W}}_{e,K_{e}}^{n},\,{\mathbf{n}}_{e,K})-\mathcal{P}({\mathbf{W}}_{K}^{n},\,{\mathbf{W}}_{K_{e}}^{n},\,{\mathbf{n}}_{e,K})\,
\end{align*}
using the notations introduced in (\ref{eq:fluxes}). We hence have:
\begin{align}
\left(\Delta x\right)^{2}\sum_{K\in\mathbb{T}}\mathbf{V}_{K}^{n+1}.\left(\mathbf{W}_{K}^{n+1}-\mathbf{\bar{{W}}}_{K}^{n+1}\right) & =\Delta t\Delta x\sum_{e\in\mathbb{F}}\delta\mathcal{F}_{e,K}^{n}\left(\mathbf{V}_{K}^{n+1}-\mathbf{V}_{K_{e}}^{n+1}\right)_{H}\nonumber \\
 & \,+\Delta t\Delta x\sum_{e\in\mathbb{F}}\delta\mathcal{G}_{e,K}^{n}.\left(\mathbf{V}_{K}^{n+1}-\mathbf{V}_{K_{e}}^{n+1}\right)_{H\mathbf{u}}\label{scalarV}\\
 & \,+\Delta t\Delta x\sum_{e\in\mathbb{F}}\delta\mathcal{P}_{e,K}^{n}.\left(\mathbf{V}_{K}^{n+1}H_{K}^{n}-\mathbf{V}_{K_{e}}^{n+1}H_{K_{e}}^{n}\right)_{H\mathbf{u}}.\nonumber 
\end{align}
We then write:
\begin{align}
\left\Vert \mathbf{V}_{K}^{n+1}-\mathbf{V}_{K_{e}}^{n+1}\right\Vert  & =\left\Vert \mathbf{V}\left(\mathbf{\bar{{W}}}_{K}^{n+1}+\mathbf{W}_{K}^{n+1}-\mathbf{\bar{{W}}}_{K}^{n+1}\right)-\mathbf{V}\left(\mathbf{\bar{{W}}}_{K_{e}}^{n+1}+\mathbf{W}_{K_{e}}^{n+1}-\mathbf{\bar{{W}}}_{K_{e}}^{n+1}\right)\right\Vert \nonumber \\
 & =\left\Vert \mathbf{V}\left(\mathbf{W}_{K}^{n}-\mathcal{A}_{K}^{n}+\mathbf{W}_{K}^{n+1}-\mathbf{\bar{{W}}}_{K}^{n+1}\right)-\mathbf{V}\left(\mathbf{\bar{{W}}}_{K_{e}}^{n}-\mathcal{A}_{K_{e}}^{n}+\mathbf{W}_{K_{e}}^{n+1}-\mathbf{\bar{{W}}}_{K_{e}}^{n+1}\right)\right\Vert \quad,\label{Vkn}
\end{align}
where the terms $\mathcal{A}_{K}^{n}=\mathbf{\bar{{W}}}_{K}^{n+1}-\mathbf{W}_{K}^{n}$
are given by the first-order scheme (\ref{mass}, \ref{mom}). Considering
that each quantity $\delta\mathcal{F}_{e,K}^{n}$ , $\delta\mathcal{G}_{e,K}^{n}$
and $\delta\mathcal{P}_{e,K}^{n}$ appearing in (\ref{Wdelta}) can
be expressed, by construction, in terms of components of $\delta\mathbf{W}_{e,K}^{n}=\mathbf{W}_{e,K}^{n}-\mathbf{W}_{K}^{n}$,
the limitation (\ref{restriction}) gives, using (\ref{eq:muscl_recons_vector}):
\begin{align*}
\max\left(\left\Vert \delta\mathcal{F}_{e,K}^{n}\right\Vert ,\left\Vert \delta\mathcal{G}_{e,K}^{n}\right\Vert ,\left\Vert \delta\mathcal{H}_{e,K}^{n}\right\Vert \right)\leq Ch^{r}\,,
\end{align*}
and therefore

\begin{align*}
\left\Vert \mathbf{W}_{K}^{n+1}-\mathbf{\bar{{W}}}_{K}^{n+1}\right\Vert \leq C\frac{\Delta t}{\Delta x}h^{r}\,.
\end{align*}
Reformulating the first-order scheme (\ref{mass}, \ref{mom}), one
can establish that a similar estimation stands for the terms $\mathcal{A}_{K}^{n}$.
By continuity arguments in (\ref{Vkn}), this finally gives:
\begin{align*}
\left\Vert \mathbf{V}_{K}^{n+1}-\mathbf{V}_{K_{e}}^{n+1}\right\Vert \leq C\left(\left\Vert \mathbf{W}_{K}^{n}-\mathbf{W}_{K_{e}}^{n}\right\Vert +C\frac{\Delta t}{\Delta x}h^{r}\right)\,.
\end{align*}
Using this estimation to control the terms appearing in the right
hand side of (\ref{scalarV}), going back to (\ref{eq:second_order_energy})
we finally get:

\begin{equation}
\mathbf{E}^{n+1}+C\left(\Delta t\right)^{2}\sum_{K,e}\left(\left\Vert \delta\Phi_{e}^{n}\right\Vert ^{2}+\left\Vert \delta H\mathbf{u}_{e}^{n}\right\Vert ^{2}\right)\leq\mathbf{E}^{n}+\Delta t\Delta x\sum_{e\in\mathbb{F}}Ch^{r}\left(\left\Vert \mathbf{W}_{K}^{n}-\mathbf{W}_{K_{e}}^{n}\right\Vert +C\frac{\Delta t}{\Delta x}h^{r}\right)\text{\,}.\label{eq:second_order_energy-1}
\end{equation}
Noting that we have an estimation of the form

\begin{align*}
\left\Vert \mathbf{W}_{K}^{n}-\mathbf{W}_{K_{e}}^{n}\right\Vert ^{2}\leq C\left(\left\Vert \delta\Phi_{e}^{n}\right\Vert ^{2}+\left\Vert \delta H\mathbf{u}_{e}^{n}\right\Vert ^{2}\right)\,,
\end{align*}
we write:
\[
\sum_{e\in\mathbb{F}}\Delta t\Delta x\left\Vert \mathbf{W}_{K}^{n}-\mathbf{W}_{K_{e}}^{n}\right\Vert \leq C\left(\sum_{e\in\mathbb{F}}\Delta t\left(\Delta x\right)^{2}\right)^{\frac{1}{2}}\left(\sum_{e\in\mathbb{F}}\Delta t\left(\left\Vert \delta\Phi_{e}^{n}\right\Vert ^{2}+\left\Vert \delta H\mathbf{u}_{e}^{n}\right\Vert ^{2}\right)\right)^{\frac{1}{2}}\,.
\]
Setting $M^{2}=\Delta t\sum_{K,e}\left(\left\Vert \delta\Phi_{e}^{n}\right\Vert ^{2}+\left\Vert \delta H\mathbf{u}_{e}^{n}\right\Vert ^{2}\right)$,
(\ref{eq:second_order_energy-1}) gives:

\begin{equation}
\mathbf{E}^{n+1}+C\Delta tM^{2}\leq\mathbf{E}^{n}+C\left(\frac{\Delta t}{\Delta x}\right)^{2}h^{2r}+C\sqrt{\Delta t}Mh^{r}\text{\,}.\label{eq:second_order_energy-2}
\end{equation}
With $\Delta t=\Delta x=h$:
\begin{equation}
\frac{\mathbf{E}^{n+1}-\mathbf{E}^{n}}{\Delta t}\leq Ch^{2r-1}+CMh^{r-1/2}-CM^{2}\text{\,}.\label{eq:second_order_energy-2-1}
\end{equation}
A trivial analysis of the quadratic polynomial in $M$ of the right
hand side leads to a condition of the form $\alpha(h)\leq C$, where
$\alpha(h)$ is $\underset{}{\mathcal{O}}(h^{2r-1},h^{r-1/2})$, leading
to the condition $r>1/2$.

\subsection{Time stepping for Coriolis force\label{subsec:time-stepping-coriolis}}

It has been demonstrated that under inequalities conditions on $\gamma$
and $\alpha$, the first-order scheme given by (\ref{mass}- \ref{mom}-\ref{phiea}-\ref{fe})
dissipates mechanical energy. This property has also been highlighted
for the second-order scheme (\ref{eq:exp-scheme-1-1}) and (\ref{eq:exp-scheme-2-1}),
at least numerically, in \S \ref{subsec:linear-waves}. The proposed
approach to incorporate the Coriolis force is designed to preserve
at best these stability properties. From this perspective, a time
stepping scheme is considered to integrate the following ordinary
differential equations~:

\begin{equation}
{\displaystyle \frac{\partial}{\partial t}\left(\begin{array}{c}
u\\
v
\end{array}\right)=f\left(\begin{array}{cc}
0 & 1\\
-1 & 0
\end{array}\right)}\left(\begin{array}{c}
u\\
v
\end{array}\right)\,.\label{eq:ode-coriolis}
\end{equation}

Among the desired stability properties, one asks the numerical approach
to be a symplectic integrator and to preserve kinetic energy, i.e.
$\left\Vert \mathbf{u}\right\Vert ^{n+1}=\left\Vert \mathbf{u}\right\Vert ^{n}$.
A first way to proceed is to consider the exact integration of the
previous ordinary differential equations (\ref{eq:ode-coriolis}),
resulting to the scheme:

\begin{equation}
\left\{ \begin{array}{cclcc}
u^{n+1} & = & \cos\left(f\Delta t^{n}\:u^{n}\right) & + & \sin\left(f\Delta t^{n}\:v^{n}\right)\\
\\
v^{n+1} & = & \cos\left(f\Delta t^{n}\:v^{n}\right) & - & \sin\left(f\Delta t^{n}\:u^{n}\right)
\end{array}\right.\,.\label{eq:scheme-cor-rot}
\end{equation}

Another way is to consider the Crank-Nicolson scheme:

\begin{equation}
\left\{ \begin{array}{cclc}
u^{n+1} & = &  & {\displaystyle \frac{f\Delta t^{n}}{2}\left(v^{n}+v^{n+1}\right)}\\
\\
v^{n+1} & = & - & {\displaystyle \frac{f\Delta t^{n}}{2}\left(u^{n}+u^{n+1}\right)}
\end{array}\right.\,.\label{eq:scheme-cor-CN}
\end{equation}

It has been found by numerical experience that the last scheme (\ref{eq:scheme-cor-CN})
with an IMEX time steeping scheme H-CN(2,2,2) defined below in Tab.\ref{tab:HCN222}
by his Butcher tableau is globally dissipative for long time simulations.

\begin{table}[H]
\centering{}%
\begin{tabular}{cc}
\begin{tabular}{c|cc}
$0$  & $0$  & $0$\tabularnewline
$1$  & 1  & $0$\tabularnewline
\hline 
 & $1/2$  & $1/2$\tabularnewline
\end{tabular}\quad{}  & \quad{}%
\begin{tabular}{c|cc}
$0$  & $0$  & $0$\tabularnewline
$1$  & 1/2  & $1/2$\tabularnewline
\hline 
 & $1/2$  & $1/2$\tabularnewline
\end{tabular}\tabularnewline
\end{tabular}\caption{Second-order IMEX scheme H-CN(2,2,2) with an explicit Heun scheme
for the model without Coriolis force and a Crank-Nicolson scheme for
the Coriolis force.\label{tab:HCN222}}
\end{table}
The above IMEX time steeping can be written for numerical implementation
purpose as follows:

\begin{equation}
\begin{array}{lcl}
\mathbf{U}_{K}^{(1)} & = & \mathbf{U}_{K}^{n}\:+\:\Delta t^{n}\:\mathcal{L}(\mathbf{U}_{K}^{(1)})\\
\mathbf{U}_{K}^{(2)} & = & \mathbf{U}_{K}^{(1)}\:+\:{\displaystyle \frac{\Delta t^{n}}{2}}\mathcal{C}(\mathbf{U}_{K}^{n})\:+\:{\displaystyle \frac{\Delta t^{n}}{2}}\mathcal{C}(\mathbf{U}_{K}^{(2)})\\
\mathbf{U}_{K}^{(3)} & = & \mathbf{U}_{K}^{(2)}\:+\:\Delta t^{n}\:\mathcal{L}(\mathbf{U}_{K}^{(2)})\\
\mathbf{U}_{K}^{n+1} & = & {\displaystyle \frac{1}{2}}\left(\mathbf{U}_{K}^{n}-\mathbf{U}_{K}^{(1)}+\mathbf{U}_{K}^{(2)}+\mathbf{U}_{K}^{(3)}\right)
\end{array}\,,\label{eq:HCN222}
\end{equation}
where $\mathcal{L}$ is the numerical space integration of the homogeneous
model (corresponding to Eqs.\ref{eq:exp-scheme-1-1}-\ref{eq:exp-scheme-2-1})
and $\mathcal{C}$ is the operator corresponding to the Coriolis force:

\begin{equation}
\mathcal{C}(\mathbf{U}_{i})=\begin{bmatrix}0\\
\phantom{-}f\:h_{i}u_{i}\\
-f\:h_{i}v_{i}
\end{bmatrix}\,.
\end{equation}

As it can be observed in Fig.\ref{fig:vortex-conv} for the long time
simulations of the baroclinic vortex, the mechanical energy is effectively
dissipated using this time stepping scheme. These energy losses gradually
become less important as the mesh resolution increases.

\subsection{Time step\label{subsec:time-step}}

Based on (\ref{CFL_1}), the numerical CFL-like condition for the
time step $\Delta t^{n}$ for all the two-dimensional simulations
presented in this article is:

\begin{equation}
\Delta t^{n}=\tau_{CFL}\;\min_{K\in\Omega}{\displaystyle \left(\frac{2\;m_{K}}{m_{\partial K}\left(\left\Vert \bar{\mathbf{u}}_{K}^{n}\right\Vert +\sqrt{g\bar{h}_{K}^{n}}\right)}\right)}\,,\label{eq:time-step}
\end{equation}
where $\tau_{CFL}$ is the CFL number, $\bar{h}_{K}^{n}$ is the total
water depth and $\left\Vert \bar{\mathbf{u}}_{K}^{n}\right\Vert $
is the mean velocity, computed from:

\begin{equation}
\left\{ \begin{array}{l}
\bar{h}_{K}^{n}={\displaystyle \sum_{i=1}^{L}h_{K,i}^{n}}\\
\\
\left\Vert \bar{\mathbf{u}}_{K}^{n}\right\Vert ={\displaystyle \frac{1}{\bar{h}_{K}^{n}}}\sqrt{\left({\displaystyle \sum_{i=1}^{L}h_{K,i}^{n}}u_{K,i}^{n}\right)^{2}+\left({\displaystyle \sum_{i=1}^{L}h_{K,i}^{n}}v_{K,i}^{n}\right)^{2}}
\end{array}\right.\,.
\end{equation}
The time step is thus calibrated on the barotropic gravity wave.


\bibliographystyle{plain}
\phantomsection\addcontentsline{toc}{section}{\refname}\bibliography{biblio_from_zotero}

\begin{thebibliography}{10}

\bibitem{COMODO}
{{COMODO}} benchmark.
\newblock \url{http://indi.imag.fr/wordpress/}.

\bibitem{FVCOM}
{{FVCOM}}: {{The Unstructured Grid Finite Volume Community Ocean Model}}.
\newblock \url{http://fvcom.smast.umassd.edu/fvcom/}.

\bibitem{SLIM}
{{SLIM}}: {{Second}}-generation {{Louvain}}-la-{{Neuve Ice}}-ocean {{Model}}.
\newblock \url{http://sites.uclouvain.be/slim/}.

\bibitem{Abgrall2009}
R{\'e}mi Abgrall and Smadar Karni.
\newblock Two-layer shallow water system: a relaxation approach.
\newblock {\em SIAM Journal on Scientific Computing}, 31(3):1603 -- 1627, 2009.

\bibitem{Audusse2011}
E.~Audusse, M.-O. Bristeau, M.~Pelanti, and J.~Sainte-Marie.
\newblock Approximation of the hydrostatic {{Navier}}--{{Stokes}} system for
  density stratified flows by a multilayer model: {{Kinetic}} interpretation
  and numerical solution.
\newblock {\em Journal of Computational Physics}, 230(9):3453 -- 3478, 2011.

\bibitem{Audusse2014}
Emmanuel Audusse, Fayssal Benkhaldoun, Saida Sari, Mohammed Seaid, and Pablo
  Tassi.
\newblock A fast finite volume solver for multi-layered shallow water flows
  with mass exchange.
\newblock {\em Journal of Computational Physics}, 272:23--45, 2014.

\bibitem{Barth2004}
Timothy Barth and Mario Ohlberger.
\newblock {\em Finite Volume Methods: Foundation and Analysis}.
\newblock John Wiley \& Sons, Ltd, 2004.

\bibitem{Beljadid2013}
Abdelaziz Beljadid, Abdolmajid Mohammadian, and Hazim~M. Qiblawey.
\newblock An unstructured finite volume method for large-scale shallow flows
  using the fourth-order {{Adams}} scheme.
\newblock {\em Computers \& Fluids}, 88:579 -- 589, 2013.

\bibitem{Berthon2006}
C.~Berthon.
\newblock Robustness of muscl schemes for 2d unstructured meshes.
\newblock {\em J. Comp. Phys.}, 218(2):495 -- 509, 2006.

\bibitem{Berthon2015}
C.~Berthon, F.~Foucher, and T.~Morales.
\newblock An efficient splitting technique for two layer shallow water model.
\newblock {\em Numerical Methods for Partial Differential Equations},
  31(5):1396 -- 1423, 2015.

\bibitem{Bleck2002}
Rainer Bleck.
\newblock An oceanic general circulation model framed in hybrid
  isopycnic-{{Cartesian}} coordinates.
\newblock {\em Ocean Modelling}, 4(1):55 -- 88, 2002.

\bibitem{Bollermann2013}
Andreas Bollermann, Guoxiana Chen, Alexander Kurganov, and Sebastian Noelle.
\newblock A well-balanced reconstruction of wet/dry fronts for the shallow
  water equations.
\newblock {\em J. Sci. Comput.}, 56(2):267 -- 290, 2013.

\bibitem{Bollermann2011}
Andreas Bollermann, Sebastian Noelle, and M~Lukacova-Medvidova.
\newblock Finite volume evolution galerkin methods for the shallow water
  equations with dry beds.
\newblock {\em Comm. Comput. Phys.}, 10:371 -- 404, 2011.

\bibitem{Bouchut2008}
Fran\c{}cois Bouchut and Tom{\'a}s Morales.
\newblock An entropy satisfying scheme for two-layer shallow water equations
  with uncoupled treatment.
\newblock {\em ESAIM: Mathematical Modelling and Numerical Analysis}, 42(4):683
  -- 698, 2008.

\bibitem{Bouchut2010}
Fran\c{}cois Bouchut and Vladimir Zeitlin.
\newblock A robust well-balanced scheme for multi-layer shallow water
  equations.
\newblock {\em Discrete and Continuous Dynamical Systems-Series B}, 13(4):739
  -- 758, 2010.

\bibitem{Bresch2011}
Didier Bresch, Rupert Klein, and Carine Lucas.
\newblock Multiscale analyses for the {{Shallow Water}} equations.
\newblock In {\em Computational {{Science}} and {{High Performance Computing
  IV}}}, volume 115 of {\em Notes on Numerical Fluid Mechanics and
  Multidisciplinary Design}, pages 149 -- 164. 2011.

\bibitem{Burguete2008}
J~Burguete, P~Garcia-Navarro, and J~Murillo.
\newblock Friction term discretization and limitation to preserve stability and
  conservation in the 1d shallow-water model: Application to unsteady
  irrigation and river flow.
\newblock {\em Int J Numer Methods Fluids}, 58:403 -- 425, 2008.

\bibitem{Diaz2014}
Manuel Castro, Yuanzhen Cheng, Alina Chertock, and Alexander Kurganov.
\newblock Solving two-mode shallow water equations using finite volume methods.
\newblock {\em Communications in Computational Physics}, 16(5):1323 -- 1354,
  2014.

\bibitem{Castro2001}
Manuel Castro, Jorge Mac{\'\i}as, and Carlos Par{\'e}s.
\newblock A {{Q}}-scheme for a class of systems of coupled conservation laws
  with source term. {{Application}} to a two-layer 1-{{D}} shallow water
  system.
\newblock {\em ESAIM: Mathematical Modelling and Numerical Analysis},
  35(01):107 -- 127, 2001.

\bibitem{Cea2012}
L.~Cea and M.~E. V{\'a}zquez-Cend{\'o}n.
\newblock Unstructured finite volume discretization of bed friction and
  convective flux in solute transport models linked to the shallow water
  equations.
\newblock {\em Journal of Computational Physics}, 231:3317 -- 3339, 2012.

\bibitem{Chertok2013}
Alina Chertock, Alexander Kurganov, Zhuolin Qu, and Tong Wu.
\newblock Three-{{Layer Approximation}} of {{Two-Layer Shallow Water
  Equations}}.
\newblock {\em Mathematical Modelling and Analysis}, 18:675 -- 693, 2013.

\bibitem{Cotter2014}
C.~J. Cotter and J.~Thuburn.
\newblock A finite element exterior calculus framework for the rotating
  shallow-water equations.
\newblock {\em Journal of Computational Physics}, 257, Part B:1506 -- 1526,
  2014.
\newblock Physics-compatible numerical methods.

\bibitem{Danilov2013}
S.~Danilov.
\newblock Ocean modelling on unstructured meshes.
\newblock {\em Ocean Modelling}, 69:195 -- 210, 2013.

\bibitem{Dellacherie2010}
S.~Dellacherie.
\newblock Analysis of {{Godunov}} type schemes applied to the compressible
  {{Euler}} system at low {{Mach}} number.
\newblock {\em Journal of Computational Physics}, pages 978 -- 1016, 2010.

\bibitem{Duchene2016}
Vincent Duch{\^e}ne.
\newblock The multilayer shallow water system in the~limit of small density
  contrast.
\newblock {\em Asymptotic Analysis}, 98(3):189 -- 235, 2016.

\bibitem{Duran2015}
A.~Duran.
\newblock A robust and {{Well Balanced}} scheme for the {{2D Saint-Venant}}
  system on unstructured meshes with friction source term.
\newblock {\em International Journal for Numerical Methods in Fluids}, pages
  89--121, 2015.

\bibitem{Eymard2000}
R.~Eymard, T.~Gallou{\"e}t, and R.~Herbin.
\newblock Finite volume methods.
\newblock {\em Handbook of Numerical Analysis}, 7:713--1018, 2000.

\bibitem{Gassmann2012}
A.~Gassmann.
\newblock A global hexagonal {{C}}-grid non-hydrostatic dynamical core
  ({{ICON-IAP}}) designed for energetic consistency.
\newblock {\em Quarterly Journal of the Royal Meteorological Society}, 139:152
  -- 175, 2012.

\bibitem{Godlewski1996}
E.~Godlewski and P.-A. Raviart.
\newblock {\em Numerical approximation of hyperbolic systems of conservation
  laws}, volume~18.
\newblock 1996.

\bibitem{Grenier2013}
N.~Grenier, J.-P. Vila, and {P.Villedieu}.
\newblock An accurate low-{{Mach}} scheme for a compressible two-fluid model
  applied to free-surface flows.
\newblock {\em Journal of Computational Physics}, 252:1--19, 2013.

\bibitem{Hou2013}
J.~Hou, {F.Simons}, M.~Mahgoub, and R.~Hinkelmann.
\newblock A robust well-balanced model on unstructured grids for shallow water
  flows with wetting and drying over complex topography.
\newblock {\em Computer Methods in Applied Mechanics and Engineering}, 257:126
  -- 149, 2013.

\bibitem{Kurganov2009}
A.~Kurganov and G.~Petrova.
\newblock Central-{{Upwind Schemes}} for {{Two-Layer Shallow Water Equations}}.
\newblock {\em SIAM Journal on Scientific Computing}, 31(3):1742 -- 1773, 2009.

\bibitem{Lemarie2015}
F.~Lemari{\'e}, L.~Debreu, G.~Madec, J.~Demange, J.~M. Molines, and
  M.~Honnorat.
\newblock Stability constraints for oceanic numerical models: implications for
  the formulation of time and space discretizations.
\newblock {\em Ocean Modelling}, 92:124 -- 148, 2015.

\bibitem{Leveque1998}
Randall~J. LeVeque.
\newblock Balancing source terms and flux gradients in high-resolution godunov
  methods: The quasi-steady wave-propagation algorithm.
\newblock {\em Journal of Computational Physics}, 146(1):346 -- 365, 1998.

\bibitem{Liou2006}
Meng-Sing Liou.
\newblock A sequel to ausm, part ii: Ausm+-up for all speeds.
\newblock {\em Journal of Computational Physics}, 214:137 -- 170, 2006.

\bibitem{Liou1993}
Meng-Sing Liou and Christopher~J. Steffen.
\newblock A new flux splitting scheme.
\newblock {\em Journal of Computational Physics}, 107:23 -- 39, 1993.

\bibitem{Madec2008}
G.~Madec and {and the NEMO team}.
\newblock {{NEMO}} ocean engine.
\newblock {\em Note du P{\^o}le de mod{\'e}lisation, Institut Pierre-Simon
  Laplace (IPSL), France, No 27, ISSN, No 1288-1619 (2008)}, 2008.

\bibitem{Mandli2013}
Kyle~T Mandli.
\newblock A numerical method for the two layer shallow water equations with dry
  states.
\newblock {\em Ocean Modelling}, 72:80--91, 2013.

\bibitem{Meister2016}
A.~Meister and S.~Ortleb.
\newblock A positivity preserving and well-balanced {{DG}} scheme using finite
  volume subcells in almost dry regions.
\newblock {\em Applied Mathematics and Computation}, 272:259 -- 273, 2016.

\bibitem{Monjarret_PHD}
R.~Monjarret.
\newblock {\em The multi-layer shallow water model with free surface.
  {{Numerical}} treatment of the open boundaries}.
\newblock PhD thesis, Institut National Polytechnique de Toulouse,
  Universit{\'e} de Toulouse, 2014.

\bibitem{Murillo2012}
J.~Murillo and P.~Garcia-Navarro.
\newblock Augmented versions of the {{HLL}} and {{HLLC Riemann}} solvers
  including source terms in one and two dimensions for shallow flow
  applications.
\newblock {\em Journal of Computational Physics}, 231:6861 -- 6906, 2012.

\bibitem{Nikolos2009}
I.~K. Nikolos and A.~I. Delis.
\newblock An unstructured node-centered finite volume scheme for shallow water
  flows with wet/dry fronts over complex topography.
\newblock {\em Computer Methods in Applied Mechanics and Engineering}, 198:3723
  -- 3750, 2009.

\bibitem{Noelle2006}
Sebastian Noelle, Normann Pankratz, Gabriella Puppo, and Jostein~R. Natvig.
\newblock Well-balanced finite volume schemes of arbitrary order of accuracy
  for shallow water flows.
\newblock {\em Journal of Computational Physics}, 213(2):474 -- 499, 2006.

\bibitem{Parisot2015}
Martin Parisot and Jean-Paul Vila.
\newblock Centered-potential regularization for the advection upstream
  splitting method.
\newblock {\em SIAM Journal on Numerical Analysis}, 54(5):3083--3104, 2016.

\bibitem{Penven2006}
Pierrick Penven, L.~Debreu, Patrick Marchesiello, and J.~C. McWilliams.
\newblock Evaluation and application of the {{ROMS}} 1-way embedding procedure
  to the central california upwelling system.
\newblock {\em Ocean Modelling}, 12:157 -- 187, 2006.

\bibitem{Ringler2010}
T.~D. Ringler, J.~Thuburn, J.~B. Klemp, and W.~C. Skamarock.
\newblock A unified approach to energy conservation and potential vorticity
  dynamics for arbitrarily-structured {{C}}-grids.
\newblock {\em Journal of Computational Physics}, 229(9):3065 -- 3090, 2010.

\bibitem{LeRoux2012}
D.~Le Roux.
\newblock Spurious inertial oscillations in shallow water models.
\newblock {\em Journal of Computational Physics}, 231:7959 -- 7987, 2012.

\bibitem{Ricchiuto2013}
D.~S\'arm\'any, M.E. Hubbard, and M.~Ricchiuto.
\newblock Unconditionally stable space-time discontinuous residual distribution
  for shallow-water flows.
\newblock {\em Journal of Computational Physics}, 253:86 -- 113, 2013.

\bibitem{Shchepetkin2005}
A.~F. Shchepetkin and J.~C. McWilliams.
\newblock The regional oceanic modeling system (roms): a split-explicit,
  free-surface, topography-following-coordinate oceanic model.
\newblock {\em Ocean Modelling}, 9:347 -- 404, 2005.

\bibitem{Stewart2016}
Andrew~L. Stewart and Paul~J. Dellar.
\newblock An energy and potential enstrophy conserving numerical scheme for the
  multi-layer shallow water equations with complete {{Coriolis}} force.
\newblock {\em Journal of Computational Physics}, 313:99 -- 120, 2016.

\bibitem{Strauss1992}
W.~A. Strauss.
\newblock {\em Partial {{Differential Equations}}~: {{An Introduction}}}.
\newblock {John Wiley}, 1992.

\bibitem{Szmelter2010}
J.~Szmelter and P.~Smolarkiewicz.
\newblock An edge-based unstructured mesh discretization in geospherical
  framework.
\newblock {\em Journal of Computational Physics}, 229:4980 -- 4995, 2010.

\bibitem{Tavelli2014}
Maurizio Tavelli and Michael Dumbser.
\newblock A high order semi-implicit discontinuous galerkin method for the two
  dimensional shallow water equations on staggered unstructured meshes.
\newblock {\em Applied Mathematics and Computation}, 234:623 -- 644, 2014.

\bibitem{Thuburn2009}
J.~Thuburn, T.~Ringler, J.~Klemp, and W.~Skamarock.
\newblock Numerical representation of geostrophic modes on arbitrarily
  structured {{C}}-grids.
\newblock {\em Journal of Computational Physics}, 228:8321 -- 8335, 2009.

\bibitem{Toro2001}
E.F. Toro.
\newblock {\em Shock-capturing methods for free-surface shallow flows}.
\newblock John Wiley, 2001.

\bibitem{Vallis2006}
G.~K. Vallis.
\newblock {\em Atmospheric and {{Oceanic Fluid Dynamics}}}.
\newblock {Cambridge University Press}, Cambridge, U.K., 2006.

\bibitem{Vila1986}
J.-P. Vila.
\newblock Simplified godunov schemes for 2 x 2 systems of conservation laws.
\newblock {\em SIAM J. Numer. Anal.}, 23(6):1173--1192, December 1986.

\bibitem{Vila1989}
J.-P. Vila.
\newblock An analysis of a class of second-order accurate godunov-type schemes.
\newblock {\em SIAM J. Numer. Anal.}, 26(4):830--853, 1989.

\bibitem{Vila2003}
J.-P. Vila and P.~Villedieu.
\newblock Convergence of an explicit finite volume scheme for first order
  symmetric systems.
\newblock {\em Numerische Mathematik}, 94:573 -- 602, 2003.

\bibitem{Xing2013}
Yulong Xing and Xiangxiong Zhang.
\newblock Positivity-{{Preserving Well-Balanced Discontinuous Galerkin
  Methods}} for the {{Shallow Water Equations}} on {{Unstructured Tria}} ngular
  {{Meshes}}.
\newblock {\em Journal of Scientific Computing}, 57(1):19--41, 2013.

\end{thebibliography}

\end{document}